\documentclass[twoside,11pt]{amsart}

\usepackage[margin=1in]{geometry}

\usepackage{tcolorbox}
\usepackage{amssymb,color}
\usepackage[dvipsnames]{xcolor}
\usepackage{amsmath,amsfonts,amssymb,listings,bbm,amsthm,enumitem}
\usepackage[ruled,vlined]{algorithm2e}
\usepackage[colorlinks=false,allbordercolors={1 1 1}]{hyperref}
\usepackage[noabbrev,capitalize]{cleveref}
\usepackage{mathrsfs,graphicx}
\usepackage{caption,subcaption}

\DeclareMathOperator*{\argmin}{\arg\!\min}

\DeclareMathOperator{\E}{\mathbb{E}}
\DeclareMathOperator{\R}{\mathbb{R}}

\DeclareMathOperator{\tr}{tr}

\DeclareMathOperator{\inte}{int}

\DeclareMathOperator{\cov}{cov}
\DeclareMathOperator{\var}{var}
\DeclareMathOperator{\spt}{supp}
\DeclareMathOperator{\dom}{\mathrm{dom}}

\theoremstyle{plain}
\newtheorem{lemma}{Lemma}[section]
\newtheorem{theorem}[lemma]{Theorem}
\newtheorem{corollary}[lemma]{Corollary}
\newtheorem{proposition}[lemma]{Proposition}
\newtheorem{remark}[lemma]{Remark}

\theoremstyle{definition}

\usepackage{lastpage}

\newcommand\blfootnote[1]{%
  \begingroup
  \renewcommand\thefootnote{}\footnote{#1}%
  \addtocounter{footnote}{-1}%
  \endgroup
}
\usepackage[symbol]{footmisc}

\begin{document}

\title[Computation and Statistics for the Empirical MEM]{On the Computational and Statistical Efficiency of the Empirical Maximum Entropy on the Mean Method}

\author[M. King-Roskamp]{Matthew King-Roskamp\textsuperscript{\dag} }
\address[M. King-Roskamp]{
 Department of Mathematics and Statistics,
       McGill University}
\email{matthew.king-roskamp@mail.mcgill.ca}
\author[G. Rioux]{Gabriel Rioux\textsuperscript{\dag} }
\address[G. Rioux]{
Department of Mathematics, Imperial College London}
\email{g.rioux@ic.ac.uk}
\author[R. Choksi]{Rustum Choksi}
\address[R. Choksi]{
Department of Mathematics and Statistics,
       McGill University}
\email{ rustum.choksi@mcgill.ca} 
\author[T. Hoheisel]{Tim Hoheisel }
\address[T. Hoheisel]{
Department of Mathematics and Statistics,
       McGill University}
\email{ tim.hoheisel@mcgill.ca} 
\thanks{
 RC and TH acknowledge the support of NSERC through its Discovery grants program.
}
\maketitle

\begin{abstract}
The {\em Maximum Entropy on the Mean (MEM)} method provides a flexible  computational framework for solving inverse problems by combining data fidelity with entropy-based regularization. In practice, however, the prior distribution is typically unknown but can  be estimated from data, giving rise to the empirical MEM method. We establish a  parametric convergence rate of $O(n^{-1/2})$ in expectation for empirical MEM, improving upon the previously established $O(n^{-1/4})$ guarantee by \cite{mkr2024maximumentropy}. Our proof is based on a novel stability analysis of the primal and dual optimization problems under perturbations of the underlying probability measure, relying only on foundational tools from convex analysis and probability. We further show that the  MEM dual problem admits a reformulation as an expected risk minimization problem, thereby placing MEM within the modern framework of stochastic optimization and enabling scalable stochastic gradient algorithms for large-scale inverse problems. Together, these results place empirical MEM as a statistically and computationally efficient  methodology for data-driven inverse problems.
\end{abstract}

\medskip

\section{Introduction}\label{sec:Intro}
\blfootnote{\textsuperscript{\dag} Denotes equal contribution to this work.}
Linear inverse problems of the form
\[
Ax \approx b
\]
are fundamental in statistics, signal processing, machine learning (in particular imaging). Such problems are frequently ill-conditioned or underdetermined, and the observation may additionally be corrupted by noise. Stable recovery therefore requires regularization, often via the incorporation of prior information about the unknown signal.

The {\em Maximum Entropy on the Mean (MEM) method} provides a convex and information-theoretic framework for incorporating such prior information; see, e.g., \cite{le1999new,rioux2021blind,vaisbourd2022maximum}. Given a prior probability measure $\mu \in \mathcal P(\mathcal X)$ supported on a compact set $\mathcal X \subset \R^ {m}$, MEM associates with $\mu$ the (convex) function
\[
    \kappa_{\mu}(x)
    :=
    \inf\left\{
        \mathrm{KL}(Q \Vert \mu)
        :
        Q \in \mathcal P(\mathcal X),\
        \E_{Q}[X]=x
    \right\},
\]
which we call the {\em MEM function}, see, e.g., \cite{vaisbourd2022maximum} for an extensive study of this object, and \Cref{sec:MEMPrelim} for more details.
The corresponding reconstruction  is obtained by solving
\begin{equation}
    x_{\mu}
    \in
    \argmin_{x \in \R^{m}}
    \left\{
        \alpha g(Ax-b)+\kappa_{\mu}(x)
    \right\},\tag{$P_\mu$}
    \label{eqn:introMEMPrimal}
\end{equation}
where $g$ is a proper,  lower semicontinuous, and  convex data-fidelity function and $\alpha>0$ balances data fidelity against prior information/regularization. This is the {\em MEM method} in a nutshell, see \cref{sec:Prob,sec:MEMPrelim} for more details. 
 
Emerging from ideas of E.T. Jaynes in 1957
\cite{jaynes1957information1,jaynes1957information2}, 
various forms and interpretations of MEM (see 
\cite{brown1986fundamentals,gamboa1989methode,marechal1997unification,dacunha1990maximum, le1999new}) 
have appeared in the literature.  Applications  have occurred in different disciplines such as earth sciences \cite{fermin2006bayesian,rietsch1976maximum,urban1996retrieval}, crystallography \cite{Navaza:a24265,navaza1986use}, and medical imaging \cite{amblard2004biomagnetic,cai2022diffuse,chowdhury2013meg,grova2006evaluation,heers2016localization}.
Recently, the MEM method has been shown to be a powerful tool for blind deblurring of images that possess some form of symbology (for example, {UPC and} QR barcodes)  \cite{rioux2021blind,rioux2020maximum}. 
However, MEM methods are not widely used and have yet to become a modern tool for solving contemporary data-driven inverse problems in image processing and machine learning.

The key to approaching the MEM problem theoretically and computationally is Fenchel-Rockafellar duality \cite{rockafellar1970convex}. The (Fenchel) dual of the (primal) MEM problem \eqref{eqn:introMEMPrimal} reads: 
\begin{equation}\label{eqn:introMEMDual}
    z_{\mu} \in \argmin_{z \in \R^{d}} \Big\{\alpha g^{*}(-z/\alpha)-\langle b,z\rangle + L_{\mu}(A^{\top}z)\Big\}, \tag{$D_\mu$}
\end{equation}
where $
L_{\mu}(y):=  \log \int_{\mathcal{X}} \exp\langle y, \cdot \rangle d\mu$
is the log-moment generating function of the prior $\mu$. The advantages of appealing to the dual problem \eqref{eqn:introMEMDual} instead of the primal MEM problem \eqref{eqn:introMEMPrimal} are as follows: 
 The log-moment generating function has closed form for many priors $\mu$ and is essentially smooth (in the sense of \cite{rockafellar1970convex}) thus enabling expedient primal-dual recovery
\begin{equation}
    x_{\mu} = \nabla L_{\mu}(A^{\top}z_{\mu}). \label{eqn:introprimalDualRecoveryFormula}
\end{equation} 
In many applications, however, the prior distribution $\mu$ is unknown. What is available instead is a collection of independent training samples from $\mu$
\[
    X_{1},\ldots,X_{n} \stackrel{\mathrm{i.i.d.}}{\sim} \mu.
\]
A natural data-driven construction replaces $\mu$ by an empirical measure 
\begin{equation}
    \hat \mu_{n}
    :=
    \frac{1}{n}\sum_{i=1}^{n}\delta_{X_{i}}.
    \label{eqn:introEmpiricalMeasure}
\end{equation}
The MEM method based on this measure, which we call the {\em empirical MEM method},  was analyzed in \cite{mkr2024maximumentropy}. That work established almost sure convergence of empirical MEM solutions $x_{\hat \mu_{n}}$ to $x_\mu$, including suitable approximate solutions, and obtained finite-sample convergence guarantees in expectation. To wit, an $n^{-1/4}$ sample complexity with the quadratic data-fidelity term, $g=\frac12\|\cdot\|^2$, was derived, but it was conjectured that this rate was not sharp and that the parametric rate $n^{-1/2}$ should hold.

The first objective of the present paper is to resolve this question: We prove that the empirical MEM estimator converges in expectation at the parametric rate
\begin{equation}\label{eqn:introPrimalRate}
    \E\left[\left\|x_{\hat \mu_{n}}-x_{\mu}\right\|\right]    =
    O\left(n^{-1/2}\right)
\end{equation}
for all convex  and  $L$-smooth fidelity terms (see Assumption \eqref{eq:ass} below). Similar to the analysis in \cite{mkr2024maximumentropy}, we heavily rely on Fenchel-Rockafellar duality: First, we establish stability results for solutions of the (Fenchel-Rockafellar) dual MEM problem \eqref{eqn:introMEMDual} under perturbations of the underlying probability measure (\Cref{thm:parametricRatesDual}). Using this, we then prove the $n^{-1/2}$-rate of convergence of empirical dual solutions in expectation. Moreover, by establishing that  the primal-dual recovery \eqref{eqn:introprimalDualRecoveryFormula} is  Lipschitz-stable (\Cref{lem:LipshitzConstant}), we obtain stability of primal solutions with respect to the underlying measure (\Cref{prop:primalStability}) and, ultimately, the desired convergence from \eqref{eqn:introPrimalRate} (\Cref{thm:parametricRatesPrimalSoln}).
Moreover, we show that, for an $\varepsilon$-optimal empirical dual solution  $z_{\hat \mu_{n},\varepsilon}$ and its associated primal reconstruction $x_{\hat \mu_{n},\varepsilon}:=\nabla L_{\hat \mu_{n}}(A^\top z_{\hat \mu_{n},\varepsilon})$,  the bound 
\begin{equation}
    \E\left[\left\|x_{\hat \mu_{n},\varepsilon}-x_{\mu}\right\|\right]
    \leq
    C_{1}n^{-1/2}+C_{2}\sqrt{\varepsilon},
    \label{eqn:introStatCompTradeoff}
\end{equation}
holds for constants independent of $n$ and $\varepsilon$ (\Cref{th:EpsSol}).

The second objective of the paper is computational: The empirical prior in \eqref{eqn:introEmpiricalMeasure} is attractive precisely when a large training set is available, but the standard empirical MEM dual involves the full sample at every evaluation. Consequently, deterministic full-batch methods become increasingly costly as $n$ grows. To address this limitation, we show that the empirical dual can be reformulated as a risk minimization problem. Concretely,   we show that with 
$
\ell(x,z)=\exp\left(\alpha g^{\ast}(-z/\alpha)-\langle b,z\rangle+\langle A^{\top}z,x\rangle \right),
$
the MEM dual problem \eqref{eqn:introMEMDual} has the same solution(s) as 
\begin{equation}
    \min_{z \in \R^{d}}
 \E_{X\sim\mu}[\ell (X,z)]
    \label{eqn:introPopulationRisk}
\end{equation}
Consequently, the empirical MEM dual can be cast as 
\begin{equation}
    \min_{z \in \R^{d}}
 \frac{1}{n}\sum_{i=1}^{n}\ell(X_{i},z).
    \label{eqn:introEmpiricalRisk}
\end{equation}
This reformulation places empirical MEM within the standard framework of stochastic optimization. It permits stochastic-gradient updates based on individual samples or mini-batches and removes the need to process the entire prior dataset at every iteration.

The exponential terms inherited from the moment generating function create a difficulty that is specific to this formulation: stochastic gradients need not be uniformly bounded and may become numerically unstable when the iterates move into unfavorable regions. Motivated by stochastic gradient clipping methods, including the analysis of \cite{mai2021stability}, we study a  projected clipped stochastic-gradient scheme adapted to the empirical MEM objective.  In this extended real-valued setting, we refine the previous treatment to explicitly account for domain considerations on $\ell$. We establish almost sure convergence under explicit assumptions and investigate the method numerically on inverse problems of increasing scale. The experiments demonstrate that the stochastic formulation can use substantially larger empirical priors than the original full-batch formulation while retaining competitive reconstruction quality.

The statistical and computational contributions are complementary. The statistical analysis controls the error introduced by replacing the population prior $\mu$ with the empirical prior $\hat \mu_{n}$. The optimization analysis controls the additional error introduced by solving the empirical problem only approximately, cf. \eqref{eqn:introStatCompTradeoff}. 
This decomposition separates the statistical error from the optimization error and gives a natural stopping principle: solving the empirical problem far beyond the accuracy permitted by the sample size provides little statistical benefit.

To summarize, our main contributions are the following:
\begin{enumerate}[label=(\roman*)]
    \item We prove nonasymptotic stability bounds for the dual and primal MEM solutions under perturbations of the prior measure.

    \item We establish the parametric convergence rate $O(n^{-1/2})$ in expectation for empirical dual and primal solutions, improving the previously known $O(n^{-1/4})$ rate from \cite{mkr2024maximumentropy}.

    \item We extend the statistical analysis to approximate empirical solutions and obtain an explicit decomposition of statistical and optimization errors.

    \item We reformulate the empirical MEM dual as an empirical risk minimization problem, thereby enabling stochastic optimization methods whose per-iteration cost does not scale with the full sample size.

    \item We develop and analyze a  projected clipped stochastic-(sub)gradient method for the reformulated problem and demonstrate its performance on large-scale image reconstruction examples.
\end{enumerate}
\medskip

\noindent
The paper is organized as follows: In  \Cref{sec:Prelim} we  collect the required material from convex analysis and probability theory,  and provide some more details on the MEM method. \Cref{sec:parametricRates} develops the stability theory and proves the parametric convergence rates for exact and approximate empirical solutions.
\Cref{sec:ER} derives the expected-risk reformulation and studies stochastic-subgradient methods for the empirical dual.  \Cref{sec:Numerics} presents the numerical experiments.

\bigskip

\noindent
{\em Notation:} Throughout we equip $\R^m$ with the Euclidean norm $\|\cdot\|:=\|\cdot\|_2$. The (Euclidean) distance of a point $x\in \R^m$  to a set $C\in \R^m$ is denoted by ${\rm dist}(x,C)$. 
The closed ball of radius $r$ (in the Euclidean norm) centered at $0$ is denoted by $B_r$.  For a matrix $A \in \R^{d \times m}$ we denote $\Vert A \Vert$ as its operator $2$-norm (i.e., its largest singular value), and $\sigma_{\min}(A) = \sqrt{\lambda_{\min}(A^{\top}A)}$ as the smallest singular value of $A$. 
For any  $C\subset \R^{d}$ set $\vert C \vert_{\infty} := \sup_{c \in C} \Vert c \Vert$.

\section{Preliminaries}\label{sec:Prelim}

In this section we review   some concepts from convex analysis and probability theory that are essential for our study. We then recall the MEM problem briefly introduced above with some more technical details. 

\subsection{Fundamentals from convex analysis}
\label{sec:convexAnalysis}
Throughout, we work with extended real-valued functions $f:\R^{m} \to \overline{\R} = \R \cup \{ \pm \infty \}$. The {\em domain} of a function $f$ is the set $\dom(f) = \{ x \in \R^{m} \: : f(x)  < + \infty \}$. If $\dom(f) \neq \emptyset$ and $f$ is never $- \infty,$ $f$ is called  {\em proper.} A function $f$ is said to be {\em lower semicontinuous (lsc)} 
 if the preimage $f^{-1}([-\infty, a])$ is a closed (possibly empty) set for all $a \in \R$. 
We say that a proper function $f$ is {\em convex} if 
\[
f(\lambda x + (1-\lambda) y) \leq \lambda f(x) + (1-\lambda) f(y)
\]
for every $x,y \in \text{dom}(f)$ and all $\lambda \in (0,1)$. 

For  $f : \R^{m} \to \overline{\R}$, its {\em subdifferential} at $\overline x$ is the collection of all {\em subgradients}  of $f$ at $\overline x$, given by
\begin{equation}
    \partial f(\overline{x}) := \{ v\in\mathbb R^m \: : \: f(x) \geq f(\overline{x} )+\langle v, x- \overline{x} \rangle,\: \forall x \in \R^{m} \}, \label{eqn:subgradientDefinition}
\end{equation}
which reduces to singleton $\{\nabla f(\bar x)\}$ when $f$ is convex and  differentiable at $\overline{x}$. We defer to  \cite{rockafellar1970convex} for more details on subdifferential calculus.  

The {\em (Fenchel) conjugate} of $f:\R^{m} \to \overline{\R}$ is the function $f^{*} : \R^{m} \to \overline{\R}$ defined as 
\[
f^{*}(x^{*}) := \sup_{x \in \R^{m}} \{ \langle x, x^{*} \rangle - f(x) \}. 
\]
An important notion for our study is {\em strong convexity}. %
The following result records useful properties of strongly convex functions %
from standard reference texts, see, e.g., \cite{rockafellar1998variational}. 

\begin{proposition}[Strong convexity]
\label{prop:strongConvexity} Let $f:\R^m \to \overline{\R}$ be lsc proper convex and let $\sigma>0$. The following hold: 

\begin{itemize}
\item[i)] (Characterizations) The following are equivalent:
\begin{itemize}
\item[1.] $f$ is $\sigma$-strongly convex.
\item[2.]  $f^*$ is {\em $1/\sigma$-smooth}, i.e.,  $\nabla f^*$ is (globally) $1/\sigma$-Lipschitz.
\item[3.]  $\partial f$ is {\em strongly monotone}, i.e., for all $x_1,x_2\in\R^m$: 
  \begin{equation}
        \langle v_{1}-v_{2}, x_{1} - x_{2} \rangle \geq \sigma \Vert x_{1} - x_{2}\Vert^{2}, \qquad \forall v_{1} \in \partial f(x_{1}), v_{2} \in \partial f(x_{2}). \label{eqn:strongConvexityMonotoneSubdiff}
    \end{equation}
    \item[4.] For all $x_{1},x_{2} \in \R^{m}$: 
\begin{equation}
    f(x_{2})-f(x_{1}) \geq \langle v,x_{2}-x_{1}\rangle + \frac{\sigma}{2} \Vert x_{2}-x_{1} \Vert^{2},\quad\forall v \in \partial f(x_{1}). \label{eqn:strongConvexVer2}
\end{equation} 
\end{itemize}

\item[ii)](Unique minimizers) If $f$ is $\sigma$-strongly convex it has  a unique minimizer $\bar x$ that satisfies
\[
f(x)-f(\bar x)\geq \frac{\sigma}{2}\|x-\bar x\|^2,\quad \forall x\in \R^m.
\]
\item[iii)] When $f$ is differentiable and $\sigma$-strongly convex, we have 
\begin{equation*}
     \sigma \Vert x_{1} - x_{2}\Vert \leq \Vert \nabla f(x_{1})- \nabla f(x_{2}) \Vert\quad \forall x_1,x_2\in \R^m.
\end{equation*}
\end{itemize}
\end{proposition}

\noindent
The central convex-analytic tool for our study is Fenchel-Rockafellar duality which we briefly recall here, see, e.g., \cite{rockafellar1998variational} for more details.

 \begin{theorem} [Fenchel-Rockafellar duality]\label{th:FenRock}
  Let $\kappa:\R^\ell\to\overline \R$, $\phi:\R^k\to\overline \R$ be proper, lsc, and convex and $L\in \R^{k\times \ell}$. 
  Define  
 \begin{equation}\label{eq:FRP}
\min_{x\in \R^\ell} \phi(Lx)+\kappa(x)\tag{P}
 \end{equation}
 and
 \begin{equation}\label{eq:FRD}
 \min_{y\in \R^k} \kappa^*(L^{\top}y)+\phi^*(-y)\tag{D}.
 \end{equation}
Set
\[
p:= \inf_{x\in \R^\ell}\{\phi(Lx)+\kappa(x)\} \quad {\rm and}\quad d:=-\inf_{y\in \R^k}\{\kappa^*(L^{\top}y)+\phi^*(-y)\},
\]
and assume  the qualification condition 

\[
0\in {\rm int}\left(L(\dom (\kappa))-\dom (\phi)\right).
\]
Then the following hold:
\begin{itemize}
\item[i)] (Strong duality) $p=d$, and the dual problem has a solution.
\item[ii)] (Primal-dual recovery)
If $\bar y$ solves the dual problem \eqref{eq:FRD} and $\kappa^*$ is differentiable (at $L^{\top}\bar y)$, then $\bar x= \nabla \kappa^*(L^{\top}\bar y)$ (uniquely) solves the primal problem \eqref{eq:FRP}. 
\end{itemize}
\end{theorem}

\subsection{Basic notions from probability theory}\label{sec:Prob}

Throughout this paper, we fix a  compact set $\mathcal X\subset \mathbb R^m$ and let $\mathcal P(\mathcal X)$ denote the set of Borel probability measures on $\mathcal X$. For any two probability measures $\nu, \mu \in \mathcal{P}(\mathcal X)$, $\nu+\mu$ is a finite Borel measure. Hence, for any $f$ which is  integrable with respect to $\mu$ and $\nu$ we often write $\int f d\nu +\int fd \mu$ as $\int fd(\nu+\mu)$ and similarly for differences of Borel measures.

The {\em expectation} of a random variable $X$ with distribution $\mu\in\mathcal P(\mathcal X)$ is  
\[
\E[X] =\int_{\mathcal{X}}xd\mu(x) \in \R^{m}.\footnotemark
\]
\footnotetext{This vector-valued integral is to be understood component-wise which is consistent with the usual definition $\E[X]=\left(\E[X_1],\dots, \E[X_m]\right)^{\top}$ of the random vector $X$. This idea also extends to random matrices.}
We sometimes write $\mathbb E_{\mu}[X]$ to stress that the expectation is a function of the underlying distribution. The {\em covariance matrix} of $X$ is 
\begin{equation}
     \mathrm{cov}(X) = \E_{\mu}\left[(X -\E_{\mu}[X])(X -\E_{\mu}[X])^{\top
     }\right] \in \R^{m \times m}, \label{eqn:VarCovMatrix}
\end{equation}
which we sometimes write as $\cov_{\mu}(X)$. 
For scalar random variables (i.e., $m=1$), we instead write $\var$ for the {\em variance}.

The {\em Kullback-Leibler (KL) divergence} \cite{kullback1951information} of $Q \in \mathcal{P}(\mathcal{X})$ with respect to $\mu \in \mathcal{P}(\mathcal{X})$ is defined as
\begin{equation}
    \text{KL}(Q\Vert \mu) := \begin{cases} \int_{\mathcal{X}} \log(\frac{dQ}{d\mu}) d Q, &
\text{ if }Q \ll \mu, \\
    + \infty, & \text{ otherwise.} \end{cases} \label{def-KL}
\end{equation}

\subsection{The MEM problem}\label{sec:MEMPrelim} 
As mentioned in the introduction,  the {\em Maximum Entropy on the Mean (MEM) function} is given by
\begin{equation*}
    \kappa_{\mu}(y) := \inf\{ \mathrm{KL}(Q \: \Vert \: \mu) : \E_{Q}[X] = y , Q \in \mathcal{P}(\mathcal{X}) \},
\end{equation*}
where $\mu\in\mathcal P(\mathcal X)$ is a given prior distribution.
Given a signal $b\in\mathbb R^d$ and an operator $A\in\mathbb R^{d\times m}$, the MEM problem, introduced in \cite{le1999new}, consists of solving the convex optimization problem
\begin{equation}
    x_{\mu} = \argmin_{x \in \R^{m}}  \left\{ \alpha  g(Ax - b) \, +\,   \kappa_{\mu}(x) \right\},  \tag{$P_\mu$} \label{eqn:MEMPrimal}
\end{equation}
where $g$ is a lsc, proper, and convex function which measures the discrepancy between $Ax$ and $b$ and $\alpha>0$ is a hyperparameter that enables balancing the influence of both terms in the objective. Therefore, the MEM problem aims to recover a signal based on a transformed and perhaps noisy observation thereof. The function $g$ is generally used to enforce the fidelity of the signal, $Ax\approx b$, whereas the MEM function enables incorporating prior beliefs about the distribution of the signal.  
As shown in, e.g., \cite[Lemma~2.1]{mkr2024maximumentropy}, this problem has a unique solution under the qualification condition $0\in \mathrm{int}(\dom(g) - A\dom(\kappa_{\mu}))$ and solutions can be obtained by first solving the Fenchel-Rockafellar dual problem
\begin{equation}
    z_{\mu} \in \argmin_{z \in \R^{d}} \phi_{\mu}(z) := \alpha g^{*}(-z/\alpha)-\langle b,z\rangle + L_{\mu}(A^{\top}z), \tag{$D_\mu$} \label{eqn:MEMDual}
\end{equation} Here $L_{\mu}$ is the log-moment generating function (also known as the cumulant generating function) of $\mu$ given by
\begin{equation*}
    L_{\mu}(y):= \log M_{\mu}(y) = \log \int_{\mathcal{X}} \exp\langle y, \cdot \rangle d\mu.
\end{equation*}
Under our standing assumption that $\mathcal{X}$ is compact, we record the following key properties of $L_{\mu}$ and $\kappa_{\mu}$ for  ease of reference.
\begin{lemma} \label{lem:PropertiesOfLogMGF}
    For compact $\mathcal{X}$, let $\mu \in \mathcal{P}(\mathcal{X})$. The following hold:
    \begin{enumerate}[label=(\roman*)]
        \item $L_{\mu}$ is convex, finite-valued, and analytic. In particular, $L_{\mu}$ is continuously differentiable with gradient
    \begin{equation}
         \nabla L_{\mu}(y) = \frac{\int (\cdot) \exp\langle y, \cdot\rangle d\mu}{\int \exp\langle y,\cdot \rangle d\mu}. \label{eqn:GradLogMGF}
    \end{equation}
    \item $\kappa_{\mu} = L^{*}_{\mu}$ and $\kappa_{\mu}^{*} = L_{\mu}$.
    \item $\lim_{\Vert x \Vert \to \infty} \frac{\kappa_{\mu}(x)}{\Vert x \Vert} =\infty$, and $\kappa_{\mu}$ is strictly convex on every convex subset of $\{ x \: \vert \: \partial \kappa_{\mu}(x) \neq \emptyset \}$. Namely, $\kappa_{\mu}$ is supercoercive and essentially strictly convex.
    \end{enumerate}
\end{lemma}

\noindent
These properties are collected in, e.g., \cite[Section 2.2]{mkr2024maximumentropy}.

Given a solution $z_{\mu}$ to \eqref{eqn:MEMDual}, the solution, $x_{\mu}$, to \eqref{eqn:MEMPrimal} can be obtained via the primal-dual recovery formula given by \cref{th:FenRock},
\begin{equation}
    x_{\mu} = \nabla L_{\mu}(A^{\top}z_{\mu}). \label{eqn:primalDualRecoveryFormula}
\end{equation} 
It is noteworthy that the dual solution $z_{\mu}$ can also be used to recover
\[
\bar Q =\argmin\{ \mathrm{KL}(Q \: \Vert \: \mu) : \E_{Q}[X] = x_{\mu} , Q \in \mathcal{P}(\mathcal{X}) \},
\]
i.e.,   the optimal distribution  in the definition of $\kappa_{\mu}(x_{\mu})$, see \cite[Lemma 3]{rioux2020maximum}. Namely, it is given by $\bar Q= \gamma_{\mu,z_{\mu}}$ where, for each $z\in\mathbb R^d$, $\gamma_{\mu,z}\in\mathcal P(\mathcal X)$ is the unique probability distribution which admits the Radon-Nikodym derivative 
   \begin{equation}
\label{eq:defGamma}
\frac{d\gamma_{\mu,z}}{d\mu}:u\in\mathcal X\mapsto \frac{\exp\langle A^{\top}z,u\rangle}{\int \exp\langle A^{\top}z,\cdot\rangle d\mu}. 
\end{equation} 
In effect, in light of the primal-dual recovery condition, it is easily seen that the expectation of $\gamma_{\mu,z_{\mu}}$ is indeed $x_{\mu}$. Precisely, we observe that
\begin{equation}
    \E_{\gamma_{\mu,z_{\mu}}}[X] = \frac{\int u\exp\langle A^{\top}z_{\mu},u\rangle d\mu(u)}{\int \exp\langle A^{\top}z_{\mu},u\rangle d\mu(u)}  = \nabla L_{\mu}(A^{\top}z_{\mu}) = x_{\mu} \label{eqn:expectationgamma}.
\end{equation}
Here the left equality is found by applying $\E_{\gamma_{\mu,z_{\mu}}}[X] = \int x d\gamma_{\mu,z_{\mu}}(x) =\int x \frac{d\gamma_{\mu,z_{\mu}}}{d\mu}(x) d\mu(x)$ as follows from the Radon-Nikodym theorem, see e.g., \cite[Theorem 5.5.4]{dudley2018real}, and the right equality follows from evaluating \eqref{eqn:GradLogMGF} at $A^{\top}z_{\mu}$.

\noindent
Throughout this work we make the following mild standing assumption: 
\begin{tcolorbox}[colback=gray!5!white,colframe=gray!75!black]
\vspace{-1em}
 \begin{equation}\label{eq:ass}
\textrm{$g:\R^d\to \R$ is convex and $\frac{1}{\beta}$-smooth,\tag{A1} i.e., $\nabla g$ is (globally) $\frac{1}{\beta}$-Lipschitz.}
\end{equation}
\end{tcolorbox}

 Moreover, \eqref{eq:ass} implies that $g^*$ is $\beta$-strongly convex which, in particular, yields that 
\begin{tcolorbox}[colback=gray!5!white,colframe=gray!75!black]
\begin{center}
The dual objective function $\phi_\mu$  is $\frac{\beta}{\alpha}$-strongly convex for all $\mu\in \mathcal{P}(\mathcal{X})$.
\end{center}
\end{tcolorbox}

 The prototypical example $g = \frac 12\Vert \cdot \Vert^{2}$ satisfies (\ref{eq:ass}), but there are other examples of interest beyond this. For example, any $g$ which is the Moreau envelope of another proper lsc convex function satisfies \eqref{eq:ass} \cite[Example 10.32]{rockafellar1998variational}. From the machine learning literature, the softplus $g(x) = \log(1+e^{\langle a, x\rangle})$ with $a\in\mathbb R^d$ and the radial loss $g(x) = \sqrt{1+ \Vert x \Vert^{2}}$ also satisfy this condition. Finally, we may also consider any separable loss functions which satisfies the assumptions component-wise, e.g., $g = \sum_{i=1}^{d} g_{i}(x_{i})$ for $g_{i}$ convex and $\frac{1}{\beta_{i}}$-smooth. This last category includes, for example, the log-cosh loss $g_{i} = \log \cosh(\cdot)$ with $\beta_i=1$.

As noted in \Cref{sec:Intro},  in the case where $\mu$ is unknown -- such as in imaging problems -- a natural approach  is to approximate $\mu$ via observed data samples. Explicitly, given samples 
\[
    X_{1},\ldots,X_{n} \stackrel{\mathrm{i.i.d.}}{\sim} \mu,
\]
   the {\em empirical distribution}, 
  \[
  \hat{\mu}_{n} = \frac{1}{n} \sum_{i=1}^{n} \delta_{X_{i}},
  \]
  serves as a natural surrogate for a population-level prior distribution. This choice of prior leads to an {\em empirical dual problem}
  \[
 z_{\hat{\mu}_{n}}\in \argmin_{z\in\mathbb R^d}\left\{ \alpha g^{\ast}\Big(-\frac{ z}\alpha\Big)-\langle b,z\rangle + L_{\hat{\mu}_{n}}(A^{\top}z)\right\}
  \]
  for the log-moment generating $L_{\hat{\mu}_{n}}$ of $\hat{\mu}_{n}$.

\section{Parametric rates for convergence in expectation}
\label{sec:parametricRates}

In the sequel, we establish the parametric mean rate of convergence of the empirical primal and dual solutions towards their population-level counterparts. The main result of this section is as follows. 

\begin{theorem}[Primal and dual rates]
    \label{thm:main}
    Fix $\alpha>0$, a fidelity function $g:\mathbb R^d\to \mathbb R$ satisfying \eqref{eq:ass}, a prior $\mu\in\mathcal P(\mathcal X)$, and let $X_1,\dots, X_n\stackrel{i.i.d.}{\sim}\mu$. Further, let $z_{\hat{\mu}_{n}, \varepsilon}$ be a  $\varepsilon$-approximate solution to the (empirical) dual MEM problem for $\hat{\mu}_{n}$ for some  $\varepsilon\geq 0$ , i.e.,
    \begin{equation*}
        \phi_{\hat{\mu}_{n}}(z_{\hat{\mu}_{n}, \varepsilon}) \leq \inf_{z}\phi_{\hat{\mu}_{n}}(z) + \varepsilon,
    \end{equation*}
    which is obtained by a measurable selection from the set of all $\varepsilon$-approximate solutions. Then, letting $z_{\mu}$ denote the solution of the (population) dual MEM problem for ${\mu}$,  
    \begin{equation}
        \label{eq:dualRate}
        \E \big[\Vert z_{\hat{\mu}_{n}, \varepsilon} - z_{\mu} \Vert\big] \leq \frac{C}{\sqrt n}+\sqrt{\frac{2\alpha}{\beta}\varepsilon},
    \end{equation}
    where $C = \frac{\alpha}{\beta}\|A\|C'$ for $C'=2|\spt(\mu)|_{\infty}\exp\left(2\|A\|\|z_{\mu}\||\spt(\mu)|_{\infty}\right)$. Moreover, setting $K'=|\spt(\mu)|^2_{\infty}\|A\|$, if $ x_{\hat{\mu}_{n}, \varepsilon} := \nabla L_{\hat{\mu}_{n}}(A^{\top}z_{\hat{\mu}_{n}, \varepsilon})$ is the recovered (approximate) primal solution for $\hat \mu_n$ and $x_{\mu}$ is the primal solution for $\mu$, 
    \begin{equation}
        \label{eq:primalRate}
        \E \big[\Vert x_{\hat{\mu}_{n}, \varepsilon} - x_{\mu} \Vert\big] \leq \frac{K'C+C'}{\sqrt{n}} + K' \sqrt{\frac{2\alpha}{\beta}\varepsilon}.
    \end{equation} 
    If, moreover, $A$ is injective, and $g$ is $\frac{1}{\gamma}$-strongly convex, then
\begin{equation*}
\mathbb E \left[    \Vert x_{\hat \mu_n,\varepsilon}- x_{\mu} \Vert\right] \leq \frac{\gamma C}{ \alpha\sigma_{\min}(A)} \frac{1}{\sqrt n}+K' \sqrt{\frac{2\alpha}{\beta}\varepsilon}. 
\end{equation*}
\end{theorem}

We first comment on the condition that solutions are obtained via a measurable selection.

\begin{remark}[Measurable selection]
\label{rmk:measurableSelection}
    Fix $n\in\mathbb N$ and consider the function 
    \[
    F_n: (x_1,\dots, x_n)\times z \in\mathcal X^n\times \mathbb R^d\mapsto \alpha g^{\ast}\Big(-\frac{ z}\alpha\Big)-\langle b,z\rangle +\log\Big(\frac 1n \sum_{i=1}^n \exp\langle z, A x_i\rangle\Big),
    \]
    which is evidently (jointly) lsc. By \cite[Example 14.31]{rockafellar1998variational}, $F_n$ is a normal integrand so that 
    \[
    (x_1,\dots,x_n)\in\mathcal X^n\mapsto\{z\in\mathbb R^d:F_n(x_1,\dots,x_n,z)\leq \inf_{z\in\mathbb R^d} F_n(x_1,\dots,x_n,z)+\varepsilon\}
    \]
    is closed-valued and measurable, noting that the infimum is always attained by strong convexity of the objective,  see Proposition 14.33 and Theorem 14.37 from the same reference. By \cite[Corollary 14.6]{rockafellar1998variational}, there exists a measurable selection of $\varepsilon$-minimizer for $F_n(x_1,\dots,x_n,\cdot)$ and, since the data generating process $\omega\in \Omega\mapsto (X_1(\omega),\dots, X_n(\omega))\in\mathcal X^n$ is measurable (by definition), there exists a measurable selection of the dual $\varepsilon$-minimizers $z_{\hat \mu_n,\varepsilon}$ as the composition of Borel measurable maps is measurable. It follows that the corresponding selection of primal solutions is  measurable. If $\varepsilon =0$, no selection is required as each problem admits a unique solution.   

Practically, any algorithm which takes the samples and returns a $\varepsilon$-minimizer based on an iterative application of measurable operations is a measurable selection rule provided that the initialization and stopping rules are also measurable. 
\end{remark}

{}

 The proof of \cref{thm:main} follows the following main steps. 
\begin{enumerate}
    \item We show that for any $\rho,\eta\in\mathcal P(\mathcal X)$, the unique minimizers $z_{\rho},z_{\eta}$, of $\phi_{\rho},\phi_{\eta},$ satisfy 
 \begin{equation}
 \label{eq:outline1}
 \left\|z_{\eta} - z_{\rho} \right\|\leq C_1 \left|\int \exp\langle A^{\top}z_{\eta},\cdot\rangle d(\eta-\rho)\right|+C_2\left\|\int(\cdot) \exp\langle A^{\top}z_{\eta},\cdot\rangle d(\rho-\eta)\right\|,
 \end{equation}
 for constants $C_1,C_2\geq 0$ which are independent of the choice of $\rho$ and $\eta$, see \cref{lem:StabilityDualGradientBound,lem:ConstantUpperBound}. This is accomplished by leveraging the strong monotonicity of the subgradients of $\phi_{\rho}$ and $\phi_{\eta}$ and the fact that the subgradients depend on $\rho,\eta$ only through the terms $\nabla L_{\rho}$ and $\nabla L_{\eta}$ whose difference can be controlled in terms of the integral terms in \eqref{eq:outline1}, see \cref{lem:elementaryBounds}. 

 \item Next, for a  prior $\mu\in\mathcal P(\mathcal X)$ and the empirical measure $\hat \mu_n$ from $n$ i.i.d. samples $X_1,\dots,X_n$ from $\mu$, it follows from \eqref{eq:outline1} that $\mathbb E\big[\|z_{\hat \mu_n}-z_{\mu}\|\big]$ is bounded above by 
 \[
    C_1 \sqrt{\var\Big[\frac{1}{n}\sum_{i=1}^n\exp\langle A^{\top}z_{\mu},X_i\rangle\Big]}+ C_2 \sqrt{\tr\left(\cov\Big[\frac{1}{n}\sum_{i=1}^nX_i\exp\langle A^{\top}z_{\mu},X_i\rangle\Big] \right)},
 \]
 so that both terms scale as $n^{-1/2}$ upon showing that that the variance/covariance of the summands are uniformly bounded, see \cref{thm:parametricRatesDual}. To this effect, we show in \cref{thm:stabilityDual} that $\|z_{\mu}\|$ can be bounded \emph{a priori}. Applying the triangle inequality,
 \[
     \mathbb E\big[\|z_{\hat \mu_{n},\varepsilon}-z_{\mu}\|\big] \leq \mathbb E\big[\|z_{\hat \mu_{n}}-z_{\mu}\|\big]+ \mathbb E\big[\|z_{\hat \mu_{n},\varepsilon}-z_{\hat \mu_n}\|\big],  
 \]
 where the first term on the right-hand side is bounded as described above whereas the second is bounded as $\sqrt{\frac{2\alpha}{\beta}(\phi_{\hat \mu_n}(z_{\hat \mu_n,\varepsilon})-\phi_{\hat \mu_n}(z_{\hat \mu_n}))}\leq\sqrt{\frac{2\alpha}\beta\varepsilon}$ as follows from \eqref{eqn:strongConvexVer2}. 

 \item Finally, we transfer the rates for dual solutions to the primal solutions by using the (approximate) primal-dual recovery formula $x_{\hat \mu_n,\varepsilon}=\nabla L_{\hat \mu_{n}}(A^{\top}z_{\hat \mu_n,\varepsilon})$. This is accomplished by leveraging the stability of the gradient with respect to perturbations in the measures mentioned previously and by establishing that $\nabla L_{\mu}$ is Lipschitz continuous. The error due to the approximate nature of solutions is handled exactly as in item 2.  
\end{enumerate}

 The remainder of this section is dedicated to proving \cref{thm:main} starting with two technical lemmas. The first will enable us to characterize the stability of $\nabla L_{\rho}$ at a fixed point $z$ under perturbations of the underlying measure as outlined in item 1 above.

\begin{lemma}[Gradient stability] 
\label{lem:elementaryBounds}
Fix  $B\in\mathbb R^{d\times m}$, $z\in\mathbb R^d$, and $\rho,\eta\in\mathcal P(\mathcal X)$. Then:
    \begin{enumerate}[label=(\roman*)]
     \setlength\itemsep{1em}
        \item \label{eq:normBound} $\begin{aligned}[t]
            \Bigg\Vert \frac{\int B(\cdot)\exp\langle A^{\top}z,\cdot\rangle d\rho}{\int \exp\langle A^{\top}z,\cdot\rangle d\rho}&- \frac{\int B(\cdot)\exp\langle A^{\top}z,\cdot\rangle d\eta}{\int \exp\langle A^{\top}z,\cdot\rangle d\eta} \Bigg\Vert \\ 
            &\hspace{-7em}\leq \kappa(B,\rho,\eta,z) \left|\int \exp\langle A^{\top}z,\cdot\rangle d(\eta-\rho)\right| 
             +\zeta(B,\eta,z)\left\|\int(\cdot) \exp\langle A^{\top}z,\cdot\rangle d(\rho-\eta)\right\|
        \end{aligned}$ \\
        where, recalling the measure $\gamma_{\rho,z}$ as defined in \eqref{eq:defGamma},
        \begin{align*}
        \zeta(B,\eta,z):=\|B\|\left(\int \exp\langle A^{\top}z,\cdot\rangle d\eta\right)^{-1},\quad \kappa(B,\rho,\eta,z):=\|\mathbb E_{\gamma_{\rho,z}}[X]\|\zeta(B,\eta,z).
       \end{align*} 
        \item \label{eq:gammaexpectation} For $\gamma_{\rho,z}$ as defined by \eqref{eq:defGamma},  $\|\mathbb E_{\gamma_{\rho,z}}[X]\|\leq \mathbb E_{\gamma_{\rho,z}}[\|X\|]\leq  |\spt(\rho)|_{\infty}$.
        \item  \label{eq:kappaBounds}  $\kappa(B,\rho,\eta,z)$ and $\zeta(B,\eta,z)$ satisfy the bounds \[
        \begin{aligned} 
        \kappa(B,\rho,\eta,z)&\leq \|B\||\spt(\rho)|_{\infty}\exp\left(\|A\|\|z\||\spt(\eta)|_{\infty}\right)  \\ 
        \zeta(B,\eta,z)&\leq \|B\|\exp\left(\|A\|\|z\||\spt(\eta)|_{\infty}\right).
        \end{aligned} 
        \]
        
    \end{enumerate}
\end{lemma}

\begin{proof} For the first assertion, we directly compute
\begin{align*}
 &\frac{\int B(\cdot)\exp\langle A^{\top}z,\cdot\rangle d\rho}{\int \exp\langle A^{\top}z,\cdot\rangle d\rho}- \frac{\int B(\cdot)\exp\langle A^{\top}z,\cdot\rangle d\eta}{\int \exp\langle A^{\top}z,\cdot\rangle d\eta} \nonumber\\[0.5em]
 &= \frac{\int B(\cdot)\exp\langle A^{\top}z,\cdot\rangle d\rho \int \exp\langle A^{\top}z,\cdot\rangle d\eta - \int B(\cdot)\exp\langle A^{\top}z,\cdot\rangle d\eta \int \exp\langle A^{\top}z,\cdot\rangle d\rho}{\int \exp\langle A^{\top}z,\cdot\rangle d\rho \int \exp\langle A^{\top}z,\cdot\rangle d\eta}  \\[0.5em]
 &= \frac{\int B(\cdot)\exp\langle A^{\top}z,\cdot\rangle d\rho\int \exp\langle A^{\top}z,\cdot\rangle d(\eta-\rho)+\int  \exp\langle A^{\top}z,\cdot\rangle d\rho\int B(\cdot) \exp\langle A^{\top}z,\cdot\rangle d(\rho-\eta)}{\int \exp\langle A^{\top}z,\cdot\rangle d\rho\int \exp\langle A^{\top}z,\cdot\rangle d\eta} 
\end{align*}
where in the last line we have added and subtracted $\int B(\cdot)\exp\langle A^{\top}z,\cdot\rangle d\rho \int \exp\langle A^{\top}z,\cdot\rangle d\rho$ from the numerator.
Taking the norm on both sides, using the linearity of the integral, and using the fact that $\|Bw\|\leq \|B\|\|w\|$ for any $w\in\mathbb R^m$,
\begin{equation*}
    \begin{aligned}
    &\left\|
\frac{\int B(\cdot)\exp\langle A^{\top}z,\cdot\rangle d\rho}{\int \exp\langle A^{\top}z,\cdot\rangle d\rho}- \frac{\int B(\cdot)\exp\langle A^{\top}z,\cdot\rangle d\eta}{\int \exp\langle A^{\top}z,\cdot\rangle d\eta}\right\| \\[0.5em]
=& \left\Vert\frac{\int B(\cdot)\exp\langle A^{\top}z,\cdot\rangle d\rho\int \exp\langle A^{\top}z,\cdot\rangle d(\eta\mspace{-2.2mu}-\mspace{-2.2mu}\rho)\mspace{-2.2mu}+\mspace{-2.2mu}\int  \exp\langle A^{\top}z,\cdot\rangle d\rho\int B(\cdot) \exp\langle A^{\top}z,\cdot\rangle d(\rho\mspace{-2.2mu}-\mspace{-2.2mu}\eta)}{\int \exp\langle A^{\top}z,\cdot\rangle d\rho\int \exp\langle A^{\top}z,\cdot\rangle d\eta} \right \Vert \\[0.5em]
\leq& \frac{\Vert B \Vert}{\int \exp\langle A^{\top}z, \cdot \rangle d\eta}\left \Vert \frac{\int (\cdot)\exp\langle A^{\top}z, \cdot \rangle d\rho}{\int \exp\langle A^{\top}z, \cdot \rangle d\rho} \right \Vert \left\vert \int \exp\langle A^{\top}z,\cdot\rangle d(\eta-\rho) \right\vert  \\
& \hspace{6.8cm} + \frac{\Vert B \Vert}{\int \exp\langle A^{\top}z, \cdot \rangle d\eta}  \left\|\int(\cdot) \exp\langle A^{\top}z,\cdot\rangle d(\rho-\eta)\right\|,
\end{aligned}
\end{equation*}
which  corresponds to item \ref{eq:normBound}. 

On the other hand, item \ref{eq:gammaexpectation} follows from a direct application of Jensen's inequality, whereby $\|\mathbb E_{\gamma_{\rho,z}}[X]\|\leq \mathbb E_{\gamma_{\rho,z}}[\|X\|]\leq |\spt(\rho)|_{\infty}$, where the latter inequality is a consequence of the definition of  $\gamma_{\rho,z}$ in \eqref{eq:defGamma} which asserts that $\spt(\gamma_{\rho,z})\subset \spt(\rho)$.

For the  estimates in item \ref{eq:kappaBounds}, we first observe that for any $u\in\spt(\eta)$, we have that
$\exp\left(-\|A\|\|z\||\spt(\eta)|_{\infty}\right)\leq 
        \exp\langle A^{\top} z,u \rangle \leq \exp\left(\|A\|\|z\||\spt(\eta)|_{\infty}\right).$
    Integrating this inequality gives
    \begin{equation*}
        \exp\left(-\|A\|\|z\||\spt(\eta)|_{\infty}\right)\leq 
        \int \exp\langle A^{\top} z,\cdot \rangle d\eta \leq \exp\left(\|A\|\|z\||\spt(\eta)|_{\infty}\right).
    \end{equation*}
    We conclude that $\zeta(B,\eta,z) = \|B\|\left(\int \exp\langle A^{\top}z,\cdot\rangle d\eta\right)^{-1}\leq \|B\|\exp\left(\|A\|\|z\||\spt(\eta)|_{\infty}\right)$. Together with item \ref{eq:gammaexpectation}, this also shows that
    \[
    \kappa(B,\rho,\eta,z)=
    \|\mathbb E_{\gamma_{\rho,z}}[X]\|\zeta(B,\eta,z)\leq \|B\||\spt(\rho)|_{\infty}\exp\left(\|A\|\|z\||\spt(\eta)|_{\infty}\right).
    \] 
\end{proof}
The second technical lemma establishes the existence of subgradients at certain points, which is used when carrying out item 1 in the proof outline.

\begin{lemma}[Nonempty subdifferential]\label{lem:nonemptySubdifferentials} Fix $\rho,\eta \in\mathcal P(\mathcal X)$ and let $z_{\rho},z_{\eta}$ be solutions to the dual MEM problem with prior $\rho$ and $\eta$. Then, $\partial \phi_{\rho}(z_{\eta})$ is nonempty and, in particular, 
\begin{equation}
    A \nabla L_{\rho}(A^{\top}z_{\eta}) - A\nabla L_{\eta}(A^{\top}z_{\eta}) \in \partial \phi_{\rho}(z_{\eta}). \label{eqn:particularSubdiffElement}
\end{equation} 
\end{lemma}

\begin{proof}
First, observe that  $\mathrm{rint}( \dom (L_{\eta} \circ A^{\top} - \langle b,\cdot \rangle) )  = \R^{d}$, and as both $g^{*}$ and $L_{\rho} \circ A^{\top} - \langle b,\cdot \rangle$ are proper convex functions, the sum rule \cite[Theorem 23.8]{rockafellar1970convex}
\begin{equation*}
        \partial \phi_{\eta}(z)  =  \partial( \alpha g^{*}( -z/\alpha) ) + \partial\left( -\langle b, z \rangle+ L_{\eta}(A^{\top}z)\right) =   - \partial g^{*}( -z/\alpha) -b +A\nabla L_{\eta}(A^{\top}z), \label{eqn:nonempty}
    \end{equation*}
    holds for each $z\in\mathbb R^d$,
    where the final equality uses the differentiability of $L_{\eta}$ from \cref{lem:PropertiesOfLogMGF}. At the minimizer $z_{\eta}$, the first-order optimality condition reads
    \begin{equation*}
        0 \in \partial \phi_{\eta}(z_{\eta}) = - \partial g^{*}( -z_{\eta}/\alpha) -b +A \nabla L_{\eta}(A^{\top} z_{\eta}),
    \end{equation*}
    and hence
    \begin{equation*}
        b-A \nabla L_{\eta}(A^{\top}z_{\eta}) \in -\partial  g^{*}( -z_{\eta}/\alpha).
    \end{equation*}
    For any other measure $\rho$, the sum rule then gives 
    $
    \partial \phi_{\rho}(z_{\eta}) = -\partial g^{*}( -z_{\eta}/\alpha) -b +A \nabla L_{\rho}(A^{\top} z_{\eta}),
    $
    which is nonempty as it contains, e.g., $A \nabla L_{\rho}(A^{\top}z_{\eta}) - A\nabla L_{\eta}(A^{\top}z_{\eta}) $ as follows from the previous display.  
\end{proof}

With the technical results of \cref{lem:elementaryBounds} and \cref{lem:nonemptySubdifferentials} in hand, we now move to analyzing the stability properties of dual solutions of the MEM problem under perturbations of the prior measure which will naturally lead to the claimed statistical rates. 

\subsection{Stability and rates for dual solutions}
As noted in the proof outline,
our approach relies on the $\beta/\alpha$-strong convexity of the dual objective, which follows from \eqref{eq:ass}. In effect, it enables us to bound the difference between solutions for different priors in terms of the norm of a subgradient of one of the objectives and to obtain \emph{a priori} estimates on the norm of dual solutions as stated next.

\begin{theorem}[Estimates for dual solutions]
\label{thm:stabilityDual}
    Fix $\rho,\eta \in\mathcal P(\mathcal X)$, and let $z_{\rho},z_{\eta}$ be solutions to the dual MEM problem with prior $\rho$ and $\eta$, respectively. Then, 
    \begin{enumerate}[label=(\roman*)]
    \item \label{eqn:normDifferenceBoundedBySubdiff} Letting $G= \partial \phi_{\eta}(z_{\rho}) \cup \partial\phi_{\rho}(z_{\eta})$, we have that
    \begin{equation*}
                \Vert z_{\rho} - z_{\eta} \Vert \leq\frac{\alpha}{\beta} \min_{v\in G} \Vert v \Vert.
    \end{equation*}
    \item \label{eqn:boundedNormOfSolutions} For any $z\in\mathbb R^d$, it holds that
    \begin{align*}
        \Vert z_{\eta} -z \Vert \leq \frac{\alpha}{\beta}\left[ \min_{v_{g} \in \partial g^{*}(-z/\alpha)} \Vert v_{g} \Vert +  \Vert b \Vert  + \Vert A \Vert \vert \spt (\eta) \vert_{\infty} \right],
    \end{align*}
    whereby $\|z_{\eta}\|\leq  \frac{\alpha}{\beta}\left[ \min_{v_{g} \in \partial g^{*}(-z/\alpha)} \Vert v_{g} \Vert +  \Vert b \Vert  + \Vert A \Vert \vert \spt (\eta) \vert_{\infty} \right]+\|z\|.$
    \end{enumerate}
\end{theorem}

\begin{proof} 
   \ref{eqn:normDifferenceBoundedBySubdiff}: First by \cref{lem:nonemptySubdifferentials}, there exists $v \in \partial \phi_{\eta}(z_{\rho})$, and $w \in \partial \phi_{\eta}(z_{\eta})$. Then, as $\phi_{\eta}$ is $\beta/\alpha$-strongly convex, by \cref{eqn:strongConvexityMonotoneSubdiff}
    \begin{equation*}
        \Vert z_{\rho} - z_{\eta} \Vert^{2} \leq \frac{\alpha}{\beta}\langle v - w , z_{\rho} - z_{\eta} \rangle \leq \frac{\alpha}{\beta} \Vert  v - w \Vert \Vert z_{\rho} - z_{\eta} \Vert,
    \end{equation*}
    so that $\Vert z_{\rho} - z_{\eta} \Vert \leq \frac \alpha\beta\Vert v - w \Vert $. In particular, as $z_{\eta}$ is a minimizer of $\phi_{\eta}$, we may choose $w = 0 \in \partial \phi_{\eta}(z_{\eta})$. Hence 
    \begin{equation*}
        \Vert z_{\rho} - z_{\eta} \Vert \leq \frac{\alpha}{\beta}\Vert v \Vert\quad \forall v \in \partial \phi_{\eta}(z_{\rho}).
    \end{equation*}
    {As $\phi_{\rho}$ is also $\beta/\alpha$-strongly convex, we may repeat this argument with reversed roles of $\eta$ and $\rho$ to find
    \[
    \|z_\rho-z_\eta\|\leq \frac{\alpha}{\beta}\|v\| \quad \forall v\in \partial\phi_{\rho}(z_\eta)
    \]
    Combining the last  two inequalities, we find
    \begin{equation*}
        \Vert z_{\rho} - z_{\eta} \Vert \leq\frac{\alpha}{\beta} \inf_{v\in G} \Vert v \Vert,
    \end{equation*}
    for $G$ as defined above.} Finally, as $G$ is closed and $\Vert \cdot \Vert$ has compact level sets, the infimum is attained.
    \smallskip

    \noindent
    \ref{eqn:boundedNormOfSolutions}: By the subdifferential sum rule \cite[Theorem 23.8]{rockafellar1970convex},
    \begin{equation*}
        \partial \phi_{\eta}(z)  =   -\partial g^{*}( -z/\alpha) -b +A\nabla L_{\eta}(A^{\top}z),\forall z\in\mathbb R^d.
    \end{equation*}
    We split the proof into two cases: first, if $\partial \phi_{\eta}(z) \neq \emptyset$, for each $v \in \partial \phi_{\eta}(z)$, and for the particular point $0 \in \partial \phi_{\eta}(z_{\eta})$, we have (by the same arguments as above)  that
    \begin{equation*}
         \Vert z_{\eta} - z \Vert \leq  \frac{\alpha}{\beta} \Vert v \Vert. 
    \end{equation*} 
 Since $v \in \partial \phi_{\eta}(z)$, it is of the form $v = -v_{g} -b +A \nabla  L_{\eta}(A^{\top}z)$ for some $v_{g} \in \partial g^{*}( -z/\alpha)$, so
    \begin{align*}
        \Vert z_{\eta} -z \Vert &\leq\frac{\alpha}{\beta} \min_{v_{g} \in \partial g^{*}(-z/\alpha)} \left \Vert -v_{g}-b + A \nabla L_{\eta}( A^{\top}z) \right\Vert\\
        &=\frac{\alpha}{\beta} \min_{v_{g} \in \partial g^{*}(-z/\alpha)} \left \Vert -v_{g}-b + A \E_{\gamma_{\eta,z}}[X] \right\Vert \\
        &\leq \frac{\alpha}{\beta}\left[ \min_{v_{g} \in \partial g^{*}(-z/\alpha)} \Vert v_{g} \Vert +  \Vert b \Vert  + \Vert A \Vert \cdot \vert \spt (\eta) \vert_{\infty} \right],
    \end{align*} 
    where the penultimate line uses the property $\nabla L_{\eta}(A^{\top}z) = \E_{\gamma_{\eta,z}}[X]$ from \eqref{eqn:expectationgamma} and where the last line follows from \cref{lem:elementaryBounds} \ref{eq:gammaexpectation} and the triangle inequality. Once again, we emphasize that the minimum over subgradients is attained. 
    
   In the case that $\partial \phi_{\eta}(z) = \emptyset$, then also $ -\partial g^{*}( -z/\alpha)= \partial \phi_{\eta} (z) + b -A \nabla  L_{\eta}(A^{\top}z) = \emptyset$. Therefore the upper bound (ii) is a minimization over the empty set and takes the value $+\infty$ so that the result holds vacuously.
\end{proof}

\begin{remark}[On solution estimates]
    As stated, the bound on the norm of solutions in \cref{thm:stabilityDual} (ii) may be vacuous depending on the choice of evaluation point $z\in\mathbb R^d$. Indeed, as stated in the proof of that result, $\partial g^{\ast}(-z/\alpha)$ may be empty.

    Nevertheless, we point out that  $\dom \partial g^*=\left\{u: \partial g^*(u)\neq \emptyset\right\}$ is non-empty as $g^*$ is proper, see, e.g.,  \cite[Corollary 23.4]{rockafellar1970convex} (and use the fact that $g^*$ is proper).

\end{remark}

We now show that the upper bound on $\|z_{\rho}-z_{\eta}\|$ furnished in \cref{thm:stabilityDual} can be rewritten in terms of an expression involving the differences between the measures $\eta$ and $\rho$.  In the sequel, it will be expedient to establish the following shorthand notation. Given a matrix $B$, and measures $\rho, \eta \in \mathcal{P}(\mathcal{X})$ we define
\begin{align*}
     \kappa'(B,\rho,\eta)&:=\max\left\{\kappa(B,\rho,\eta,z_{\eta}),\kappa(B,\eta,\rho,z_{\eta})\right\}, \nonumber \\
     \zeta'(B,\rho,\eta)&:=\max\left\{\zeta(B,\rho,z_{\eta}),\zeta(B,\eta,z_{\eta})\right\}, \label{eqn:primeConstDefn}
\end{align*}
where $
        \zeta(B,\eta,z)=\|B\|\left(\int \exp\langle A^{\top}z,\cdot\rangle d\eta\right)^{-1}$ and  $\kappa(B,\rho,\eta,z)=\|\mathbb E_{\gamma_{\rho,z}}[X]\|\zeta(B,\eta,z),$  as defined in \cref{lem:elementaryBounds} (i). 
\begin{corollary}[Stability of dual solutions]\label{lem:StabilityDualGradientBound}
 Fix $\rho,\eta \in\mathcal P(\mathcal X)$ and let $z_{\rho},z_{\eta}$ solve the dual MEM problem with prior $\rho$ and $\eta$, respectively.  Then
     \[
        \left\|z_{\eta}\mspace{-2mu}-\mspace{-2mu} z_{\rho} \right\|\leq \frac{\alpha}{\beta}\kappa'(A,\rho,\eta) \left|\int \exp\langle A^{\top}z_{\eta},\cdot\rangle d(\eta\mspace{-2mu}-\mspace{-2mu}\rho)\right|+\frac{\alpha}{\beta} \zeta'(A,\rho,\eta)\left\|\int(\cdot) \exp\langle A^{\top}z_{\eta},\cdot\rangle d(\rho\mspace{-2mu}-\mspace{-2mu}\eta)\right\|.
     \]
\end{corollary}

\begin{proof}
We first apply \cref{thm:stabilityDual}, which asserts that 
$    \Vert z_{\eta} - z_{\rho} \Vert \leq \frac{\alpha}{\beta} \min_{v 
\in \partial \phi_{\rho}(z_{\eta}) \cup \partial \phi_{\eta}(z_{\rho})} \Vert v \Vert.$
Taking $v = A \nabla L_{\rho}(A^{\top}z_{\eta}) - A\nabla L_{\eta}(A^{\top}z_{\eta}) \in \partial \phi_{\rho}(z_{\eta})$ from \cref{lem:nonemptySubdifferentials} \eqref{eqn:particularSubdiffElement},
\begin{equation}
    \Vert z_{\eta} - z_{\rho} \Vert \leq \frac{\alpha}{\beta}\Vert A \nabla L_{\rho}(A^{\top}z_{\eta}) - A\nabla L_{\eta}(A^{\top}z_{\eta}) \Vert ,\label{eqn:intermediateSubdiffBound}
\end{equation}
and it remains to bound the right-hand side. Recalling the expressions for $\nabla L_{\rho}$ and $\nabla L_{\eta}$ from \eqref{eqn:GradLogMGF}, we then apply \cref{lem:elementaryBounds} \ref{eq:normBound} with $B=A$ and $z = z_{\eta}$ to obtain that
        \begin{align*}
            \| v \| &= \Vert A \nabla L_{\rho}(A^{\top}z_{\eta}) - A\nabla L_{\eta}(A^{\top}z_{\eta}) \Vert \\
            &=\left\Vert \frac{\int A(\cdot)\exp\langle A^{\top}z_{\eta},\cdot\rangle d\rho}{\int \exp\langle A^{\top}z_{\eta},\cdot\rangle d\rho}- \frac{\int A(\cdot)\exp\langle A^{\top}z_{\eta},\cdot\rangle d\eta}{\int \exp\langle A^{\top}z_{\eta},\cdot\rangle d\eta} \right \Vert \\
            &\leq \kappa(A,\rho,\eta,z_{\eta}) \left|\int \exp\langle A^{\top}z_{\eta},\cdot\rangle d(\eta-\rho)\right|+\zeta(A,\eta,z_{\eta})\left\|\int(\cdot) \exp\langle A^{\top}z_{\eta},\cdot\rangle d(\rho-\eta)\right\|.
        \end{align*}
        
\noindent Repeating this computation with $-v$ in place of $v$, we apply \cref{lem:elementaryBounds} \ref{eq:normBound} once more $B=A$ and $z= z_{\eta}$ with the roles of $L_{\rho}$ and $L_{\eta}$ interchanged. This gives
\begin{align*}
\| v \|
&= \Vert - v \Vert\\
&=\Vert  A\nabla L_{\eta}(A^{\top}z_{\eta}) - A\nabla L_{\rho}(A^{\top}z_{\eta}) \Vert \\
&\leq \left\Vert \frac{\int A(\cdot)\exp\langle A^{\top}z_{\eta},\cdot\rangle d\eta}{\int \exp\langle A^{\top}z_{\eta},\cdot\rangle d\eta}- \frac{\int A(\cdot)\exp\langle A^{\top}z_{\eta},\cdot\rangle d\rho}{\int \exp\langle A^{\top}z_{\eta},\cdot\rangle d\rho} \right \Vert\\
&\leq \kappa(A,\eta,\rho,z_{\eta}) \left|\int \exp\langle A^{\top}z_{\eta},\cdot\rangle d(\rho -\eta)\right|+\zeta(A,\rho,z_{\eta})\left\|\int(\cdot) \exp\langle A^{\top}z_{\eta},\cdot\rangle d(\eta-\rho)\right\|.
\end{align*}
Finally, as $\vert \int \exp\langle A^{\top}z_{\eta},\cdot\rangle d(\eta-\rho) \vert = \vert \int \exp\langle A^{\top}z_{\eta},\cdot\rangle d(\rho-\eta) \vert$, and $\Vert \int (\cdot)\exp\langle A^{\top}z_{\eta},\cdot\rangle d(\eta-\rho) \Vert = \Vert \int (\cdot)\exp\langle A^{\top}z_{\eta},\cdot\rangle d(\rho-\eta) \Vert$, we take the maximum of these bounds and substitute into \eqref{eqn:intermediateSubdiffBound} which yields
\begin{align*}
    \Vert z_{\eta}\mspace{-2mu} -\mspace{-2mu}z_{\rho} \Vert &\leq\frac{\alpha}{\beta} \kappa'(A,\rho,\eta)\left|\int \exp\langle A^{\top}z_{\eta},\cdot\rangle d(\eta\mspace{-2mu}-\mspace{-2mu}\rho)\right|+\frac{\alpha}{\beta} \zeta'(A,\rho,\eta) \left\|\int(\cdot) \exp\langle A^{\top}z_{\eta},\cdot\rangle d(\rho\mspace{-2mu}-\mspace{-2mu}\eta)\right\|
\end{align*}
 with the constants $\kappa'(A,\rho,\eta), \zeta'(A,\rho,\eta)$ defined previously.
\end{proof}

In order to establish the desired parametric rate of convergence for the empirical dual MEM solutions, it remains to provide uniform bounds on the constants from \cref{lem:StabilityDualGradientBound}.
        
\begin{lemma}[Constant bounds] \label{lem:ConstantUpperBound}
    Let $\mu \in \mathcal{P}(\mathcal{X})$ and set $\hat \mu_n:=\frac 1n\sum_{i=1}^n\delta_{X_i}$ to be the empirical measure from $n$ i.i.d. samples $X_1,\dots, X_n$ from $\mu$. Then, with probability $1$,
    \begin{align*}
        \kappa'(B,\hat \mu_n,\mu )&\leq  \Vert B \Vert  \vert \spt(\mu) \vert_{\infty} \exp(\Vert A \Vert  \Vert z_{\mu} \Vert  \vert \spt(\mu) \vert_{\infty}), \\
         \zeta'(B,\hat \mu_{n},\mu)&\leq \Vert B \Vert  \exp(\Vert A \Vert  \Vert z_{\mu} \Vert  \vert \spt(\mu) \vert_{\infty}).
    \end{align*}
\end{lemma}

\begin{proof} \smallskip \noindent (i) First, from \cref{lem:elementaryBounds} \ref{eq:kappaBounds}, 
\begin{equation*}
\begin{aligned}
\kappa(B,\hat \mu_n,\mu, z_{\mu}) &\leq \Vert B \Vert \vert \spt(\hat{\mu}_{n}) \vert_{\infty}\exp(\Vert A \Vert  \Vert z_{\mu} \Vert  \vert \spt(\mu) \vert_{\infty}) \\
&\leq   \Vert B \Vert  \vert \spt(\mu) \vert_{\infty} \exp(\Vert A \Vert  \Vert z_{\mu} \Vert  \vert \spt(\mu) \vert_{\infty}),
\end{aligned}
\end{equation*} 
where the second inequality holds since $\spt (\hat{\mu}_{n})\subset \spt(\mu)$ with probability $1$. Similarly, we have:
\begin{equation*}
\begin{aligned}
     \kappa(B, \mu,\hat{\mu}_{n},z_{\mu}) &\leq \Vert B \Vert \vert \spt(\mu) \vert_{\infty}\exp(\Vert A \Vert  \Vert z_{\mu} \Vert  \vert \spt(\hat{\mu}_{n}) \vert_{\infty}) \\
     &\leq   \Vert B \Vert  \vert \spt(\mu) \vert_{\infty} \exp(\Vert A \Vert  \Vert z_{\mu} \Vert  \vert \spt(\mu) \vert_{\infty}).
\end{aligned}
\end{equation*}
It readily follows that 
\begin{align*}
    \kappa'(B,\hat \mu_n,\mu ) &= \max\{\kappa(B, \mu,\hat{\mu}_{n},z_{\mu}),\kappa(B, \hat{\mu}_{n},\mu,z_{\mu}) \} \\
    &\leq  \Vert B \Vert  \vert \spt(\mu) \vert_{\infty} \exp(\Vert A \Vert  \Vert z_{\mu} \Vert  \vert \spt(\mu) \vert_{\infty}).
\end{align*}
\smallskip \noindent (ii) As for item (i), applying \cref{lem:elementaryBounds} \ref{eq:kappaBounds} we have
\begin{equation*}
    \zeta(B,\hat \mu_n,z_{\mu}) \leq \Vert B \Vert  \exp(\Vert A \Vert  \Vert z_{\mu} \Vert  \vert \spt(\hat{\mu}_{n}) \vert_{\infty}) \leq \Vert B \Vert  \exp(\Vert A \Vert  \Vert z_{\mu} \Vert  \vert \spt(\mu) \vert_{\infty}),
    \end{equation*}
    where we have again used that $\spt(\hat{\mu}_{n}) \subset \spt(\mu)$. On the other hand, we have directly from \cref{lem:elementaryBounds} \ref{eq:kappaBounds} that
    \begin{equation*}
     \zeta(B,\mu,z_{\mu}) \leq \Vert B \Vert  \exp(\Vert A \Vert  \Vert z_{\mu} \Vert  \vert \spt(\mu) \vert_{\infty}).
    \end{equation*} 
    Hence, 
    $
        \zeta'(B,\hat \mu_{n},\mu) =\max \{\zeta(B,\mu,z_{\mu}) ,\zeta(B,\hat\mu_{n},z_{\mu})  \} \leq \Vert B \Vert  \exp(\Vert A \Vert  \Vert z_{\mu} \Vert  \vert \spt(\mu) \vert_{\infty}).
    $
    \end{proof}

With these preliminary results in hand, we now establish the parametric rate of convergence of empirical dual solutions as stated in \cref{thm:main}.

\begin{theorem}[Parametric rates for dual solutions]
\label{thm:parametricRatesDual} Fix $\mu\in\mathcal P(\mathcal X)$ and let $\hat \mu_n:=\frac 1n\sum_{i=1}^n\delta_{X_i}$  be the empirical measure from $n$ i.i.d. samples $X_1,\dots, X_n$ from $\mu$. Denoting $z_{\mu},z_{\hat \mu_{n}}$ as the solutions to the dual MEM problem with priors $\mu$ and $\hat{\mu}_{n}$, respectively, we have
\[
    \mathbb E[\|z_{\hat\mu_n}-z_{\mu}\|] \leq \frac{C}{\sqrt n},
\]
where $C = \frac{2\alpha}{\beta}\|A\||\spt(\mu)|_{\infty}\exp\left(2\|A\|\|z_{\mu}\||\spt(\mu)|_{\infty}\right)$. 
\end{theorem}
\begin{proof}
   We first apply \Cref{lem:StabilityDualGradientBound} with  $\rho = \hat{\mu}_{n}$ and $\eta = \mu$ to obtain that
   \begin{equation}\label{eq:intermediateExpectation}
   \begin{aligned}
         \|z_{\hat \mu_n}-z_{\mu}\| 
        &\leq\frac{\alpha}{\beta} \kappa'(A,\hat \mu_n,\mu)\left|\int \exp\langle A^{\top}z_{\mu},\cdot\rangle d(\mu-\hat\mu_n)\right|
       \\ 
       & 
        +\frac{\alpha}{\beta} \zeta'(A,\hat \mu_n,\mu)\left\|\int (\cdot)\exp\langle A^{\top}z_{\mu},\cdot\rangle d(\hat\mu_n - \mu)\right\|.      
  \end{aligned} 
   \end{equation}
   Given that $\kappa'(A,\hat \mu_n,\mu),\zeta'(A,\hat \mu_n,\mu)$ admit upper bounds which do not depend on the choice of samples by \cref{lem:ConstantUpperBound}
   we proceed by bounding the expected value of the two integral terms. 

   First, let $f=\exp\langle A^{\top}z_{\mu},\cdot\rangle$ and observe that 
   \begin{align*}
        \mathbb E\left[\left|\int fd\mu-\frac 1n \sum_{i=1}^nf(X_i)\right|\right]&\leq\sqrt{\mathbb E\left[\left|\int fd\mu-\frac 1n \sum_{i=1}^nf(X_i)\right|^2\right]}\\
        &=\sqrt{\mathbb E\left[\left| \frac{1}{n} \sum_{i=1}^{n} \int fd\mu-\frac 1n \sum_{i=1}^nf(X_i)\right|^2\right]} \\
        &=\sqrt{\frac 1{n^2}\var\left[\sum_{i=1}^nf(X_i)\right]},
   \end{align*}
   where the first line follows from Jensen's inequality, and we have used the fact that $\int f d\mu=\frac{1}{n}\sum_{i=1}^n\int fd\mu$ in the next equality. The final equality follows from setting $Y =\sum_{i=1}^nf(X_i)$ and applying the definition of the variance of $Y$, \eqref{eqn:VarCovMatrix}.
   As the samples $X_1,\dots,X_n$ are independent, \[\var\left[\sum_{i=1}^nf(X_i)\right]=n\var[f(X_1)]\leq n\mathbb E_{\mu}[f^2(X_1)]\leq n\exp\left(2\|A\| \|z_{\mu}\||\spt(\mu)|_{\infty}\right),\] 
   where we have used that $f(X_{1})^{2} = (\exp\langle A^{\top}z_{\mu}, X_{1}\rangle)^{2} \leq \exp\left(2\|A\| \|z_{\mu}\||\spt(\mu)|_{\infty}\right)$ in the final inequality.
Hence, we have shown 
\begin{equation}
    \mathbb E \left[\left|\int \exp\langle A^{\top}z_{\mu},\cdot\rangle d(\mu-\hat\mu_n)\right|\right]\leq \frac{1}{\sqrt n}\exp\left(\|A\| \|z_{\mu}\||\spt(\mu)|_{\infty}\right) \label{eqn:IntegralBound1}
\end{equation}

For the other term of \eqref{eq:intermediateExpectation}, define $h=(\cdot)\exp\langle A^{\top}z_{\mu},\cdot\rangle$. Then once more via Jensen's inequality, we have
\begin{align*}
        \mathbb E\left[\left\|\int hd\mu-\frac 1n \sum_{i=1}^n h(X_i)\right\|\right]
        &\leq\sqrt{\mathbb E\left[\left\|\int hd\mu-\frac 1n \sum_{i=1}^n h(X_i)\right\|^2\right]}
        \\
        &=\sqrt{\mathbb E\left[\left\| \frac{1}{n} \sum_{i=1}^n \left(\int hd\mu-h(X_i)\right)\right\|^2\right]}
        \\
        &
        = \sqrt{\frac 1{n^2}\mathrm{Tr}\left(\cov\left[\sum_{i=1}^n h(X_i)\right]\right)}.
\end{align*}
Here, the final equality can be observed by computing the covariance matrix of the random vector $Y= \sum_{i=1}^n h(X_i).$
Moreover, $\cov\left[\sum_{i=1}^n h(X_i)\right]=n\cov\left[h(X_1)\right]$ due to independence of the samples. Now, expanding the diagonal entries of the covariance matrix, \[
\mathrm{Tr}\left(\cov\left[h(X_1)\right]\right)=\mathbb E\left[\|h(X_1)-\mathbb E[h(X_1)]\|^2\right]\leq |\spt(\mu)|^2_{\infty}\exp(2\|A\|\|z_{\mu}\||\spt(\mu)|_{\infty}),
\]
where  the inequality follows from observing that 
\[
\mathbb E\left[\|h(X_1)-\mathbb E[h(X_1)]\|^2\right]=\mathbb E[\|h(X_1)\|^2]-\|\mathbb E[h(X_1)]\|^2
\leq \mathbb E\left[\|h(X_1)\|^2\right].
\]
Therefore, we have 
\begin{equation}
    \mathbb E\left[\left\|\int (\cdot)\exp\langle A^{\top}z_{\mu},\cdot\rangle d(\hat\mu_n - \mu)\right\|\right]\leq {\frac{1}{\sqrt n}}|\spt(\mu)|_{\infty}\exp\left(\|A\| \|z_{\mu}\||\spt(\mu)|_{\infty}\right).  \label{eqn:integralBound2}
\end{equation}
 Substituting \eqref{eqn:IntegralBound1} and \eqref{eqn:integralBound2} into \eqref{eq:intermediateExpectation} along with the bounds on $\kappa'(A,\hat \mu_n,\mu)$ and $\zeta'(A,\hat \mu_n,\mu)$ from \cref{lem:ConstantUpperBound} yields that  
\begin{equation*}
\mathbb E\left[\left\|z_{\hat \mu_n}-z_{\mu}\right\|\right]\leq \frac{1}{\sqrt n}\frac{2\alpha}\beta \Vert A \Vert  \vert \spt(\mu) \vert_{\infty} \exp\left(2\|A\| \|z_{\mu}\||\spt(\mu)|_{\infty}\right).
\end{equation*}
\end{proof}

The above result naturally generalizes to approximate dual solutions as stated next.

\begin{corollary}[Parametric rates for approximate dual solutions] In the setting of \cref{thm:parametricRatesDual}, let $z_{\hat \mu_n,\varepsilon}$ be a measurable selection of an $\varepsilon$-approximate minimizer of $\phi_{\hat \mu_n}$. Then, 
\[
    \mathbb E[\|z_{\hat \mu_n,\varepsilon}-z_{\mu}\|]\leq \frac{C}{\sqrt n} + \sqrt{\frac{2\alpha}{\beta}\varepsilon}.
\]
\end{corollary}
\begin{proof}
    By the triangle inequality,
    \[
        \mathbb E[\|z_{\hat \mu_n,\varepsilon}-z_{\mu}\|]\leq \mathbb E[\|z_{\hat \mu_n}-z_{\mu}\|]+\mathbb E[\|z_{\hat \mu_n,\varepsilon}-z_{\hat \mu_n}\|].
    \]
    The first term on the right-hand side is bounded by $Cn^{-1/2}$ by \cref{thm:parametricRatesDual}. Since $0\in\partial \phi_{\hat \mu_n}(z_{\hat \mu_n})$, \eqref{eqn:strongConvexVer2} implies that 
    \[
        \|z_{\hat \mu_n,\varepsilon}-z_{\hat \mu_n}\|^2\leq \frac{2\alpha}{\beta}\left( \phi_{\hat \mu_n}(z_{\hat \mu_n,\varepsilon})-\phi_{\hat \mu_n}(z_{\hat \mu_n})\right)\leq \frac{2\alpha}{\beta}\varepsilon,
    \]
    proving the claimed result.
\end{proof}
With this result in hand, we proceed to proving the corresponding convergence rate for the primal solutions. 

\subsection{Stability and rates for primal solutions}

To derive stability and convergence rate results for primal solutions, we now first establish  that the gradient of $L_{\rho}$ is Lipschitz continuous for any choice of $\rho\in\mathcal P(\mathcal X)$.

\begin{lemma}[Lipschitz continuity] \label{lem:LipshitzConstant}
    For any $\rho \in \mathcal{P}(\mathcal{X})$, $L_{\rho} \circ A^{\top}$ is $K$-smooth, where $ K= \vert \spt(\rho) \vert_{\infty}^2 \Vert A \Vert^2$. 
\end{lemma}

\begin{proof}
We compute the first and second partial derivatives of  $L_{\rho}\circ A^{\top}$ below: 
\[
    \begin{aligned}
    \partial_i (L_{\rho}\circ A^{\top}):y\in\mathbb R^d&\mapsto \frac{\int \sum_{k=1}^mA_{ik}u_k\exp\langle y,Au\rangle d\rho(u)}{\int \exp\langle y,Au\rangle d\rho(u)},
    \\
    \\
    \partial_j\partial_i (L_{\rho}\circ A^{\top}):y\in\mathbb R^d&\mapsto \frac{\int \sum_{k=1}^m A_{ik}u_k\sum_{l=1}^mA_{jl}u_l\exp\langle y,Au\rangle d\rho(u)}{\int \exp\langle y,Au\rangle d\rho(u)},
\\
&-\frac{\int \sum_{k=1}^mA_{ik}u_k\exp\langle y,Au\rangle d\rho(u)\int \sum_{l=1}^mA_{jl}u_l\exp\langle y,Au\rangle d\rho(u)}{\left(\int \exp\langle y,Au\rangle d\rho(u)\right)^2}.
    \end{aligned}
\]
Consequently, the Hessian of $L_{\rho}\circ A^{\top}$ can be expressed as 
\[
    \nabla^2 (L_{\rho}\circ A^{\top})(y)=\mathbb E_{\gamma_{\rho,y}}\left[ AUU^{\top}A^{\top} \right]- \mathbb E_{\gamma_{\rho,y}}\left[ AU\right]\mathbb E_{\gamma_{\rho,y}}\left[ AU\right]^{\top}
\]
where $\gamma_{\rho,y}$ is as defined in \eqref{eq:defGamma} and $U\sim\gamma_{\rho,y}$. We can rewrite this expression as
\[
    \nabla^2 (L_{\rho}\circ A^{\top})(y)=A\mathbb E_{\gamma_{\rho,y}}\left[ \left(U-\mathbb E_{\gamma_{\rho,y}}\left[U\right]\right)\left(U-\mathbb E_{\gamma_{\rho,y}}\left[ U\right]\right)^{\top}\right]A^{\top},
\]
so that, for any $w\in\mathbb R^d$ with $\|w\|\leq 1$,
\[
    w^{\top}(\nabla^2(L_{\rho} \circ A^{\top})(y)) w= \mathbb E_{\gamma_{\rho,y}}\left[ \left\|\left(U-\mathbb E_{\gamma_{\rho,y}}\left[ U\right]\right)^{\top}A^{\top}w\right\|^2\right]\leq  \mathbb E_{\gamma_{\rho,y}}\left[ \left\|U-\mathbb E_{\gamma_{\rho,y}}\left[ U\right]\right\|^2\right]\|A\|^2,
\]
where we have applied the Cauchy-Schwarz inequality. We may further bound the expectation as 
\[
    \mathbb E_{\gamma_{\rho,y}}\left[ \left\|U-\mathbb E_{\gamma_{\rho,y}}[U]\right\|^2\right]=\mathbb E_{\gamma_{\rho,y}}\left[ \left\|U\right\|^2\right]-\left\|\mathbb E_{\gamma_{\rho,y}}[U]\right\|^2\leq \vert \spt(\rho)\vert_{\infty}^2.
\]
It follows from the previous two lines that all eigenvalues of $\nabla^2(L_{\rho}\circ A^{\top})(y)$ are bounded above by $\vert \spt(\rho) \vert_{\infty}^2 \|A\|^2 =: K$ so that $\nabla (L_{\rho}\circ A^{\top})$ is $K$-Lipschitz continuous.
\end{proof}

We briefly comment on some implications of \cref{lem:LipshitzConstant}.

\begin{remark} \label{rem:LipschitzCorollaries}
We make two observations as immediate consequences of \cref{lem:LipshitzConstant}: 
\begin{enumerate}[label=(\roman*)]
    \item As $\spt(\rho)\subset \mathcal{X}$,  $\nabla (L_{\rho} \circ A^{\top})$ is $ \vert \mathcal{X} \vert_{\infty}^{2} \Vert A \Vert^{2}$-Lipschitz for all $\rho \in \mathcal{P}(\mathcal{X})$. 
    \item For the choice $A =I$, this also shows that $L_{\rho}$ is $\vert \spt(\rho) \vert_{\infty}^2$-smooth.
\end{enumerate}
\end{remark}

\noindent
With this technical lemma in place, we can leverage the primal-dual recovery formula to establish a stability property for primal solutions under perturbations of the prior. 

\begin{proposition}[Primal solution stability] \label{prop:primalStability}
        Fix $\rho,\eta\in\mathcal P(\mathcal X)$. Let $x_{\rho},x_{\eta}$ solve the primal MEM problem with priors $\rho,\eta$, and $z_{\rho},z_{\eta}$ solve the respective dual MEM problems. Then:
        \begin{enumerate}[label=(\roman*)]
            \item \label{item:dualStab} Setting  $K'=\|A\||\spt(\rho)|_{\infty}^2$, \[ \begin{aligned}  \|x_{\rho}-x_{\eta}\|
            &\leq K'\|z_{\rho}-z_{\eta}\|
            +\kappa'(I,\rho,\eta)\left|\int \exp\langle A^{\top}z_{\eta},\cdot\rangle d(\eta-\rho)\right|\\
            &+\zeta'(I,\rho,{\eta})\left\|\int(\cdot) \exp\langle A^{\top}z_{\eta},\cdot\rangle d(\rho-\eta)\right\|. 
            \end{aligned}\] 
       \item \label{item:dualStabAdditional} If $A$ is injective and $g$ is $\frac{1}{\gamma}$-strongly convex, then:
        \begin{equation*}
        \Vert x_\rho- x_{\eta} \Vert \leq \frac{\gamma }{ \alpha\sigma_{\min}(A)} \Vert z_{\rho} - z_{\eta} \Vert. 
        \end{equation*}
        
        \end{enumerate}
        
\end{proposition}

\begin{proof}
\smallskip \noindent \ref{item:dualStab} As a consequence of the primal-dual recovery formula \eqref{eqn:primalDualRecoveryFormula}, we have
\begin{align*}
    \|x_{\rho}-x_{\eta}\|&=\|\nabla L_{\rho}(A^{\top}z_{\rho})-\nabla L_{\eta}(A^{\top}z_{\eta})\|\\
    &\leq \|\nabla L_{\rho}(A^{\top}z_{\rho})-\nabla L_{\rho}(A^{\top}z_{\eta})\|+\|\nabla L_{\rho}(A^{\top}z_{\eta})-\nabla L_{\eta}(A^{\top}z_{\eta})\|.
\end{align*}
For the former term, we use \cref{rem:LipschitzCorollaries} (ii) to obtain that
\begin{align*}
    \|\nabla L_{\rho}(A^{\top}z_{\rho})-\nabla L_{\rho}(A^{\top}z_{\eta})\| &\leq \vert \spt(\rho) \vert_{\infty}^2  \Vert A^{\top} (z_{\rho} -z_{\eta}) \Vert\\
    &\leq \vert \spt(\rho) \vert_{\infty}^2 \Vert A \Vert \Vert z_{\rho} - z_{\eta} \Vert.
\end{align*}
For the latter term, in view of \eqref{eqn:GradLogMGF}, $\|\nabla L_{\rho}(A^{\top}z_{\eta})-\nabla L_{\eta}(A^{\top}z_{\eta}) \|$ is exactly in the form of \Cref{lem:elementaryBounds} \ref{eq:normBound} for the choices of $B=I$ and $z= z_{\eta}$. Hence
\begin{align*}
    &\|\nabla L_{\rho}(A^{\top}z_{\eta})-\nabla L_{\eta}(A^{\top}z_{\eta}) \| \\
    &\leq \kappa(I,\rho,\eta,z_{\eta}) \left|\int \exp\langle A^{\top}z_{\eta},\cdot\rangle d(\eta-\rho)\right| 
             +\zeta(I,\eta,z_{\eta})\left\|\int(\cdot) \exp\langle A^{\top}z_{\eta},\cdot\rangle d(\rho-\eta)\right\|.
\end{align*} 
Observing that $ \|\nabla L_{\rho}(A^{\top}z_{\eta})-\nabla L_{\eta}(A^{\top}z_{\eta}) \| =  \Vert \nabla L_{\eta}(A^{\top}z_{\eta})- \nabla L_{\rho}(A^{\top}z_{\eta}) \| $, we also apply \Cref{lem:elementaryBounds} \ref{eq:normBound} for the choices of $B=I$ and $z= z_{\eta}$, with the order of $L_{\eta}$ and $L_{\rho}$ interchanged. This gives
\begin{align*}
    &\|\nabla L_{\eta}(A^{\top}z_{\eta})- \nabla L_{\rho}(A^{\top}z_{\eta}) \| \\
    &\leq \kappa(I,\eta,\rho,z_{\eta}) \left|\int \exp\langle A^{\top}z_{\eta},\cdot\rangle d(\eta-\rho)\right| 
             +\zeta(I,\rho,z_{\eta})\left\|\int(\cdot) \exp\langle A^{\top}z_{\eta},\cdot\rangle d(\rho-\eta)\right\|.
\end{align*} 
Taking the maximum over these two bounds gives the result with the corresponding constants $\kappa'(I,\rho,\eta)$ and $\zeta'(I,\rho,\eta)$ as was stated in the theorem.

\smallskip \noindent \ref{item:dualStabAdditional} Assuming that $A$ is injective and $g$ is $\frac{1}{\gamma}$-strongly convex, $g^{*}$ is differentiable and $\gamma$-smooth. Rearranging the first order optimality conditions $\nabla \phi_{\rho}(z_{\rho}) = \nabla \phi_{\eta}(z_{\eta}) =0$ gives
    \begin{equation*}
        A\nabla L_{\eta}(A^{\top}z_\eta) - A \nabla L_{\rho}(A^{\top}z_\rho) =  \nabla g^{*}(-z_{\eta}/\alpha) - \nabla g^{*}(-z_{\rho}/\alpha). 
    \end{equation*}
    As $\Vert Ax \Vert \geq \sigma_{\min}(A) \Vert x \Vert$ for any $x  \in \R^{m}$, and we also have $\sigma_{\min}(A) > 0$ as $A$ is injective,
    \begin{equation}
         \sigma_{\min}(A) \Vert\nabla L_{\eta}(A^{\top}z_\eta) - \nabla L_{\rho}(A^{\top}z_\rho)\Vert \leq  \Vert \nabla g^{*}(-z_{\eta}/\alpha) - \nabla g^{*}(-z_{\rho}/\alpha) \Vert.\label{eqn:InjectiveA} 
    \end{equation}
    Returning to the primal-dual recovery formula \cref{eqn:primalDualRecoveryFormula}, we have $x_\eta - x_{\rho} =  \nabla L_{\eta}(A^{\top}z_\eta) - \nabla L_{\rho}(A^{\top}z_\rho).$ Hence, we find
    \begin{align*}
        \Vert x_{\eta} - x_{\rho} \Vert &= \Vert  \nabla L_{\eta}(A^{\top}z_\eta) -  \nabla L_{\rho}(A^{\top}z_\rho) \Vert \\
        &\leq \frac{1}{\sigma_{\min}(A)} \Vert\nabla g^{*}(-z_{\eta}/\alpha) - \nabla g^{*}(-z_{\rho}/\alpha) \Vert \\
        &\leq \frac{\gamma }{ \alpha\sigma_{\min}(A)} \Vert z_{\eta} - z_{\rho} \Vert,
    \end{align*}
    where we have used \eqref{eqn:InjectiveA} for the first inequality, and the $\gamma$-smoothness of $g^{*}$ in the latter.
\end{proof}

As with the dual solutions, the stability properties provided above enable us to obtain the parametric rate for empirical primal solutions.

\begin{theorem}[Parametric rates for primal solutions] \label{thm:parametricRatesPrimalSoln} Fix $\mu\in\mathcal P(\mathcal X)$ and let $\hat \mu_n:=\frac 1n\sum_{i=1}^n\delta_{X_i}$ be the empirical measure from $n$ i.i.d. samples $X_1,\dots, X_n$ from $\mu$. Let $x_{\mu},x_{\hat \mu_{n}}$ denote the solutions to the primal MEM problem with priors $\mu$ and $\hat{\mu}_{n}$ respectively. Then:
\begin{enumerate}[label=(\roman*)]
        \item Letting $C' = 2|\spt(\mu)|_{\infty}\exp\left(2\|A\|\|z_{\mu}\||\spt(\mu)|_{\infty}\right)$, $C = \frac{\alpha}{\beta} \Vert A \Vert C'$ be the constants from \Cref{thm:parametricRatesDual} and $K' = \vert \spt(\mu) \vert_{\infty}^{2} \Vert A \Vert$, we have \label{item:primalStab} $$
    \mathbb E[\|x_{\hat\mu_n}-x_{\mu}\|] \leq  \frac{K'C+C'}{\sqrt n}.
    $$
\item \label{item:primalStabAdditional} If $A$ is injective and $g$ is $\frac{1}{\gamma}$-strongly convex, then
\begin{equation*}
\mathbb E \left[    \Vert x_{\hat \mu_n}- x_{\mu} \Vert\right] \leq \frac{\gamma C}{ \alpha\sigma_{\min}(A)} \frac{1}{\sqrt n}. 
\end{equation*}
\end{enumerate}

\end{theorem}

\begin{proof}
\smallskip \noindent \ref{item:primalStab} Applying \Cref{prop:primalStability} \ref{item:dualStab} with $\eta = \mu, \rho = \hat{\mu}_{n}$ gives
\begin{align*}
    \|x_{\hat{\mu}_{n}}-x_{\mu}\|
            &\leq K\|z_{\hat{\mu}_{n}}-z_{\mu}\|+\kappa'(I,\hat{\mu}_{n},\mu)\left|\int \exp\langle A^{\top}z_{\mu},\cdot\rangle d(\mu-\hat{\mu}_{n})\right|
            \\
            &+\zeta'(I,\hat{\mu}_{n},{\mu})\left\|\int(\cdot) \exp\langle A^{\top}z_{\mu},\cdot\rangle d(\hat{\mu}_{n}-\mu)\right\|,
\end{align*}
where $K =\|A\||\spt(\hat{\mu}_{n})|_{\infty}^2\leq K'$, noting that $\spt(\hat{\mu}_{n}) \subset \spt(\mu)$.
Taking the expectation of the above inequality and applying \Cref{thm:parametricRatesDual} to the first term yields
\begin{equation*}
\begin{aligned}
    \E[\|x_{\hat\mu_n}-x_{\mu}\|] 
    &\leq  \frac{K'C}{\sqrt n} + \E \left[ \kappa'(I,\hat{\mu}_{n},\mu)\left|\int \exp\langle A^{\top}z_{\mu},\cdot\rangle d(\mu-\hat{\mu}_{n})\right|\right]
    \\
    &
    +\mathbb E\left[\zeta'(I,\hat{\mu}_{n},{\mu})\left\|\int(\cdot) \exp\langle A^{\top}z_{\mu},\cdot\rangle d(\hat{\mu}_{n}-\mu)\right\|\right].
\end{aligned}
\end{equation*}

Hence, it remains to bound the expectation of the latter terms. From \eqref{eqn:IntegralBound1} and \eqref{eqn:integralBound2},
\begin{align*}
    \mathbb E\left[\left|\int \exp\langle A^{\top}z_{\mu},\cdot\rangle d(\mu-\hat\mu_n)\right|\right] &\leq \frac{1}{\sqrt n}\exp\left(\|A\|\|z_{\mu}\||\spt(\mu)|_{\infty}\right), 
    \\\E \left[\left\|\int (\cdot)\exp\langle A^{\top}z_{\mu},\cdot\rangle d(\mu-\hat\mu_n)\right\|\right] &\leq {\frac{1}{\sqrt n}}|\spt(\mu)|_{\infty}\exp\left(\|A\|\|z_{\mu}\||\spt(\mu)|_{\infty}\right).
\end{align*}
As for $\kappa'(I,\hat \mu_n,\mu), \zeta'(I,\hat \mu_n,\mu)$, we apply \cref{lem:ConstantUpperBound} with  $B=I$ to obtain that $\kappa'(I, \hat{\mu}_{n},\mu) \leq |\spt(\mu)|_{\infty}\exp\left(\|A\|\|z_{\mu}\||\spt(\mu)|_{\infty}\right)$  and $\zeta'(I,\hat{\mu}_{n},{\mu}) \leq  \exp(\Vert A \Vert  \Vert z_{\mu} \Vert  \vert \spt(\mu) \vert_{\infty})$. This enables us to control the terms involving the integrals as  
\begin{equation*}
    \frac{1}{\sqrt{n}}\left( \vert \spt(\mu)\vert_{\infty} \exp\left(2 \|A\|\|z_{\mu}\||\spt(\mu)|_{\infty}\right)+ |\spt(\mu)|_{\infty}\exp\left(2\|A\|\|z_{\mu}\||\spt(\mu)|_{\infty}\right) \right),
\end{equation*}
which coincides with $\frac{C'}{\sqrt{n}}$, thus completing the proof of item (i).

\smallskip \noindent \ref{item:primalStabAdditional} Supposing that $A$ is injective and $g$ is $\frac{1}{\gamma}$-strongly convex, we  apply \Cref{prop:primalStability} \ref{item:dualStabAdditional} with $\rho=\hat{\mu}_{n}, \eta = \mu$, giving
\begin{equation*}
    \Vert x_{\hat{\mu}_{n}}- x_{\mu} \Vert \leq \frac{\gamma }{ \alpha\sigma_{\min}(A)} \Vert z_{\hat{\mu}_{n}} - z_{\mu} \Vert .
\end{equation*}
Taking the expectation and applying \Cref{thm:parametricRatesDual} yields
\begin{equation*}
    \E \big[\Vert x_{\hat{\mu}_{n}} - x_{\mu} \Vert\big] \leq  \frac{\gamma }{ \alpha\sigma_{\min}(A)} \E\big[ \Vert z_{\hat{\mu}_{n}} - z_{\mu} \Vert \big]\leq  \frac{\gamma C }{ \alpha\sigma_{\min}(A)} \frac{1}{\sqrt{n}}.
\end{equation*}
    
\end{proof}

\noindent
As before, we extend the sample complexity to the case of  approximate minimizers.

\begin{theorem}[Parametric rates for approximate solutions]\label{th:EpsSol}
    
    Let $z_{\hat{\mu}_{n}, \varepsilon}$ be a measurable selection of an $\varepsilon$-approximate solution to the dual MEM problem for $\hat{\mu}_{n}$%
    and set $ x_{\hat{\mu}_{n}, \varepsilon} := \nabla L_{\hat{\mu}_{n}}(A^{\top}z_{\hat{\mu}_{n}, \varepsilon})$.
    Then,
    \begin{align*}
        \E \big[\Vert x_{\hat{\mu}_{n}, \varepsilon} - x_{\mu} \Vert\big] \leq \frac{K'C+C'}{\sqrt{n}} + K' \sqrt{\frac{2\alpha}{\beta}\varepsilon},
    \end{align*}
    where $K',C,C'$ are the constants of \cref{thm:parametricRatesPrimalSoln}. If, additionally, $A$ is injective and $g$ is $\frac{1}{\gamma}$-strongly convex, then
\begin{equation*}
\mathbb E \left[    \Vert x_{\hat \mu_n,\varepsilon}- x_{\mu} \Vert\right] \leq \frac{\gamma C}{ \alpha\sigma_{\min}(A)} \frac{1}{\sqrt n}+ K' \sqrt{\frac{2\alpha}{\beta}\varepsilon}. 
\end{equation*}
\end{theorem}
\begin{proof}
    The triangle inequality yields
    \begin{align}
        \Vert x_{\hat{\mu}_{n}, \varepsilon} - x_{\mu} \Vert &\leq \Vert x_{\hat{\mu}_{n}, \varepsilon} - x_{\hat{\mu}_{n}} \Vert + \Vert  x_{\hat{\mu}_{n}} - x_{\mu} \Vert \nonumber \\
        &= \Vert \nabla L_{\hat{\mu}_{n}}(A^{\top}z_{\hat{\mu}_{n}, \varepsilon}) - \nabla L_{\hat{\mu}_{n}}(A^{\top}z_{\hat{\mu}_{n}}) \Vert + \Vert  x_{\hat{\mu}_{n}} - x_{\mu} \Vert . \label{eqn:substitutionApprox}
    \end{align}
The expectation of the latter term has already been bounded in these different cases in \cref{thm:parametricRatesPrimalSoln}, so it remains to bound the former term. However, \cref{lem:LipshitzConstant} yields that $\nabla L_{\hat{\mu}_{n}}$ is $\vert \spt(\hat{\mu}_{n}) \vert^{2}_{\infty}$-Lipschitz continuous and hence 
\begin{align*}
   \Vert \nabla L_{\hat{\mu}_{n}}(A^{\top}z_{\hat{\mu}_{n}, \varepsilon}) - \nabla L_{\hat{\mu}_{n}}(A^{\top}z_{\hat{\mu}_{n}}) \Vert &\leq  \vert \spt(\hat{\mu}_{n}) \vert_{\infty}^2  \Vert A^{\top}( z_{\hat{\mu}_{n}, \varepsilon} - z_{\hat{\mu}_{n}}) \Vert\\
   &\leq \vert \spt({\mu}) \vert_{\infty}^2 \Vert A \Vert  \Vert z_{\hat{\mu}_{n}, \varepsilon} - z_{\hat{\mu}_{n}} \Vert.
\end{align*}
In turn, using the $\beta/\alpha$-strong convexity of $\phi_{\hat{\mu}_{n}}$, we use  \eqref{eqn:strongConvexVer2} and the fact that $z_{\hat{\mu}_{n}}$ is a minimizer of $\phi_{\hat{\mu}_{n}}$ (hence $0 \in \partial \phi_{\hat{\mu}_{n}}(z_{\hat{\mu}_{n}})$) to obtain
\begin{equation*}
    \frac{\beta}{2\alpha} \Vert z_{\hat{\mu}_{n}, \varepsilon} - z_{\hat{\mu}_{n}} \Vert^{2} \leq \phi_{\hat{\mu}_{n}}( z_{\hat{\mu}_{n}, \varepsilon} )- \phi_{\hat{\mu}_{n}}(z_{\hat{\mu}_{n}})  \leq \varepsilon,
\end{equation*}
where the second inequality follows by the definition of $z_{\hat{\mu}_{n}, \varepsilon}$. Hence,
\begin{equation*}
    \Vert \nabla L_{\hat{\mu}_{n}}(A^{\top}z_{\hat{\mu}_{n}, \varepsilon}) - \nabla L_{\hat{\mu}_{n}}(A^{\top}z_{\hat{\mu}_{n}}) \Vert \leq \vert \spt({\mu}) \vert_{\infty}^2 \Vert A \Vert \sqrt{ \frac{2\alpha}{\beta} \varepsilon}.
\end{equation*}
Taking the expectation in \eqref{eqn:substitutionApprox} then gives the result.
\end{proof}

\section{Maximum Entropy on the Mean as Empirical Risk Minimization}\label{sec:ER}

In the sequel, we show that the dual MEM problem can be construed as a risk minimization problem. This enables us to solve the  dual  empirical MEM problem using stochastic gradient-based methods and to leverage standard results on the statistical properties of empirical risk minimization problems.    
\subsection{An Empirical Risk Minimization Perspective}
The classical study of  empirical risk minimization consists of analyzing the effects of estimating $\theta_{\mu}\in\argmin_{\theta\in\Theta} \mathbb E_{\mu}[\ell(X,\theta)]$ by $\theta_{\hat \mu_n}\in\argmin_{\theta\in\Theta} \mathbb E_{\hat \mu_n}[\ell(X,\theta)]$ where $\ell$ is a loss function. In words, we wish to perform inference on the minimizer of a function of interest, but only observe surrogates of this function which are obtained via sample-based approximations. Standard results regarding consistency, rates of convergence, and limit distributions  are available in the reference textbooks \cite{kosorok2008introduction,
vaart1996empirical,van2000asymptotic} under the broader study of $M$-estimation.  

Remarkably, for a given prior $\mu\in\mathcal P(\mathcal X)$, the dual problem \eqref{eqn:MEMDual} can be written in the standard form described above. Precisely, for each $z\in\mathbb R^d$, 
\[
\begin{aligned}
    \alpha g^{\ast}(-z/\alpha)-\langle b,z\rangle + \log \int \exp\langle A^{\top}z,\cdot\rangle d\mu &= \log \int \exp\left(\alpha g^{\ast}(-z/\alpha)-\langle b,z\rangle+\langle A^{\top}z,\cdot\rangle \right) d\mu.
\end{aligned}
\]
 As the logarithm is strictly monotonically increasing, the optimization problem
\[
\min_{z\in \mathbb R^d}\int \exp\left(\alpha g^{\ast}(-z/\alpha)-\langle b,z\rangle+\langle A^{\top}z,\cdot\rangle \right) d\mu
\]
has the exact same solution(s) as the MEM dual problem \eqref{eqn:MEMDual}.
It follows that the population-level dual problem can be written  equivalently as 
\[z_{\mu}\in\argmin_{z\in\mathbb R^d}\mathbb E_{\mu}\left[ \ell(X,z)\right]\text{ where }\ell(x,z)=\exp\left(\alpha g^{\ast}(-z/\alpha)-\langle b,z\rangle+\langle A^{\top}z,x\rangle \right),
\] 
whereas the empirical dual problem is equivalent to 
\[
z_{\hat \mu_n}\in\argmin_{z\in \mathbb R^d}\mathbb E_{\hat \mu_n}\left[ \ell(X,z)\right]=\argmin_{z\in \mathbb R^d}\frac 1n\sum_{i=1}^n\exp\left(\alpha g^{\ast}(-z/\alpha)-\langle b,z\rangle+\langle A^{\top}z,X_i\rangle \right),
\]
 from which one can recover $x_{\mu}$ and $x_{\hat{\mu}_{n}}$ from their respective primal-dual recovery formulae, see \eqref{eqn:primalDualRecoveryFormula}. This formulation coincides with the standard empirical risk minimization setting described above and, importantly, preserves strong convexity of the original objective. 

\begin{proposition}[Strong convexity under composition with $\exp$] \label{lem:strongConvexityExponential}
    Suppose that $h:\mathbb R^d\to \overline{\mathbb R}$ is proper, lsc, and $\gamma$-strongly convex. Then, $z\in\mathbb R^d\mapsto e^{h(z)}$ is $\gamma'$-strongly convex for $\gamma'=\gamma e^{\inf_{\mathbb R^d}h}>0$ and $\partial(e^h)(z)=\{ v e^{h(z)}:v\in \partial h(z)\}$. 
\end{proposition}
\begin{proof}
    We record that $\exp$ is  proper, lsc, convex, increasing, and finite valued.  Hence, with the convention $e^{+\infty}=+\infty$, for any $z \in \dom(h) = \dom(e^h)$, we may apply the chain rule of \cite[Corollary 4]{burke2021study} to obtain that
    \begin{equation*}
        \partial(e^{h})(z) = e^{h(z)} \partial h(z). 
    \end{equation*}
    The same reference also asserts that $e^h$ is convex so that this subdifferential is well-defined. 
    Now, let $z_{1},z_{2} \in \dom(e^h)$, and take $v_{2} \in \partial h(z_{2})$. Then,
    \[
    \begin{aligned}
        e^{h(z_1)}-e^{h(z_2)}\geq e^{h(z_2)}\left(h(z_1)-h(z_2)\right)&\geq e^{h(z_2)}\left(\langle v_2,z_1-z_2\rangle+\frac{\gamma}{2}\|z_1-z_2\|^2\right),
    \end{aligned} 
    \]
    where the first inequality uses the (sub)gradient inequality \eqref{eqn:subgradientDefinition} for the convex function $e^{x}$, and the latter uses the strong convexity of $h$, namely  \eqref{eqn:strongConvexVer2}.
    We conclude by using the chain rule at $v_{2}$, observing $\{v_2e^{h(z_2)}:v_2\in\partial h(z_2)\}=\partial(e^h)(z_2)$ and $\gamma e^{h(z_2)}\geq  \gamma e^{\inf_{\mathbb R^d}h}$, whereby
    \begin{equation*}
        e^{h(z_1)}-e^{h(z_2)} \geq\langle u_2,z_1-z_2\rangle+\frac{\gamma'}{2}\|z_1-z_2\|^2,\qquad  \forall u_{2} \in \partial(e^h)(z_2),
    \end{equation*}
    which, by  \cref{prop:strongConvexity}, implies strong convexity of $e^h$ with modulus $\gamma'$.
\end{proof}

\noindent
The implications of connecting MEM problems to empirical risk minimization are twofold: 

First, we may leverage the existing literature on $M$-estimation to establish statistical properties of the empirical MEM problem. As noted above, there is a vast literature on analyzing $M$-estimators under different assumptions on the underlying problem. A recent example is given in the work of \cite{brunel2025asymptotics} which establishes limit laws for $M$-estimators in a general setting. %
While rates of convergence (in expectation) for $M$-estimators are accounted for in the literature (cf., e.g., \cite{birge1993rates}), these results are stated in greater generality than is required here,  and hence require the verification of a number of technical conditions.  
By contrast, the rates obtained in
\cref{sec:parametricRates} follow from first principles of  convex analysis and probability theory, enabling a simpler, and more self-contained exposition. 

Second, the alternative formulation of the dual MEM problem is amenable to optimization via  stochastic gradient methods. As such, even in the case where a population-level prior, $\mu$, is known, but does not admit a tractable $\log$ moment generating function, we may approximate the gradient via a sample-based approximation and derive meaningful convergence results as expounded next.    

\subsection{Clipped SGD and convergence}

As noted above, the dual MEM problem with prior $\mu \in \mathcal{P}(\mathcal{X})$ can be reformulated as
\begin{equation}
\min_{z\in \R^{d}}\mathbb E_{\mu}\left[ \exp\left(\alpha g^{\ast}(-z/\alpha)-\langle b,z\rangle+\langle A^{\top}z,X\rangle \right)\right] =: \min_{z\in \R^{d}} \E_{\mu}\left[ \ell(X,z)\right]. \tag{$D_{\mu}'$} \label{eqn:reformulated}
\end{equation} 
As such, the corresponding empirical MEM dual problem is
\begin{equation}
\min_{z\in \R^{d}} \frac{1}{n} \sum_{i=1}^{n} \exp\left(\alpha g^{\ast}(-z/\alpha)-\langle b,z\rangle+\langle A^{\top}z,X_{i}\rangle \right) = \min_{z \in \R^{d}} \frac{1}{n} \sum_{i=1}^{n}  \ell(X_{i},z). \tag{$D_{\hat{\mu}_{n}}'$} \label{eqn:reformulated_empirical}
\end{equation} 
Both of these problems are thus amenable to optimization via stochastic gradient methods. The main distinction between these problems is that optimizing $(D_{\mu}')$ using such methods requires a mechanism to draw new i.i.d. samples from the population measure $\mu$ at each iteration to approximate the population gradient via mini-batching. On the other hand, if we only have a fixed finite number of samples from $\mu$ (e.g., a dataset) we may also consider solving $(D_{\hat \mu_n}')$ using stochastic methods by taking random subsets of the finite collection of samples to form mini-batched gradients. This is particularly useful in the case that the number of samples is so large that storing the entire dataset in memory is prohibitively expensive or that summing over all of the samples is costly.  The results of the previous section establish that $x_{\hat \mu_n}$ can serve as a good proxy for $x_{\mu}$ (on average) if the samples used to form $\hat \mu_n$ are i.i.d. from $\mu$ and $n$ is sufficiently large.         

This perspective is thus advantageous in cases where excessively large amounts of data are available, as it enables scalability via mini-batching gradients and the ability to stream the data in place of storing it.       
On the other hand, dropping the outer logarithm introduces numerical instability as the gradients involve exponential terms. 

For instance, if $g = g^{*} = \frac{1}{2} \Vert \cdot \Vert^{2}$,
\begin{equation*}
   \nabla \E_{\mu}\left[ \ell(X,z)\right] =  \E_{\mu} \left[ \left(\frac{1}{\alpha}z -  b  + AX\right)\exp\left( \frac{1}{2 \alpha}\Vert z \Vert^{2} + \langle AX-b, z\rangle \right) \right],
\end{equation*}
which can lead to numerical overflow if $\|z\|$ is even moderately large. 

To account for this, we propose to use a stochastic gradient method with clipped gradients as discussed in \cite{mai2021stability}. The crux of this approach consists of normalizing the stochastic gradient before taking a step to avoid taking  excessively large steps.  
We note, however, that the method analyzed in that work appears to tacitly assume that the function being minimized has full domain\footnote{In effect, the clipped SGD  update \cite[Equation 2]{mai2021stability} does not enforce that iterates lie in the domain of the objective. Further, their assumption A4 requires the expectation of the squared norm of a subgradient to be uniformly bounded. As the subdifferential of a convex function $f$ is non-empty and bounded only at points in $\inte(\dom(f))$ such conditions generally cannot hold at boundary points of $\dom(f)$.} and that subgradients are chosen in a measurable manner. To account for functions with restricted domain, we introduce a projection step into the clipped SGD method and work under appropriate conditions guaranteeing that all iterates of the algorithm are in $\inte(\dom(g^{\ast}(-\cdot/\alpha)))$ which guarantees that subgradients at each iterate are bounded.  We provide a sufficient condition for this assumption to hold in \eqref{eqn:ass_domain} ahead.

\subsubsection{Projected Clipped SGD}

We now state and explain the proposed algorithm. As aforementioned, there are two situations to consider, namely 
\begin{tcolorbox}[colback=gray!5!white,colframe=gray!75!black]
\textrm{(S1) we have access to a mechanism which generates fresh collections of i.i.d. samples of a prescribed size from the population $\mu$ which are conditionally independent from the previously generated collections. }
\\

\textrm{(S2) we have a fixed dataset $\mathcal D=(x_i)_{i=1}^n$ of samples from $\mu$ and have access to a mechanism which generates fresh collections of i.i.d. samples from the corresponding empirical measure $\hat\mu_n$ which are conditionally independent from the previous collections and from the original dataset.}
\end{tcolorbox}

\begin{remark}[Generating batches under (S2)]
    Given a fixed dataset $\mathcal D$, a standard approach for generating batches of samples  satisfying the conditions of (S2) consists of simply resampling $m$ samples from $\mathcal D$ uniformly at random with replacement. That is, we generate an i.i.d. collection of indices $\{\sigma_1,\dots, \sigma_m\}$ from the uniform distribution on $\{1,\dots, n\}$ which is independent from previous draws and the original dataset, and take the mini-batch to be $\{x_{\sigma_1},\dots,x_{\sigma_m}\}$. With this, $\{x_{\sigma_1},\dots,x_{\sigma_m}\}$ is conditionally i.i.d. from $\hat \mu_n$ and independent  from the previously drawn mini-batches conditionally on $\mathcal D$ so that \textrm{(S2)} holds.     
\end{remark}

Notably, when applying \cref{alg:clippedSGD}  ahead using mini-batched gradients with collections generated according to \textrm{(S1)} the aim is to approximately solve \eqref{eqn:reformulated}, i.e., the dual MEM problem with prior $\bar \mu=\mu$. If the collections are generated according to \textrm{(S2)}, we approximately solve  \eqref{eqn:reformulated_empirical}, i.e.,    
the dual MEM problem with prior $\bar \mu=\hat \mu_n$. Here and in the sequel, we use $\bar \mu$ to denote the population prior $\mu$  under (S1) or the empirical prior $\hat \mu_n$ under (S2) to simplify notation. To clearly state and motivate the projected clipped SGD method, we make the following assumption
\begin{tcolorbox} 
\vspace{-1em}
   \begin{equation} 
   \tag{A2}
   \exists r>0:\; B_{ r} \subset \inte(\dom(\mathbb E_{\bar \mu}[\ell(X, \cdot)]))=\inte(\dom(g^{\ast}(-(\cdot)/\alpha)))\text{ and }z_{\bar \mu} \in B_{ r}.
    \label{eqn:ass_domain}
    \end{equation}
\end{tcolorbox} 

As noted in \cref{thm:stabilityDual} \ref{eqn:boundedNormOfSolutions}, 
$
\Vert z_{\bar \mu} \Vert \leq \frac{\alpha}{\beta} \big(\min_{v \in  \partial g^{*}(0)}\Vert v \Vert +  \Vert b \Vert  + \Vert A \Vert \vert \mathcal{X} \vert_{\infty} \big)=: r,
$
where we underscore that $\partial g^{\ast}(0)$ is non-empty and bounded as $0\in\inte(\dom(g^{\ast}))$ by assumption, see \cite[Theorem 23.4]{rockafellar1970convex}. It follows that the above assumption holds if $\inte(B_{\bar r/\alpha})\subset \inte(\dom(g^{\ast}(\cdot)))$ for some $r<\bar r$ which is equivalent to the following condition on $g$ itself. 
\begin{tcolorbox}
  The fidelity term $g$ satisfies
    $    \liminf_{\|z\|\to \infty}\frac{g(z)}{\|z\|}\geq \frac{\bar r}{\alpha}$ for some $\bar r>r$. %
\end{tcolorbox} 
Thus, if $\lim_{\|z\|\to \infty}\frac{g(z)}{\|z\|}=\infty$, i.e., $g$ is supercoercive, \eqref{eqn:ass_domain} is satisfied. For instance, the prototypical fidelity term $g=\frac{1}{2}\|\cdot\|^2$ is supercoercive.  
We now state and prove this equivalence. 

\begin{proposition} \label{lem:technicalBall}
    Suppose that $h:\mathbb R^d\to\overline{\mathbb R}$ is  proper, lsc, and convex. Then, $\mathrm{int}(\dom(h^{\ast}))$ contains the open ball of radius $\bar r>0$ centered at $0$ if and only if $\liminf_{\|z\|\to \infty}\frac{h(z)}{\|z\|}\geq \bar r.$
\end{proposition}
\begin{proof}
By 
\cite[Corollary 13.3.4 (c)]{rockafellar1970convex}, $z\in\mathrm{int}\left(\dom\left(h^{\ast}\right)\right)$ if and only if the recession function of $h-\langle \cdot,z\rangle$ is strictly greater than $0$ on $\mathbb R^d\backslash\{0\}$.  
Following  \cite[Theorem 3.26 (a)]{rockafellar1998variational}, this latter condition is equivalent to the property that  $h-\langle \cdot,z\rangle$ is bounded below on bounded sets and \[
\liminf_{\|w\|\to \infty}\frac{h(w)-\langle w,z\rangle}{\|w\|}>0.\]

Now, suppose that $\liminf_{\|w\|\to \infty}\frac{h(w)}{\|w\|}\geq \bar r>0$, then, for any $z\in\mathbb R^d$ with $\|z\|<\bar r$,  
\[
\liminf_{\|w\|\to \infty}\frac{h(w)-\langle w,z\rangle}{\|w\|}\geq \bar r-\limsup_{\|w\|\to \infty}\frac{\langle w,z\rangle}{\|w\|}>\bar r-\bar r=0.
\]
Evidently $h-\langle \cdot,z\rangle$ is bounded below on any bounded set 
so that $z\in\mathrm{int}(\dom(h^{\ast}))$ by the previous deliberations. Thus the open ball of radius $\bar r$ is contained in $\mathrm{int}(\dom(h^{\ast}))$.

On the other hand, if the open ball of radius $\bar r$ is contained in $\mathrm{int}(\dom(h^{\ast}))$, we have from \cite[Theorems 3.26 and 11.5]{rockafellar1998variational} that, for any $r'<\bar r$,
\[
\liminf_{\|w\|\to \infty}\frac{h(w)}{\|w\|}=\inf_{\|w\|=1}\sup_{v\in\dom(h^{\ast})}\langle w,v\rangle\geq r'\inf_{\|w\|=1}\|w\|=r', 
\]
since, for any $w\in\mathbb R^d$, $r'\frac{w}{\|w\|}\in\dom(h^{\ast})$ and so $\sup_{v\in\dom(h^{\ast})}\langle w,v\rangle\geq r'\|w\|$. Taking the limit $r'\to \bar r$ above proves the claimed result.
\end{proof}

With this explanation of the projection step in hand, we now state the algorithm:
 
\begin{algorithm}[!htb]
\DontPrintSemicolon
\caption{Projected Clipped SGD}\label{alg:clippedSGD}
\KwData{Fix a radius $r$ satisfying  \eqref{eqn:ass_domain}. Initialize $z_{0}\in B_{r}$, a stepsize schedule $(\beta_{k})_{k\in\mathbb N}$, a sequence of  batch sizes $(m_{k})_{k\in\mathbb N}$, and  a clipping parameter $\gamma>0$.}
\For{$k = 0,1,2,3 \ldots$}{
    Draw a collection of samples $x_{k_{1}}, \ldots , x_{k_{m_k}}$ according to \textrm{(S1)} or \textrm{(S2)} \;
    $g_{k} \gets \frac{1}{m_{k}} \sum_{i=1}^{m_{k}}\ell'( x_{k_{i}},z_{k})$, for a measurable selection  $\ell'( x_{k_{i}},z_{k}) \in \partial_z \ell( x_{k_{i}},z_{k}) $\;
    $z_{k+1} \gets P_{B_{ r}}\left(z_{k} - \beta_{k} \min \{ 1, \frac{\gamma}{\Vert g_{k} \Vert }\} g_{k}\right) $ \:
}
\end{algorithm}

The main feature of this algorithm is the final line, known as the clipping step, where instead of taking the usual gradient step $-\beta_{k} g_{k} $, the update is enforced to have norm $\beta_k\min\{\|g_k\|,\gamma\}$.  This prevents instability if the gradient has a large magnitude. Our method also includes a projection onto $B_{r}$, denoted $P_{B_{r}},
$ which ensures that the iterates remain in the interior of the domain of the objective; this step is not present in the original algorithm. %
As presented, the algorithm does not include a stopping criterion and, in our examples, is run up to a pre-specified allowance of iterations.

The remainder of this section is devoted to proving the almost sure convergence of \cref{alg:clippedSGD} for the empirical risk formulation of the MEM dual problem. We do so by adapting the proof techniques of \cite{mai2021stability} for the projected version of the CSGD algorithm. As noted previously, the results of that work only appear to apply to objectives with full domain. We begin by compiling some useful properties of $\ell$.

\begin{lemma}[Properties of $\ell$] \label{lem:clipped_Sgd_properties}
    Let $g,r$ satisfy assumptions \eqref{eq:ass} and \eqref{eqn:ass_domain} and, for $x\in \mathcal X$, let 
    \begin{equation}
        \ell'(x,z) = \left(v(z)-b +Ax\right)\exp(\alpha g^{\ast}(-z/\alpha)-\langle b,z\rangle +\langle A^{\top}z,x\rangle), \label{eqn:decompositionSubdiff}
        \end{equation}
     where $v(z)\in -\partial g^{\ast}(-z/\alpha)$ is independent of $x$ and $z\in B_r\mapsto v(z)$ is measurable. Then $\ell'(x,z)$ is a measurable selection of subgradient of $\ell(x,\cdot)$ at $z\in B_r$ and, if $X\sim\bar \mu$, the following hold:
    \item[(P1)] (Unbiased gradient estimates). $\E[\ell'(X,z)] \in \partial \E [\ell(X,z)] $.
        \item[(P2)] \label{P2} (Quadratic growth). Let $\{z_{\bar \mu}\} = \argmin_{\mathbb R^d} \E[\ell(X,\cdot)]$. There exists $a>0$ for which
        \begin{equation*}
           \E [\ell(X,z) ]- \E [\ell(X,z_{\bar \mu})] \geq a  \Vert z - z_{\bar \mu} \Vert^{2} \text{ for all $z \in B_{r}$}
        \end{equation*}
        \item[(P3)] \label{P3} (Finite variance). There exists $\sigma >0$ such that
        \begin{equation*}
            \E[\Vert \ell'(X,z) - \E[\ell'(X,z)] \Vert^{2} ] \leq \sigma^{2} \text{ for all } z \in B_{r}.
        \end{equation*}
\end{lemma}

\begin{proof}
We begin by proving that $\ell'(x,z)$ corresponds to a measurable selection of subgradient. From \cref{lem:strongConvexityExponential}, for each $x\in\mathcal X$ and $z \in B_{r} \subset \inte \left(\dom \left(\ell(x,\cdot)\right)\right)$,
    \begin{align*}
        \partial_z \ell(x,z) &= \left(-\partial g^{\ast}(-z/\alpha)-b +Ax\right)\exp(\alpha g^{\ast}(-z/\alpha)-\langle b,z\rangle +\langle A^{\top}z,x\rangle).
    \end{align*}
    Hence, $\ell'(x,z)$ from \eqref{eqn:decompositionSubdiff} corresponds to a certain choice of subgradient of $\ell(x,\cdot)$. For fixed $z, v$, $\ell'(\cdot,z)$ is a continuous function on the compact set $\mathcal{X}$ and is, in particular, bounded and measurable and, in fact, integrable.
    
    \noindent \smallskip (P1)  
     Given that $\ell'(x,z)\in\partial_z\ell(x,z)$, 
    \begin{equation*}
        \ell(x,z') - \ell(x,z) \geq \langle \ell'(x,z) , z'-z \rangle,\text{ for each }z' \in \R^{d}.
    \end{equation*}
    Taking expectations on both sides of the inequality, we obtain that
    \begin{equation*}
        \E[\ell(X,z')] - \E[\ell(X,z)] \geq \langle \E [\ell'(X,z)] , z'-z \rangle, \text{ for each }z'\in\mathbb R^d. 
    \end{equation*}
    This demonstrates that $\E[\ell'(X,z)]$ is a subgradient of $\E[\ell(X,\cdot)]$ at $z$ as desired. 
    
    \noindent \smallskip (P2)  Fix $z \in B_{r}$, and let $z_{\bar \mu} \in B_{r}$ be the unique minimizer of $\E[\ell(X,z)]$. Then
    \begin{align*}
        \E[\ell(X,z)] - \E[\ell(X,z_{\bar \mu})] &= \E[ \exp h_{X}(z)] - \mathbb E[\exp h_{X}(z_{\bar \mu})],
    \end{align*}
    where $h_{x}(z) := \alpha g^{\ast}(-z/\alpha)-\langle b,z\rangle+\langle A^{\top}z,x\rangle$ for $x\in \spt(\bar \mu)$. This function is proper, lsc, and $\beta/\alpha$-strongly convex for all $x$ (cf. \eqref{eq:ass}), so that, for 
    \[
    \gamma_x:=\frac{\beta}{\alpha}e^{\inf_{\mathbb R^d}h_x},
    \]
    the function   $\exp{h_{x}}$ is $\gamma_x$-strongly convex in $z$ by \cref{lem:strongConvexityExponential}.
    Thus, for each $\lambda\in[0,1]$ and any $z,z'\in B_r$, 
    \[
       \exp{h_{x}}\left( \lambda z +(1-\lambda)z'\right) \leq 
       \lambda \exp{h_{x}}( z) + (1-\lambda)\exp{h_x}(z')-\frac{\gamma_x}{2} \lambda(1-\lambda)\|z-z'\|^2.  
    \]
    Taking the expectation on both sides of the inequality shows that $\mathbb E[\ell(X,\cdot)]$ is strongly convex with modulus $\mathbb E[\gamma_X]\geq \frac{\beta}{\alpha}e^{\inf_{x\in \spt(\bar \mu)}\inf_{\mathbb R^d}h_x}>0$. This final lower bound is due to the fact that, setting $\bar z$ as the unique global minimizer of $g^{\ast}(-\cdot/\alpha)$,  
    \[ 
    \begin{aligned}
        h_x(z)&\geq \alpha g^{\ast}(-\bar z/\alpha)+\frac{\beta}{2\alpha}\|z-\bar z\|^2 -\|z\|\|b\|-\|A\||\mathcal X|_{\infty} \|z\|
        \\
        &\geq \alpha g^{\ast}(-\bar z/\alpha)+\frac{\beta}{2\alpha}(\|z\|^2-2\|z\|\|\bar z\|+\|\bar z\|^2) -\|z\|\|b\|-\|A\||\mathcal X|_{\infty} \|z\|, 
    \end{aligned} 
    \]
    where the final lower bound is a convex parabola in $\|z\|$ which is independent of $x$ and hence admits a finite minimum.
    By \eqref{eqn:strongConvexVer2} and using the fact that $z_{\bar \mu}$ is the minimizer of $\mathbb E[\ell(X,\cdot)]$,
    \[
        \mathbb E[\ell(X,z)]- \mathbb E[\ell(X,z_{\bar \mu})]\geq \frac{\mathbb E[\gamma_X]}{2}\|z-z_{\bar \mu}\|^2,
    \]
    proving the claim.

    \noindent \smallskip (P3) Given $z \in B_{r}$, we have from \eqref{eqn:decompositionSubdiff} that 
    \[
    \ell'(x,z) = \left(v(z)-b +Ax\right)\exp(\alpha g^{\ast}(-z/\alpha)-\langle b,z\rangle +\langle A^{\top}z,x\rangle)
    \]
    for some 
     $v(z) \in - \partial g^{*}(-z/\alpha)$. Thus, 
    \begin{align*}
        &\ell'(x,z) - \E [\ell'(X,z)] \\
        &= (v(z)\mspace{-2mu}-\mspace{-2mu}b) \left( \exp(\alpha g^{\ast}(-z/\alpha)\mspace{-2mu}-\mspace{-2mu}\langle b,z\rangle\mspace{-2mu} +\mspace{-2mu}\langle A^{\top}z,x\rangle)\mspace{-2mu} -\mspace{-2mu} \E \left[ \exp(\alpha g^{\ast}(-z/\alpha)\mspace{-2mu}-\mspace{-2mu}\langle b,z\rangle\mspace{-2mu} +\mspace{-2mu}\langle A^{\top}z,X\rangle) \right] \right) \\
        &+ A \left( x\exp(\alpha g^{\ast}(-z/\alpha)-\langle b,z\rangle +\langle A^{\top}z,x\rangle)- \E \left[ X\exp(\alpha g^{\ast}(-z/\alpha)-\langle b,z\rangle +\langle A^{\top}z,X\rangle) \right] \right). 
    \end{align*}
    Taking the squared norm on both sides of the equality followed by the expectation, we have
    \begin{align*}
        \mathbb E\left[\Vert \ell'(X,z) - \E [\ell'(X,z)] \Vert^{2}\right] &\leq 2 \|v(z)-b\|^2\var\left(\exp(\alpha g^{\ast}(-z/\alpha)-\langle b,z\rangle+\langle A^{\top}z,X\rangle)\right) \\
        & + 2\|A\|^2\tr\left(\cov\left(X\exp(\alpha g^{\ast}(-z/\alpha)-\langle b,z\rangle+\langle A^{\top}z,X\rangle)\right)\right).
    \end{align*}
    Applying standard bounds for the variance and trace of the covariance, 
    \[
    \begin{aligned}
           \mathbb E\left[\Vert \ell'(X,z) - \E [\ell'(X,z)] \Vert^{2}\right] &\leq 2 \|v(z)-b\|^2\mathbb E\left[\exp(\alpha g^{\ast}(-z/\alpha)-\langle b,z\rangle+\langle A^{\top}z,X\rangle)^2\right] \\
        & + 2\|A\|^2\mathbb E\left[\left\|X\exp(\alpha g^{\ast}(-z/\alpha)-\langle b,z\rangle+\langle A^{\top}z,X\rangle)\right\|^2\right]
    \end{aligned} 
    \]
    and so 
    \[
    \begin{aligned}&
            \mathbb E\left[\Vert \ell'(X,z) - \E [\ell'(X,z)] \Vert^{2}\right] \\&\leq 2 \left(\sup_{v\in -\partial g^{\ast}(-z/\alpha)}\|v-b\|^2 +\|A\|^2|\mathcal X|_{\infty}^2\right)\exp\left(2\left(\sup_{z\in B_r}\alpha g^{\ast}(-z/\alpha)+\|b\|r+|\mathcal X|_{\infty}\|A\|r \right)\right). %
    \end{aligned} 
    \]
Note that since 
 $B_{\frac{r}{\alpha}}\subset \inte(\dom g^{\ast})$ by assumption, \cite[Theorem 24.7]{rockafellar1970convex} asserts that $\partial g^{\ast}(B_{\frac{r}{\alpha}}) = \cup_{z \in B_{\frac{r}{\alpha}}} \partial g^{*}(z)$ is bounded. Therefore $\sup_{z\in B_r}\sup_{v\in -\partial g^{\ast}(-z/\alpha)}\|v-b\|^2<\infty$ and, moreover, it holds that
 $\sup_{z \in B_{r}} g^{\ast}(-z/\alpha) < + \infty$ as any convex function is continuous on the interior of its domain \cite[Theorem 10.1]{rockafellar1970convex} and $B_{r}$ is compact. It follows that $\mathbb E\left[\Vert \ell'(X,z) - \E [\ell'(X,z)] \Vert^{2}\right]\leq \sigma^2$ for a finite constant $\sigma^2>0$. 
\end{proof}

 We  note that a measurable selection of subgradient always exists under (A2) (cf. e.g., \cite[Corollary 14.6 and Theorem 14.56]{rockafellar1998variational}) and, in cases where $g^{\ast}$ is differentiable, the subdifferential is a singleton so that no choice is required. In any case subgradient selection rule which can be expressed in terms of compositions of Borel measurable operations will be measurable recalling \cref{rmk:measurableSelection}. Furthermore by taking measurable subgradients it is guaranteed that the iterates of the \cref{alg:clippedSGD} are also measurable.

The properties proved in \cref{lem:clipped_Sgd_properties} are paramount in showing the convergence of  \cref{alg:clippedSGD} as stated next.

\begin{theorem}[Convergence of \cref{alg:clippedSGD}]
\label{thm:convSGD} 
    Let $g$ and $r$ satisfy assumptions \ref{eq:ass} and \ref{eqn:ass_domain} and 
    suppose that the stepsizes $\beta_{k} \geq 0$ satisfy $\sum_{k=0}^{\infty}\beta_{k}= +\infty$, $\sum_{k=0}^{\infty} \beta_k^2 <\infty$ and $\sum_{k=0}^{\infty} \beta_{k}/m_{k} < + \infty$ Then, any sequence of iterates $z_{k}$ generated by \cref{alg:clippedSGD} as applied to $f(x,z) = \ell(x,z)$ converges almost surely to $\{z_{\bar \mu}\} = \argmin_{\mathbb R^d} \E_{X\sim\bar\mu}[\ell(X,\cdot)].$
\end{theorem}

\begin{proof} The proof of this result follows similar lines to that of \cite[Theorem 1]{mai2021stability}, but accounts for extended real-valued functions. For notational brevity, we introduce the quantity
$
    \rho_{k} = \min \left\{ 1, \frac{\gamma}{\Vert g_{k} \Vert} \right\}.
$
Then, as $P_{B_{r}}$ is non-expansive,
\begin{align}
    \|z_{k+1}-z_{\bar \mu}\|^2 &= \Vert P_{B_{r}}(z_k-\beta_k\rho_kg_k) - z_{\bar \mu} \Vert \nonumber  \\
    &\leq \|z_k-\beta_k\rho_kg_k-z_{\bar \mu} \|^2 \nonumber\\
    &= \|z_{k}-z_{\bar \mu}\|^2-2\beta_k\rho_k\langle g_k-H_k+H_k,z_k-z_{\bar \mu}\rangle+\beta_k^2\rho_k^2\|g_k\|^2.
\label{eq:iteratesClippedSGD}
\end{align}
Here, $z_k\in B_r$, by the projection step combined with the assumption $z_{0} \in B_{r}$, and $H_k=\mathbb E[g_k|\mathcal F_k]$, where $\mathcal F_k$ is the $\sigma$-algebra generated by the first $k$ batches\footnote{i.e. $\mathcal{F}_{k} =\sigma( x_{0_{1}}, x_{0_2}, \ldots x_{0_{m_{0}}} , \ldots , x_{{k-1}_{1}}, \ldots, x_{{k-1}_{m_{k-1}}} )$ is the smallest $\sigma$-algebra containing the first $k$ batches, where $x_{i_{j}}$ as defined in \cref{alg:clippedSGD} is the $j$-th mini-batch sample at iteration $i$.} under (S1) and, under (S2), $\mathcal F_k$ also includes the dataset $\mathcal D$ as a generator. From \cref{lem:clipped_Sgd_properties} (P2), 
\begin{equation}
\label{eq:P2equation}
   a\|z_k-z_{\bar \mu}\|^2\leq \E [\ell(X,z_{k}) ]- \E [\ell(X,z_{\bar \mu})]\leq \langle H_k,z_k-z_{\bar \mu}\rangle,
\end{equation}
where the second inequality follows from the subgradient inequality for convex functions. Indeed, noting that conditionally on $\mathcal F_k$, $z_k$ is fixed and the mini-batch used to obtain $g_k$ is independent of $\mathcal F_k$, $\mathbb E[g_k|\mathcal F_k]=\mathbb E_{X\sim\bar \mu}[\ell'(X,z_k)]\in\partial(\mathbb E[\ell(X,\cdot)])(z_k)$ by \cref{lem:clipped_Sgd_properties} (P1). Applying the Cauchy-Schwarz and $\varepsilon$-Young inequalities with $\varepsilon=a$, we obtain that 
\begin{align*}
\left|\langle g_k-H_k,z_k-z_{\bar \mu} \rangle \right| &\leq \Vert g_{k} - H_{k} \Vert \cdot \Vert z_{k} - z_{\bar\mu} \Vert\\
&\leq \frac{1}{2a}\|g_k-H_k\|^2+\frac{a}{2}\|z_k-z_{\bar \mu}\|^2 \\
&\leq \frac{1}{2a}\|g_k-H_k\|^2+\frac{1}{2}\langle H_k,z_k-z_{\bar \mu}\rangle, 
\end{align*}
where we have substituted \eqref{eq:P2equation} in the final step. Applying this bound in \eqref{eq:iteratesClippedSGD} yields that  
\[
    \|z_{k+1}-z_{\bar \mu}\|^2\leq \|z_{k}-z_{\bar \mu}\|^2-\beta_k\rho_k\langle H_k,z_k-z_{\bar \mu}\rangle+\frac{\beta_k\rho_k}{a}\|g_k-H_k\|^2+\beta_k^2\rho_k^2\|g_k\|^2
\]
and so, taking the conditional expectation and using the fact that $\rho_k\|g_k\|\leq \gamma$ and $\rho_k\leq 1$,  
\begin{equation}
\label{eq:RSInequality}
\begin{aligned}
   &\mathbb E\Big[\|z_{k+1}-z_{\bar \mu}\|^2\vert \mathcal F_k\Big] 
    \\
    &\leq \|z_{k}-z_{\bar \mu}\|^2-\beta_k \langle H_k,z_k-z_{\bar \mu}\rangle \mathbb E[\rho_k|\mathcal F_k]+\frac{\beta_k}{a}\mathbb E\big[\|g_k-H_k\|^2|\mathcal F_k\big] + \gamma^2\beta_k^2.
\end{aligned}
\end{equation}

We now invoke the classical almost supermartingale convergence theorem of \cite[Theorem 1]{ROBBINS1971233}, which states that if $X_k,Y_k,U_k,V_k,W_k$ are non-negative $\mathcal F_k$-measurable random variables satisfying
\[
\mathbb E[X_{k+1}|\mathcal F_k] \leq X_k(1+U_k)+V_k -W_k, 
\]
for each $k\in\mathbb N$, then $\lim_{k\to \infty}X_k$ exists and is almost surely finite and $\sum_{k=1}^{\infty} W_k<\infty$ almost surely on the event that $\sum_{k=1}^{\infty} U_k$ and $\sum_{k=1}^{\infty} V_k$ are finite. Evidently \eqref{eq:RSInequality} is of this form with $X_k=\|z_k-z_{\bar \mu}\|^2, U_k=0,V_k=\frac{\beta_k}{a}\mathbb E\big[\|g_k-H_k\|^2|\mathcal F_k\big] + \gamma^2\beta_k^2,$ and $W_k=\beta_k \langle H_k,z_k-z_{\bar \mu}\rangle \mathbb E[\rho_k|\mathcal F_k]$ (which was shown to be non-negative in \eqref{eq:P2equation}).  
It follows that there exists an almost surely finite random variable $Z_{\infty}$ for which
\[
   \|z_{k}-z_{\bar \mu}\|^2 \text{ converges a.s. to }Z_{\infty} \text{ and }  \sum_{k=0}^{\infty}\beta_k \langle H_k,z_k-z_{\bar \mu}\rangle \mathbb E[\rho_k|\mathcal F_k]<\infty\text{ a.s.}
\]
on the event that $\sum_{k=0}^{\infty}\left(\frac{\beta_k}{a}\mathbb E\big[\|g_k-H_k\|^2|\mathcal F_k\big] + \gamma^2\beta_k^2\right)<\infty$. We now establish that this event occurs with probability $1$.

To this end, recall from the previous deliberations that, conditionally on $\mathcal F_k$, $z_k$ is fixed and, since the mini-batch samples used for $g_k$ are independent of $\mathcal F_k$ and independent of each other,  
\begin{equation}
\label{eq:varBound}
\begin{aligned}
\mathbb E\big[\|g_k-H_k\|^2|\mathcal F_k\big]
&= \mathbb E\Bigg[\Big\|{m_k}^{-1}\sum_{i=1}^{m_k}\left( \ell'(x_{k_{i}},z_k)-\mathbb E[\ell'(X,z_k)]\right)\Big\|^2\Big|\mathcal F_k\Bigg]
\\
&= \frac{1}{m_k^2}\sum_{i=1}^{m_k}\mathbb E_{X\sim \bar \mu}\Bigg[\Big\| \ell'(X,z_k)-\mathbb E[\ell'(X,z_k)]\Big\|^2\Bigg]\leq \sigma^2/m_k,
\end{aligned}
\end{equation}
where the final inequality is due to 
\cref{lem:clipped_Sgd_properties} (P3). Conclude that 
\[
\sum_{k=0}^{\infty}\left(\frac{\beta_k}{a}\mathbb E\big[\|g_k-H_k\|^2|\mathcal F_k\big] + \gamma^2\beta_k^2\right)\leq \sum_{k=0}^{\infty}\left(\frac{\beta_k\sigma^2}{am_k} + \gamma^2\beta_k^2\right)<\infty 
\]
as follows from the assumptions that $\sum_{k=0}^{\infty}\beta_k^2<\infty$ and $\sum_{k=0}^{\infty}\frac{\beta_k}{m_k}<\infty$. That is, this event occurs with probability $1$. 

It follows that $\sum_{k=0}^{\infty}\beta_k \langle H_k,z_k-z_{\bar \mu}\rangle \mathbb E[\rho_k|\mathcal F_k]<\infty$ a.s. and $\|z_{k}-z_{\bar \mu}\|^2 \text{ converges a.s. to }Z_{\infty}$ which is almost surely finite. It remains to be shown that $Z_{\infty} =0$. Following the proof of \cite[Theorem 1]{mai2021stability} we have that
\begin{equation}
\label{eq:convergenceBd1}
     \mathbb E[\rho_k|\mathcal F_k]\geq \frac{\gamma}{\gamma + \left( \mathbb E[\|g_k\|^2\vert \mathcal F_k]\right)^{\frac 12}}.
\end{equation}
We now note that 
\begin{equation}
\label{eq:convergenceBd2}
    \mathbb E[\|g_k\|^2\vert \mathcal F_k] =\|H_k\|^2+ \mathbb E[\|g_k-H_k\|^2\vert \mathcal F_k]\leq\|H_k\|^2+\frac{\sigma^2}{m_k},
\end{equation}
by applying \eqref{eq:varBound}. By the conditional Jensen's inequality, 
\[
   \Big \|\mathbb E\Big[\frac{1}{m_k}\sum_{i=1}^{m_k} \ell'(x_{k_i},z_k)|\mathcal F_k\Big]\Big\|\leq \mathbb E\Big[\frac{1}{m_k}\sum_{i=1}^{m_k} \|\ell'(x_{k_i},z_k)\||\mathcal F_k\Big] 
\]
and, recalling \eqref{eqn:decompositionSubdiff}, 
\[
\begin{aligned} 
    \|\ell'(x_{k_i},z_k)\| &= \left\|v_{k_i}-b +Ax_{k_i}\right\|\exp(\alpha g^{\ast}(-z_k/\alpha)-\langle b,z_k\rangle +\langle A^{\top}z_k,x_{k_i}\rangle)
    \\
    &\leq  \left(\|v_{k_i}\|+\|b\| +\|A\||\mathcal X|_{\infty}\right)\exp\left(\sup_{B_r}\alpha g^{\ast}(-(\cdot)/\alpha)+\|b\|r +\|A\||\mathcal X|_{\infty} r\right) 
\end{aligned}
\]
where $v_{k_i}\in -\partial g^{\ast}(-z_k/\alpha)$. Combining these two displayed equations, we obtain that 
\[
    \|H_k\|\leq  \Big(\sup_{z\in B_r}\sup_{v\in \partial g^{\ast}(-\frac{z}{\alpha})}\|v\|+\|b\| +\|A\||\mathcal X|_{\infty}\Big)\exp\left(\sup_{B_r}\alpha g^{\ast}(-(\cdot)/\alpha)+\|b\|r +\|A\||\mathcal X|_{\infty} r\right),  
\]
and underscore that $\sup_{z\in B_r}\sup_{v\in \partial g^{\ast}(-\frac{z}{\alpha})}\|v\|<\infty$ as argued in the proof of \cref{lem:clipped_Sgd_properties} (P3). Letting $S$ denote the right-hand side of the above display, insert this bound in \eqref{eq:convergenceBd1} which yields  that 

\[
   \infty>\sum_{k=0}^{\infty}\beta_k \langle H_k,z_k-z_{\bar \mu}\rangle \mathbb E[\rho_k|\mathcal F_k]\geq a\frac{\gamma}{\gamma + \sqrt{S^2+\sigma^2}}\sum_{k=0}^{\infty}\beta_k\|z_k-z_{\bar \mu}\|^2, \text{ almost surely,}   
\]
where we have also applied \eqref{eq:P2equation} and \eqref{eq:convergenceBd2} and used the fact that $m_k\geq 1$. 
It follows that $\sum_{k=0}^{\infty}\beta_k \|z_k-z_{\bar \mu}\|^2<\infty$ almost surely. On the other hand, $\|z_k-z_{\bar \mu}\|^2$ converges a.s. to some $Z_{\infty}<\infty$. 

Suppose now that  $Z_{\infty}>0$ on a set with positive probability (note that $Z_{\infty}$ is non-negative by construction). Then, $Z_{\infty}>0$ and $\|z_k-z_{\bar \mu}\|^2\to Z_{\infty}$, $\sum_{k=0}^{\infty}\beta_k\|z_k-z_{\bar \mu}\|^2 = \infty$ since $\sum_{k=0}^{\infty}\beta_k=+\infty$, but this contradicts the fact that $\sum_{k=0}^{\infty}\beta_k \|z_k-z_{\bar \mu}\|^2<\infty$. Conclude that $Z_{\infty}=0$ with probability $1$ so that $\|z_k-z_{\bar \mu}\|\to 0$ a.s., as required. 
\end{proof}

\section{Numerical Experiments}\label{sec:Numerics}

Having established the convergence of projected clipped SGD we now study the performance of this method on image denoising problems. Throughout we set the fidelity term to $g = \frac{1}{2} \Vert \cdot \Vert_{2}^{2}$ so that $g^{\ast} = \frac{1}{2} \Vert \cdot \Vert_{2}^{2}$ has full domain and $g^{\ast}$ is smooth so that no measurability issues arise in the subgradient selection. Moreover, \eqref{eqn:ass_domain} is satisfied for any choice of $r>0$  such that $z_{\mu} \in B_{r}$. We emphasize that \Cref{thm:stabilityDual} \ref{eqn:boundedNormOfSolutions} ensures that such $r$ is available. 

In the following experiments we set $r=500$ for simplicity and note that the iterates in each run all lie in the interior of this ball. In the sequel, we adopt setting (S2) wherein a dataset $\mathcal D$ of fixed sized $n$ is given and the corresponding prior is $\hat \mu_n$.
In this context, the primal-dual recovery can be accomplished as follows.

\begin{algorithm}
\DontPrintSemicolon
\caption{Primal Solution Recovery}\label{alg:PrimalSolnRecov}
\KwData{Output $z^{*}$ of \cref{alg:clippedSGD}.}
Initialize $\ell_{1} =0 \in \R^{m}$, $\ell_{2} = 0 \in \R$.\;
\For{$k = 1,2,3,\dots, n$}{
    $\ell_{1} \gets \ell_{1} + x_{k} \exp \langle A^{\top} z^{*},x_{k}\rangle$ \;
    $\ell_{2} \gets  \ell_{2} + \exp \langle A^{\top} z^{*},x_{k}\rangle$
}
Return $ x^{\ast}=\ell_{1}/\ell_{2}$.
\end{algorithm}

For our experiments, we apply some common \emph{ad hoc} engineering techniques to improve the performance of the method. %
Concretely, we apply a ``cosine learning rate'',as implemented in, e.g., PyTorch which consists of taking $\beta_{k}$ to behave like a cosine curve and is common in deep learning problems, see the specific hyperparameters for each experiment for explicit expressions. Moreover, for certain larger experiments we incorporate momentum to this algorithm as is standard in  the SGD literature, see \cref{sec:SVHN} for details.

The numerics are arranged as follows. We begin with an experiment on the extended MNIST (EMNIST) dataset \cite{DBLP:journals/corr/CohenATS17} which consists of $28\times 28$ black and white pictures of letters and digits, see \cref{sec:EMNIST} for details and \cref{fig:CSGD5,fig:CSGD7} for the results. We next apply our method to the more realistic Street View House Numbers (SVHN) dataset \cite{netzer2011reading} which consists of $32\times 32$ black and white and colored cropped images of house number digits. Exact details are provided in \cref{sec:SVHN}. We use this dataset to compare the performance of stochastic gradient methods with deterministic methods, see \cref{fig:SVHN_data_limit,fig:SVHN_BW_noise,fig:SVHN_color_Noise} for examples. We conclude with a case where the method fails to accurately recover a noisy image which illustrates the importance of having representative data in the prior, see \cref{fig:SVHN_cropped}.

\subsection{Extended MNIST}
\label{sec:EMNIST}
For the first host of experiments, we use the Extended MNIST dataset which consists of 800,000 hand drawn 28x28 pixel images including the numbers 0-9, and letters A-Z. We underscore that directly loading this complete dataset into memory requires 4GB of RAM, so that stochastic gradient methods which allow streaming data will be required for larger real-life datasets.  
Importantly, optimizing the original objective with the empirical prior on the entire dataset using gradient-based methods requires evaluating the full gradient $\nabla \phi_{\mu_{n}}(z)$ at any point $z$ and hence requires a full data-pass of $800,000$  inner products. While one could in principle store each $Ax_{i}$ to expedite this process (which is already prohibitive), this does not sidestep the need to compute the sum $\sum_{i=1}^ne^{\langle z, Ax_{i} \rangle }$ at each step of optimization, as is required to evaluate $\nabla \phi_{\mu_{n}}(z)$. %
In the following experiments, the ground truth images are hand drawn by the authors and hence are not present in the original dataset which is used to construct the prior.

 For our first experiment, \cref{fig:CSGD7}, we denoise and deblur an image generated as $b = Ax + \mathcal{N}(0,(0.1 \Vert x \Vert)^2)$  where $A$ is a tridiagonal matrix with constant diagonal entries $0.25,0.5,0.25$ and $\mathcal{N}(0,(0.1 \Vert x \Vert)^2)$ is the mean-zero normal distribution with variance $(0.1 \Vert x \Vert)^2$. For this experiment, we set the following hyperparameters for the projected clipped SGD : $\alpha = 10$, $\beta_{k} = 10^{-4}+ \frac{1}{2}(10^{-2}-10^{-4})(1+\cos(\frac{k+1}{60000}\pi))$, $m_{k} =5$, $\gamma = 100$, $z_{0} = 0$, and a maximum iteration count of $2000$. Moreover in the final display we postprocess the output by setting pixels with intensity above 0.8 to 1 and those below 0.2 to 0.  %

\begin{figure}[h]
    \centering
    \begin{subfigure}[t]{0.22\textwidth}
         \includegraphics[width=\textwidth]{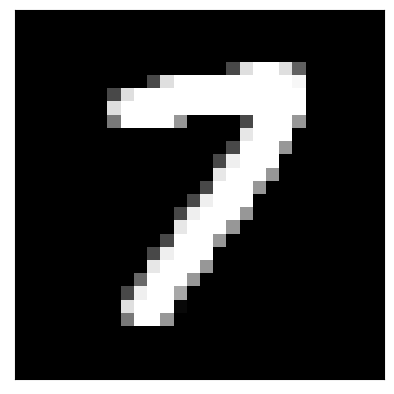}
         \caption{Ground truth}
     \end{subfigure}
    \begin{subfigure}[t]{0.22\textwidth}
         \includegraphics[width=\textwidth]{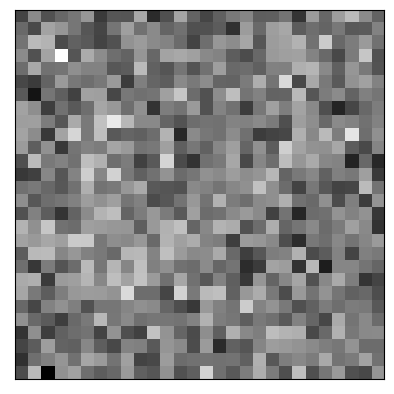}
         \caption{Observed image}
     \end{subfigure}
     \begin{subfigure}[t]{0.22\textwidth}
          \includegraphics[width=\textwidth]{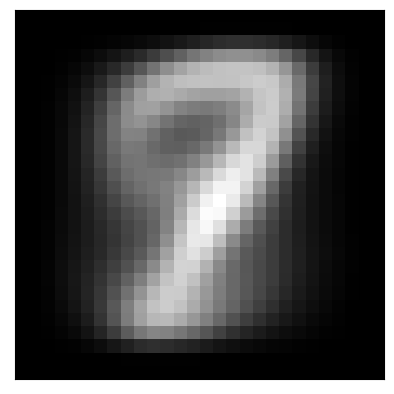}
         \caption{\centering CSGD output \\ (2000 iterations)}
     \end{subfigure}
     \begin{subfigure}[t]{0.22\textwidth}
          \includegraphics[width=\textwidth]{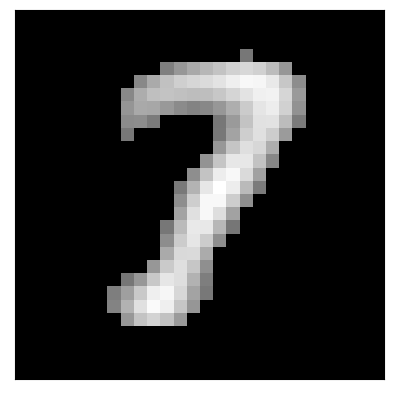}
         \caption{Postprocessed}
     \end{subfigure}
     \caption{Result of solving the MEM problem using Projected Clipped SGD using a mini-batched stochastic gradient. }
     \label{fig:CSGD7}
\end{figure}

For \cref{fig:CSGD5} the same blurring kernel $A$ is used along with the following hyperparameters for the projected clipped SGD : $\alpha = 2$, $\beta_{k} = 10^{-4}+ \frac{1}{2}(10^{-2}-10^{-4})(1+\cos(\frac{k+1}{60000}\pi))$, $m_{k} =5$, $\gamma = 10$, and $z_{0} = 0$ with a maximum iteration count of $50000$.

\begin{figure}[h]
    \centering
    \begin{subfigure}[t]{0.22\textwidth}
         \includegraphics[width=\textwidth]{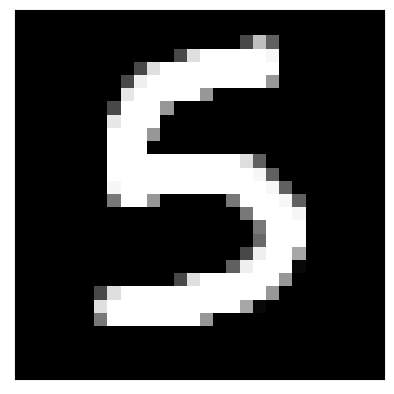}
         \caption{Ground truth}
     \end{subfigure}
    \begin{subfigure}[t]{0.22\textwidth}
         \includegraphics[width=\textwidth]{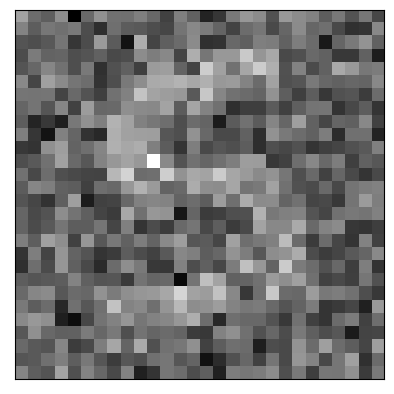}
         \caption{Observed image}
     \end{subfigure}
     \begin{subfigure}[t]{0.22\textwidth}
          \includegraphics[width=\textwidth]{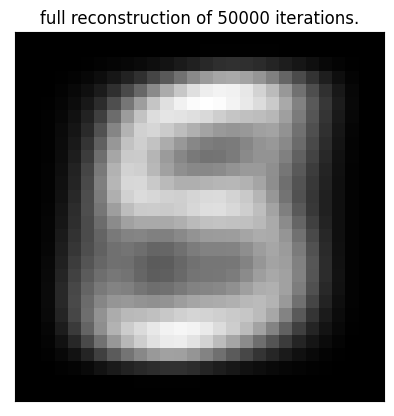}
         \caption{\centering CSGD output \\ (50000 iterations)}
     \end{subfigure}
     \begin{subfigure}[t]{0.22\textwidth}
          \includegraphics[width=\textwidth]{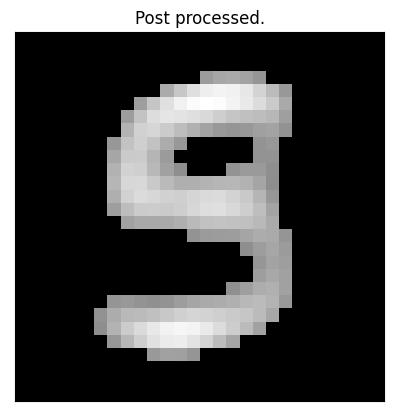}
         \caption{Postprocessed}
     \end{subfigure}
     \caption{Result of solving the MEM problem using Projected Clipped SGD using a mini-batched stochastic gradient.}
     \label{fig:CSGD5}
\end{figure}

In both cases, the postprocessed output is close to the ground truth image. While the hyperparameters used above are tuned for each example, we underscore that, in testing, once the hyperparameters fall within a reasonable range further tuning does not significantly improve the quality of the recovered image except for the batchsize. Indeed, if the batchsize $m_k$ is set to be too small with the same amount of iterations this can result in significantly worse results.

We remark that in \cref{fig:CSGD5}, the postprocessed image could be interpreted as a $5$ or the character g. This visual artifact can be explained by noting that in many handwriting styles these characters differ only by a small stroke and so, in the presence of large noise, it is difficult to discern if that stroke was present or not.

\subsection{Street View House Numbers}
\label{sec:SVHN}

We now consider the more realistic Street View House Numbers (SVHN) dataset \cite{netzer2011reading} which contains 630000 images of house numbers taken directly from Google street view. This dataset contains both original images of the numbers as well as cropped images of each digit of size $32\times 32$ for a total of 1.8 million digits. This dataset presents additional real-world challenges, such as confounding features, color channels, differing aspect ratios, rotations, and so on. See \cref{fig:SVHN_examples} for an example of the original image and the corresponding cropped digit. 
\begin{figure}[h]
    \centering
    \begin{subfigure}[t]{0.22\textwidth}
         \includegraphics[width=\textwidth]{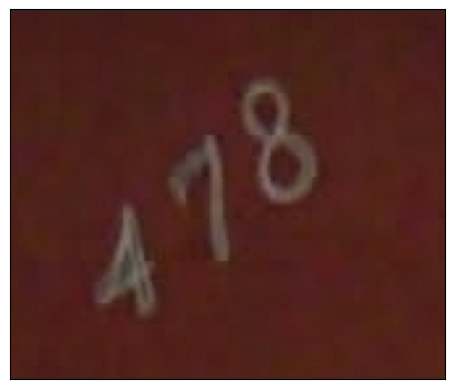}
         \caption{Example data}
     \end{subfigure}
    \begin{subfigure}[t]{0.22\textwidth}
         \includegraphics[width=\textwidth]{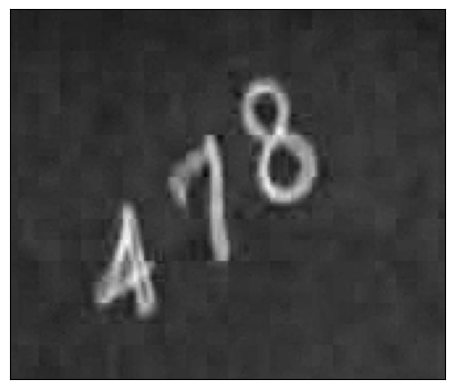}
         \caption{Example data}
     \end{subfigure}
     \begin{subfigure}[t]{0.22\textwidth}
          \includegraphics[width=\textwidth]{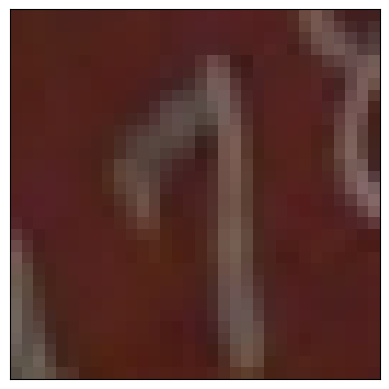}
         \caption{32$\times$32 cropped}
     \end{subfigure}
     \begin{subfigure}[t]{0.22\textwidth}
          \includegraphics[width=\textwidth]{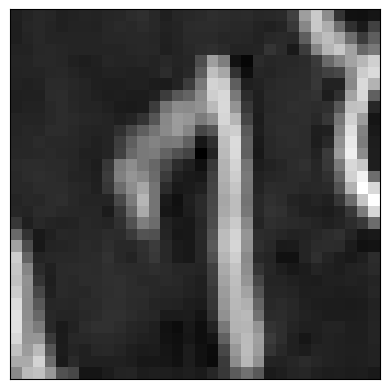}
         \caption{32$\times$32 cropped}
     \end{subfigure}
     \caption{Examples of original house numbers and the corresponding cropped digits from the SVHN dataset.}
     \label{fig:SVHN_examples}
\end{figure}

Throughout we consider only the cropped images so that all images are of the same size (i.e., $32\times 32$). The corresponding problems are thus $d=1024$ dimensional in black and white or $d=3072$ dimensional in color. Furthermore, for each of the following experiments, the ground truth is taken as a point originally found in SVHN, but during recovery is removed from the dataset and hence not used to construct the prior. 

We now demonstrate the advantage of clipped SGD over solving the original formulation of the problem. As before, we generate the observed image as $b = x + \mathcal{N}(0,(0.15\Vert x\Vert)^2)$ which corresponds to adding noise to the image only. To denoise the image we draw  $n=50000$ and $n=100000$ samples and solve the corresponding full empirical MEM problem with the empirical prior via black-box BFGS applied to the original dual objective. It is critical to emphasize that even for $n=500000$, it has become prohibitively expensive to compute the MEM solution for the full empirical problem, as simply storing the partial dataset already requires $4$GB of ram,  with each gradient evaluation requiring a full data-pass, a problem that will only become worse as the ambient dimension increases. Furthermore, we solve the reformulated problem using the projected clipped method with the parameters $\alpha = 1$, $\beta_{k} = 10^{-4}+ \frac{1}{2}(10^{-2}-10^{-4})(1+\cos(\frac{k+1}{60000}\pi))$, $m_{k} =10$, $\gamma = 10$, and $z_{0} = 0$ with a maximum iteration count of $500000$.
We also incorporate a momentum term which is seen to speed up convergence. Concretely, the update in \cref{alg:clippedSGD} for $g_k$ is replaced by%
\begin{align*}
    g_k &\gets 0.1 g_{k-1} + \frac 1 m_k\sum_{i=1}^{m_k} \ell'(x_{k_i},z_k) \text{ for a measurable selection }  \ell'( x_{k_{i}},z_{k}) \in \partial_z \ell( x_{k_{i}},z_{k}) 
\end{align*}
with the convention that $g_{-1}=0$. This modification preserves some information about the direction of previous  gradients before the clipping is performed. The results are compiled in \cref{fig:SVHN_data_limit}.

\begin{figure}[h]
    \centering
    \begin{subfigure}[t]{0.22\textwidth}
         \includegraphics[width=\textwidth]{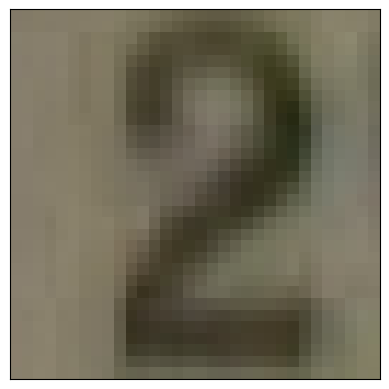}
         \caption{Ground truth}
     \end{subfigure}
    \begin{subfigure}[t]{0.22\textwidth}
         \includegraphics[width=\textwidth]{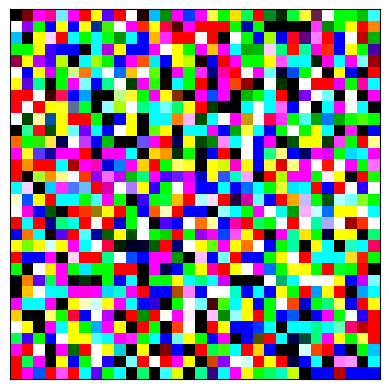}
         \caption{Observed image}
     \end{subfigure} \\
     
     \begin{subfigure}[t]{0.22\textwidth}
          \includegraphics[width=\textwidth]{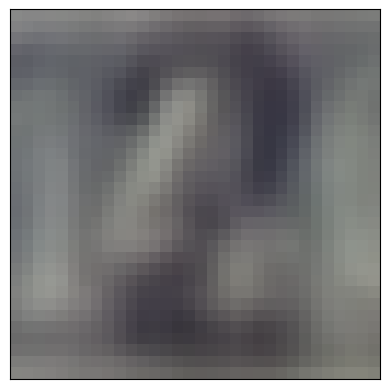}
         \caption{\centering Original formulation, $n=50000$}
     \end{subfigure}
     \begin{subfigure}[t]{0.22\textwidth}
          \includegraphics[width=\textwidth]{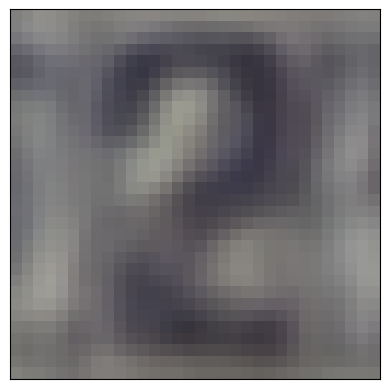}
         \caption{\centering Original formulation, $n=100000$}
     \end{subfigure}
     \raisebox{1cm}{$\xrightarrow[\text{use reformulation}]{\text{Too much data}}$}
     \begin{subfigure}[t]{0.22\textwidth} \includegraphics[width=\textwidth]{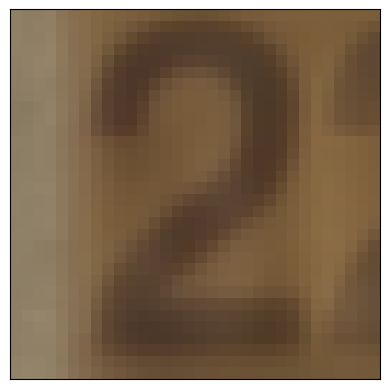}
         \caption{\centering Reformulation \\
         500000 iterations}
     \end{subfigure}
     \caption{Comparison of solving the MEM problem with the prior $\hat \mu_n$ with $n=50000$ and $n=100000$ using full gradient-based methods and using $500000$ iterations of the stochastic gradient method.}
     \label{fig:SVHN_data_limit}
\end{figure}

We underscore that, while the stochastic method is more scalable, a natural downside is that any given iteration may fail to decrease the objective value and it is prone to oscillate in some neighborhood of the minimum.

Next, we illustrate how the solution quality improves as the number of stochastic gradient steps increases. To this end, we plot the recovered solutions after a fixed number of iterations for a black and white image and a colored image. In both cases case, we set $\alpha = 2$, $\beta_{k} = 10^{-4}+ \frac{1}{2}(10^{-2}-10^{-4})(1+\cos(\frac{k+1}{60000}\pi))$, $m_{k} =10$, $\gamma = 10$, and $z_{0} = 0$ with a maximum iteration count of $250000$. The results are compiled in \cref{fig:SVHN_BW_noise,fig:SVHN_color_Noise}.

\begin{figure}[h]
    \centering
    \begin{subfigure}[t]{0.22\textwidth}
         \includegraphics[width=\textwidth]{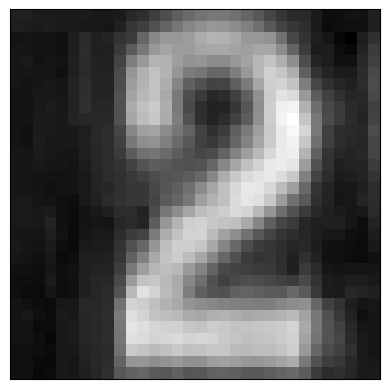}
         \caption{Ground truth}
     \end{subfigure}
    \begin{subfigure}[t]{0.22\textwidth}
        \includegraphics[width=\textwidth]{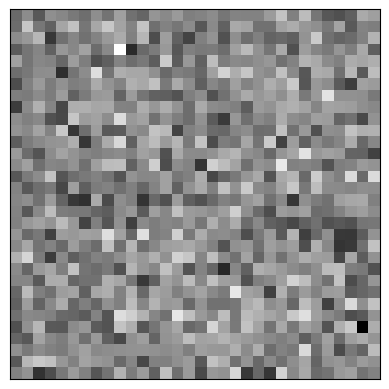}
         \caption{Observed image}
     \end{subfigure}
     \begin{subfigure}[t]{0.22\textwidth}
          \includegraphics[width=\textwidth]{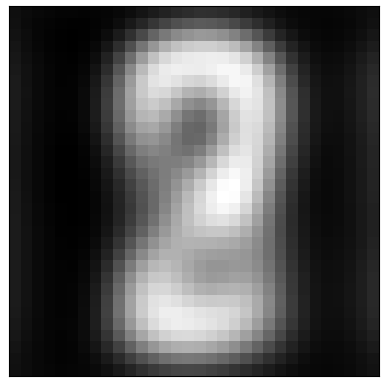}
         \caption{50000 iterations}
     \end{subfigure}
     \begin{subfigure}[t]{0.223\textwidth}
          \includegraphics[width=\textwidth]{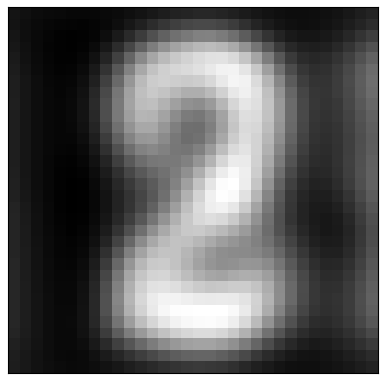}
         \caption{500000 iterations}
     \end{subfigure}
     \caption{Result of applying the clipped SGD algorithm after 50000 and 500000 iterations. The observed image is generated via $b=x+\mathcal N(0,(0.2 \Vert x \Vert)^2)$. After 50000  iterations the digit $2$ is visible albeit with some blurring around the edges. After 500000 iterations the digit is clearer. }
     \label{fig:SVHN_BW_noise}
\end{figure}

\begin{figure}[h]
    \centering
    \begin{subfigure}[t]{0.19\textwidth}
         \includegraphics[width=\textwidth]{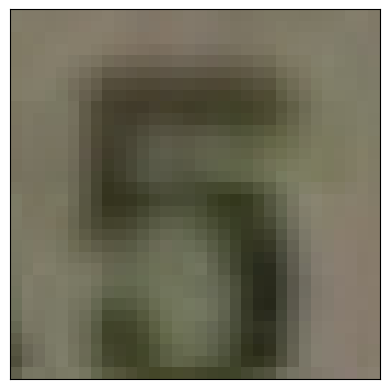}
         \caption{Ground truth}
     \end{subfigure}
    \begin{subfigure}[t]{0.19\textwidth}
        \includegraphics[width=\textwidth]{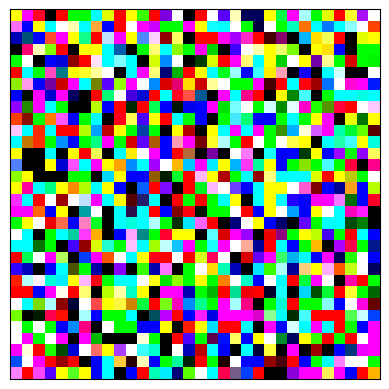}
         \caption{\centering Observed \\image}
     \end{subfigure}
    \begin{subfigure}[t]{0.19\textwidth}
        \includegraphics[width=\textwidth]{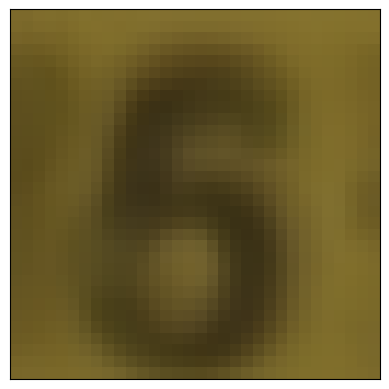}
         \caption{\centering 1000 \\ iterations}
     \end{subfigure}
     \begin{subfigure}[t]{0.19\textwidth}
          \includegraphics[width=\textwidth]{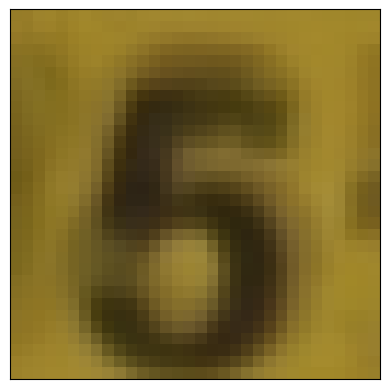}
         \caption{\centering 100000 \\iterations}
     \end{subfigure}
     \begin{subfigure}[t]{0.19\textwidth}
          \includegraphics[width=\textwidth]{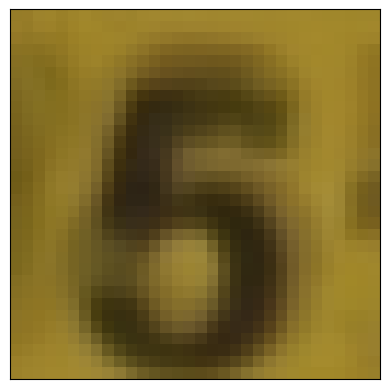}
         \caption{\centering 250000 \\ iterations}
     \end{subfigure}
     \caption{Result of applying the clipped SGD algorithm after 1000, 100000, and 2500000 iterations. The observed image is generated via $b=x+\mathcal N(0,(0.3 \Vert x \Vert )^2)$. We note that the output at $100000$ iterations is very similar to that at $250000$ indicating that the convergence occurs relatively fast.}
     \label{fig:SVHN_color_Noise}
\end{figure}

We note that the solutions obtained in \cref{fig:SVHN_color_Noise} have a discolored background; this can be understood by noting that the solution is  obtained by taking a weighted combination of the datapoints so that the different background colors get averaged together. While color-valued postprocessing could potentially aid in the visual recovery, this introduces another layer of complexity beyond the scope of the MEM problem.

Finally, we note that this approach depends strongly on the dataset used to form the prior. Indeed, as illustrated in  \cref{fig:SVHN_cropped}, if the ground truth is not well-represented by elements of the dataset our method cannot accurately approximate it. In that figure, the ground truth is blurry and includes two digits due to poor cropping. Given that such images are outliers within the SVHN dataset it is unsurprising that the corresponding recovered image does not match the ground truth well. Here, this image is constructed with the hyperparameters $\alpha = 2$, $\beta_{k} = 10^{-4}+ \frac{1}{2}(10^{-2}-10^{-4})(1+\cos(\frac{k+1}{60000}\pi))$, $m_{k} =10$, $\gamma = 10$, and $z_{0} = 0$.

\begin{figure}[h] 
\centering
    \begin{subfigure}[t]{0.22\textwidth}
         \includegraphics[width=\textwidth]{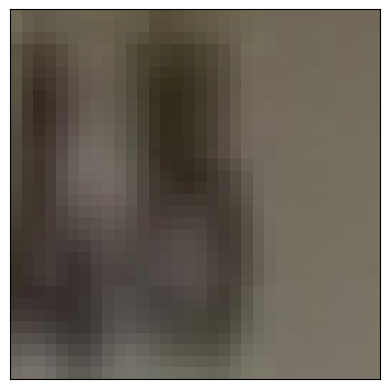}
         \caption{Ground truth}
     \end{subfigure}
    \begin{subfigure}[t]{0.22\textwidth}
         \includegraphics[width=\textwidth]{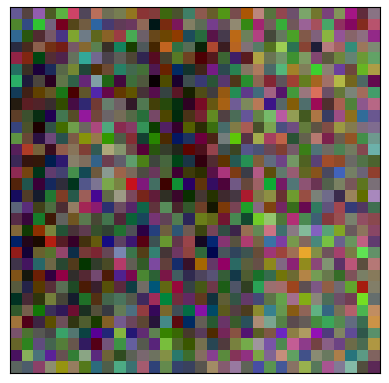}
         \caption{Observed image}
     \end{subfigure}
     \begin{subfigure}[t]{0.22\textwidth}
          \includegraphics[width=\textwidth]{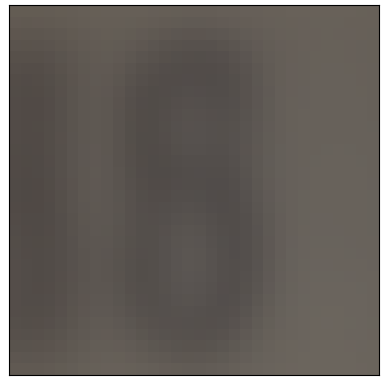}
         \caption{Recovered image}
     \end{subfigure}
     \caption{Results of attempting to denoise an outlier within the SVHN dataset.}
     \label{fig:SVHN_cropped}
\end{figure}

\section{Conclusion}

This work has advanced the statistical and computational study of the MEM problem. Namely, a parametric rate of convergence for empirical dual and primal solutions in expectation was derived which improves upon the $O(n^{-1/4})$ rate obtained in \cite{mkr2024maximumentropy}. This statistical result was obtained by characterizing the stability properties of the primal and dual MEM solutions under perturbations of the prior distribution and follows from standard results in convex analysis and statistics. Furthermore, a connection between the MEM problem and conventional expected risk minimization was established. This connection was leveraged to propose and analyze new stochastic gradient methods for solving MEM problem which were shown to yield good results even in the presence of large amounts of noise.

Many promising research directions stem from this line of work. For instance, it is of interest to generalize the sample complexity results derived herein to the case of priors with unbounded support and to establish if these rates are minimax optimal. Establishing further statistical properties of these estimators such as  central limit theorem or finite sample Gaussian approximation results \`a la Stein's method would also be useful for understanding the quality of approximations by empirical measures in the MEM framework and for calibrating statistical and optimization error.  

The aforementioned connection to expected risk minimization also opens the door to a more engineering-based study of alternative stochastic optimization methods for the MEM problem with the aim of improving its empirical performance.

\newpage

\appendix
\section{Alternate Derivation of Parametric Rates}
\label{app:MestimationRates}

In this section, we establish how the expected rate of convergence of the dual MEM solutions (established in \cref{thm:parametricRatesDual}) can be derived using standard techniques in the $M$-estimation literature. Namely, we will apply Corollary 1 in \cite{birge1993rates}. To leverage this result, we make two additional assumptions on top of the already global assumptions (\ref{eq:ass}):
\begin{tcolorbox}
\vspace{-1em}
\begin{equation}
    \textrm{$g$ is level bounded and $g^{*}$ is Lipschitz continuous on the closed ball $\alpha^{-1}B_r$.} \tag{A3} \label{eq:ass_appendix}
\end{equation}
\end{tcolorbox} 
Here $r$ is defined as 
\begin{equation*}
    r = \frac{\alpha}{\beta}\left[ \min_{v_{g} \in -\partial g^{*}(0)} \Vert v_{g} \Vert +  \Vert b \Vert  + \Vert A \Vert \vert \spt (\mu) \vert \right],
\end{equation*}
which is well defined by the level-boundedness of $g$, and has a necessary property, $z_{\mu} \in B_{r}$ by \cref{thm:stabilityDual}, that we will need later. This continuity assumption was notably not required in the derivation of the parametric rate via elementary arguments, see \cref{thm:parametricRatesDual}.

\begin{lemma}
\label{lem:lipschitzConstantEmpirical}
    For any $x\in\mathcal X$,  $h_x:z\in B_r\mapsto \exp\left(\alpha g^{\ast}(-z/\alpha)-\langle b,z\rangle +\langle A^{\top}z,x\rangle\right)$ is Lipschitz continuous with parameter 
    \begin{equation}
    \label{eq:lipschitzH_x}
   \begin{aligned} 
    L_{h_x}&:=\sup_{B_r}h_x \left( L_{g^{\ast}}+ \|b\|+\| A\|\|x\|\right)
    \\
    &\leq \exp\left(\alpha \sup_{B_r}g^{\ast}(-\cdot/\alpha)+\|b\| r +\|{A}\|\|x\|r\right) \left( L_{g^{\ast}}+ \|b\|+\| A\|\|x\|\right),
    \end{aligned}
    \end{equation}
    where $L_{g^{\ast}}$ is the Lipschitz constant of $g^{\ast}$ on $\alpha^{-1}B_r$.
\end{lemma}
\begin{proof}
    As the function $t\in(-\infty,a)\mapsto e^t$ is $e^a$-Lipschitz continuous for any $a\in\mathbb R$,
    \[
    \begin{aligned}
    &\left|\exp\left(\alpha g^{\ast}(-z/\alpha)-\langle b,z\rangle +\langle A^{\top}z,x\rangle\right)-\exp\left(\alpha g^{\ast}(-z'/\alpha)-\langle b,z'\rangle +\langle A^{\top}z',x\rangle\right)\right|
    \\
    &\leq\sup_{B_r} \exp\left(\alpha g^{\ast}(-\cdot/\alpha)-\langle b,\cdot\rangle +\langle A^{\top}\cdot,x\rangle\right) \times\\
    &\hspace{12em}\left|\alpha \left(g^{\ast}(-z/\alpha)-g^{\ast}(-z'/\alpha)\right)-\langle b,z-z'\rangle +\langle A^{\top}(z-z'),x\rangle\right|
    \\
    &\leq \sup_{B_r} \exp\left(\alpha g^{\ast}(-\cdot/\alpha)-\langle b,\cdot\rangle +\langle A^{\top}\cdot,x\rangle\right) \left( L_{g^{\ast}}\| z-z'\|+ \|b\|\|z-z'\| +\| A\|\|x\|\|z-z'\|\right),
    \end{aligned} 
    \]
    proving the claimed result.
\end{proof}

With these lemmas in hand, we will apply Corollary 1 in \cite{birge1993rates}.  We underscore that the cited reference deals with a more general setting, wherein a random variable $Z_i = f(X_i,W_i,s)$ taking values in a measurable space $\mathscr Z$ is observed. Here,     $X_i,W_i$ are independent random variables taking values in the measurable spaces $\mathscr X$ and $\mathscr W$ respectively, $s\in\mathscr S$ is a parameter, and $f:\mathscr X\times \mathscr W\times \mathscr S$ is a measurable function. In this setting, $(X_i,W_i,Z_i)$ is distributed according to a population distribution $P_s^i$ which depends on the choice of $s$. The aim is to perform inference on the (unknown) parameter $s\in\mathscr{S}$ based on the observations $(Z_i)_{i=1}^n$ under the assumption  that there is a known function $\gamma:\mathscr{Z}\times\mathscr S\to \mathbb R$ for which 
\begin{equation}
\label{eq:contrastFunction}
    \inf_{t\in\mathscr{S}}\mathbb E_{\bar P_s} \left[\gamma(Z,t)\right] = \mathbb E_{\bar P_s}\left[\gamma(Z,s)\right] \text{ for every }s\in\mathscr S,
\end{equation}
where $\bar P_s:=\frac{1}{n}\sum_{i=1}^nP_s^i$. $\gamma$ is called a \emph{contrast function} and the \emph{minimum contrast estimator} is any minimizer of the function $t\in \mathscr S\mapsto \frac 1n \sum_{i=1}^n\gamma(Z_i,t)$, recalling that $(Z_i)_{i=1}^n$ is the observed sample. 

However, the setting considered in the present work  assumes access to i.i.d. observations from a fixed prior $\mu$ and we wish to perform inference on the unique minimizer, $z_{\mu},$ of 
\[
z\in B_r\mapsto \int \exp\left(\alpha g^{\ast}(-z/\alpha)-\langle b,z\rangle +\langle A^{\top}z,x\rangle\right)d\mu(x)=\mathbb E_{\mu}\left[h_X(z)\right].
\]
This corresponds to minimum contrast estimation with $\mathscr{S}=B_r$, where the true parameter is given by $s=z_{\mu}$,  $Z_i = f(X_i,W_i,s)= X_i$, $\gamma(Z,t) = \ell(X,t)= h_X(t)$, and $P_s^i$ can be identified with $\mu$ for every $i=1,2,\dots,n$ so that $\bar P_s=\mu$. An important distinction is that \eqref{eq:contrastFunction} holds for $s=z_{\mu}$, but not for any other choice of $t\in\mathscr{S}$ (indeed, $\bar P_t$ is not even defined for $t\neq s$). This difference is innocuous however, as in \cite{birge1993rates} the true parameter $s$ is assumed to be fixed and the population-level distribution $\bar P_s$ is considered only for the true parameter in the subsequent arguments.

As the conditions of Corollary 1 in \cite{birge1993rates} are relatively technical, we summarize them here in the interest of keeping the presentation self-contained. The following correspond to conditions (A1)-(A'4) in their work. 
\begin{enumerate}
    \item[(A1)] There exist positive functions $M:\mathscr W\to \mathbb R$ and $\Delta:\mathscr X\times \mathscr S^2\to \R$ and $B>0$ for which 
    \[
        \left|\gamma(z,t)-\gamma(z,t')\right|\leq M(w)\Delta(x,t,t')\leq BM(w), \bar P_s\text{-a.s.}
    \]
    where $z=f(x,w,s)$ and the inequality holds for every $t,t'\in\mathscr{S}$. 
    \item[(A2)] There exists positive constants $a,\Gamma$ such that $\int e^{a M(w)}dP_s^i\leq \Gamma$ for every $1\leq i\leq n$. 
    \item[(A3)] For any $\delta>0$ there exists a covering $\mathscr S_{\delta}$ of $\mathscr S$ of finite cardinality for which 
    \[
        \mathbb E_{\bar P_s}\left[\left(\sup_{t,t'\in  S_{\delta}}\Delta^2(X,t,t')\right)\right]\leq \delta^2, \text{ for every }S_{\delta}\in\mathscr S_{\delta}.
    \]
    Moreover, we set $H(\delta,\mathscr{S})=\log(\mathrm{card}(\mathscr S_{\delta}^{\star}))$ where  $\mathscr S^{\star}_{\delta}$ is the cover with minimal cardinality which satisfies the above properties. 
    \item[(A'4)] There exists a constant $C\geq Ba/\Gamma$ for which 
    \begin{equation}
    \label{eq:pseudoDistBound}
        \mathbb E_{\bar P_s}\left[\Delta^2(X,s,t)\right]\leq C\mathbb E_{\bar P_s}\left[\gamma(Z,t)-\gamma(Z,s)\right].
    \end{equation}
\end{enumerate}

\begin{proposition}
    Assumptions (A1)-(A'4) hold for $\gamma(z,t)=\ell(X,t)=h_{X}(t)$ with the following choices
    \begin{enumerate}
        \item $M(w)=\sup_{x\in\spt(\mu)}L_{h_x}=:M$, $\Delta(x,t,t')=\|t-t'\|$, and $B=2r$ in (A1),
        \item $a=2/M$ and $\Gamma =e^2$ in (A2),
        \item $\mathscr{S}_{\delta}$ to be a cover of $\mathscr S=B_r$ by closed balls of radius $\delta/2$ in (A3). Then,  \begin{equation}
    \label{eq:entropyBall}
        H(\delta,B_r)\leq F(\delta,r):=\begin{cases}
             d\log\left( \frac{4r}{\delta}+1\right),&\text{if }\delta\leq 2r,
             \\
             0,&\text{otherwise.}
        \end{cases}
    \end{equation}
        \item 
        $C=\max\left\{ \frac{2\alpha}{\beta} \sup_{x\in\spt(\mu)}(\inf_{\mathbb R^d}h_x)^{-1},4r/(Me^2)\right\}$
         in (A'4).
    \end{enumerate}
\end{proposition}
\begin{proof}
    For item (1), we may apply \cref{lem:lipschitzConstantEmpirical} to obtain that 
    \[
        |h_x(z)-h_x(z')|\leq L_{h_x}\|z-z'\|\leq \sup_{x\in\spt(\mu)}L_{h_x}\|z-z'\|\leq 2\sup_{x\in\spt(\mu)}L_{h_x}r
    \]
    for every $z,z'\in B_r$ and $\mu$-a.e. $x\in\mathcal X$. That is, $M(w)=\sup_{x\in\spt(\mu)}L_{h_x}=:M$, $\Delta(x,t,t')=\|t-t'\|$, and $B=2r$ in (A1).

    For item (2), we follow point (2) of the remark on p.118 of \cite{birge1993rates} and take $a=2/M$ and $\Gamma = e^{aM}=e^2$ so that $\sqrt{\Gamma}/a=Me/2$ in (A2).

    For item (3) we have that $\Delta(x,t,t')=\|t-t'\|$ as noted above. Consequently, any cover $\mathscr S_\delta$ of $B_r$ by closed balls of radius $\delta/2$ satisfies 
    \[
        \mathbb E_{\mu}\left[\sup_{t,t'\in B_{\delta}}\|t-t'\|^2\right]\leq \delta^2
    \] 
    and 
    is hence admissible for (A3). Applying Lemma 4.14 of \cite{massart2007concentration}, the minimal such cover consists of at most 
    \[
       \begin{cases}
             \left( \frac{4r}{\delta}+1\right)^d,&\text{if }\delta\leq 2r,
             \\
             1,&\text{otherwise,}
        \end{cases}
    \]
    sets.

    Item (4) follows directly from the proof of (P2) in \cref{lem:clipped_Sgd_properties}, noting that we must set $C=\max\left\{ \frac{2\alpha}{\beta} \sup_{x\in\spt(\mu)}(\inf_{\mathbb R^d}h_x)^{-1},4r/(Me^2)\right\}$ recalling that $\frac{Ba}{\Gamma}=4r/(Me^2)$ from the definitions of $a,B,\Gamma,$ and $M$ from points (A1) and (A2).
\end{proof}

The desired convergence rate result, Corollary 1 in \cite{birge1993rates}, depends implicitly on Theorem 2 from the same reference which we state presently. To state this result we define $H(u,B_{\sigma}(z_\mu)) = \log(\mathrm{card}(B_{\sigma}(z_\mu)_{u}^{\star}))$, where $B_{\sigma}(z_\mu)_{u}^{\star}$ is the minimal covering which satisfies 
\begin{equation*}
     \mathbb E_{\mu} \left[\sup_{t,t'\in B_{u}}\|t-t'\|^2\right]\leq u^2, \text{ for each } B_{u} \in B_{\sigma}(z_\mu)_{u}^{\star},
\end{equation*}
analogously to the definition of $H(u,\mathscr{S})$. Moreover, we define the outer probability of an arbitrary subset $B$ of the underlying sample set $\Omega$ as 
\[
    \mathbb P^{\ast}(B) = \inf\left\{\mathbb P(A):B\subset A, A \text{ is Borel measurable}\right\}.
\]

\begin{lemma}[Theorem 2 in \cite{birge1993rates}]
\label{lem:probBound}
Under assumptions (A1)-(A'4),  let $B_{\sigma}(z_{\mu}) = \left\{t\in B_r:\|t-z_{\mu}\|\leq \sigma\right\}$ and let $\tilde H:(0,\infty)\times (0,1]\to [1,\infty)$ be such that 
   \[
        H(u,B_{\sigma}(z_{\mu}))\leq \tilde{H}(u,\sigma) \text{ for any }(u,\sigma)\in (0,\infty)\times (0,1]
   \]
   and such that the function $\varphi:\sigma\in(0,1/2]\mapsto\int_{\frac{\sigma^2}{128\Gamma C}}^{2\sigma}\sqrt{\tilde H(u,2\sigma)}du$ satisfies the property that $\sigma\in(0,1/2]\mapsto \varphi(\sigma)/\sigma$ is non-increasing and continuous. %
   Then, if $\sigma^{*}>0$ is the unique solution of the equation 
   \begin{equation}
   \label{eq:equationSigma}
       \frac{3840C\sqrt{\Gamma}}{a}\varphi(\sigma)=\sqrt n \sigma^2, 
   \end{equation} 
   for any non-negative integer $L$ satisfying $2^{2L}(\sigma^*)^2\leq \frac{BC\Gamma}{a}$, it holds that  
    \[
        \mathbb P^{*}\left(\sup_{t\in \mathscr S}\frac{\left|\left(\frac 1n \sum_{i=1}^n\gamma(X_i,t)- \frac 1n \sum_{i=1}^n\gamma(X_i,s) -\mathbb E_{\bar P_s}\left[ \gamma(Z,t)-\gamma(Z,s)\right]\right)\right|}{\max\{\|s-t\|^2,2^{2L}(\sigma^*)^2\}}\geq \frac 1{2C}\right)
    \]
    is at most $8.1\exp\left(-2^{2L+1}n(\sigma^*)^2/b\right)$ for $b=\frac{9(3840)^2C^2\Gamma}{a^2}\leq n(\sigma^*)^2$.
\end{lemma}

To apply this Lemma in the present setting, it will be useful to characterize the solutions of \eqref{eq:equationSigma}. 
    
\begin{proposition}
\label{prop:OptimalSolution}
The function  
\[
\tilde H(\delta,\sigma)=\begin{cases}\frac{4\sigma d}{\delta},&\text{if }\delta\leq 2\sigma,
\\
2d,&\text{otherwise}.
\end{cases}
\]
satisfies the conditions of \cref{lem:probBound}. Moreover, for each $n\in\mathbb N$ sufficiently large, the solution of \eqref{eq:equationSigma} is unique and of the form  
$\sigma^{\star}(n)= K(n)n^{-1/2}\in(0,1/2]$ 
where 
\[
K(n)=\left(\frac{1}{2}\sqrt{\frac{R^2n^{-1/2}}{256\Gamma C}+4 R}-\frac{Rn^{-1/4}}{2\sqrt{256\Gamma C}}\right)^2,\text{ for }R=\frac{30720C\sqrt{\Gamma d}}{a}.
\] 
\end{proposition}
\begin{proof}
     Recalling \eqref{eq:entropyBall} and applying the bound $\log(1+x)\leq x$ for every $x\geq 0$, we note that 
\[
H(\delta,\sigma)\leq F(\delta,\sigma)\leq  \tilde H(\delta,\sigma)=\begin{cases}\frac{4\sigma d}{\delta},&\text{if }\delta\leq 2\sigma,
\\
2d,&\text{otherwise}.
\end{cases}
\]
The expression for $\varphi$ then follows by direct integration;
\[
    \varphi(\sigma) = \int^{2\sigma}_{\frac{\sigma^2}{128\Gamma C}} \sqrt{\tilde H(u,2\sigma)}du= \begin{cases} 8\sqrt{\sigma d}\left(\sqrt{\sigma}-\frac{\sigma}{\sqrt{256\Gamma C}}\right),&\text{if }0<\sigma \leq 512 \Gamma C,
    \\
    4\sigma
    \sqrt{d}\left(2-\sqrt 2 \right)-\frac{\sigma^2\sqrt{2d}}{128\Gamma C},&\text{if } \sigma >512\Gamma C.
    \end{cases}
\]
From these expressions we see that $\varphi(\sigma)/\sigma$ is nonincreasing and continuous on $(0,1/2]$. Now, we wish to find some $\sigma^{\star}(n)$ solving \eqref{eq:equationSigma}. Taking the ansatz $\sigma^{\star}(n) = K(n)n^{-1/2}$ in the equation
\[
    \frac{3840C\sqrt{\Gamma}}{a}\varphi(\sigma)=\sqrt n \sigma^2,
\]
we observe that if $ K(n)n^{-1/2}\leq 512\Gamma C$, the equation reads 
\begin{equation}
\label{eq:defeq}K(n)^2n^{-1/2}=
   \frac{3840C\sqrt{\Gamma}}{a} 8\sqrt d K(n)n^{-1/2}\left( 1 - \frac{K(n)^{1/2}n^{-1/4}}{\sqrt{256\Gamma C}}\right).%
\end{equation}
Letting $R= \frac{3840C\sqrt{\Gamma}}{a} 8\sqrt d$, the above equation can be written as 
\[
   K(n)n^{-1/2}\left(K(n)+\frac{Rn^{-1/4}}{\sqrt{256\Gamma C}}K(n)^{1/2}- R\right)=0, 
\]
which admits the solutions $K(n)=0$ and $K(n)=\left(\frac{1}{2}\sqrt{\frac{R^2n^{-1/2}}{256\Gamma C}+4 R}-\frac{Rn^{-1/4}}{2\sqrt{256\Gamma C}}\right)^2$, as follows by solving a quadratic equation in $Z(n)=K(n)^{1/2}$ and preserving the positive root. Note that $K(n)=0$ is not admissible since $\sigma$ must be positive and that if $K(n)n^{-1/2}>512\Gamma C$, the right-hand side of \eqref{eq:defeq} is negative. Finally, note that $\sigma^{\star}(n)\in(0,1/2]$ for each $n$ sufficiently large.

Furthermore, if $ K(n)n^{-1/2}> 512\Gamma C$, the equation reads 
\[
K(n)^2n^{-1/2}= \frac{3840C\sqrt{\Gamma}}{a}\left(
4K(n)n^{-1/2}
    \sqrt{d}\left(2-\sqrt 2 \right)-\frac{K(n)^2n^{-1}\sqrt{2d}}{128\Gamma C}\right).
\]
Note that the right-hand side is negative in this case since $\frac{K(n)n^{-1/2}}{128\Gamma C}>4$ so that no solution exists in this case.
\end{proof}

With this result, establishing the desired parametric convergence rate is straightforward.

\begin{theorem}
   In the setting of \cref{prop:OptimalSolution}, if $z_{\hat \mu_n,\varepsilon}$ is a measurable $\varepsilon$-minimizer of $\hat \ell_n:z\in B_r\mapsto \frac 1n\sum_{i=1}^n\ell(X_i,z)$ in the sense that $\hat \ell_n(z_{\hat\mu_n,\varepsilon})\leq \hat \ell_n(z_{\hat\mu_n})+\varepsilon$, then it holds that 
   \[
        \mathbb E[\|z_{\hat \mu_n,\varepsilon}-z_{\mu}\|]\leq 4.05\sqrt{2\pi}K(n)n^{-1/2}+\sqrt{2C\varepsilon}
   \]
   provided that  $K(n)n^{-1/2}\in(0,1/2]$. On the other hand, the bound $\mathbb E[\|z_{\hat \mu_n,\varepsilon}-z_{\mu}\|]\leq 2r$ holds for any choice of $n\in\mathbb N$. 
\end{theorem}

\begin{proof} 
   Fix $\lambda >0$ and consider the set, $\Omega_{\lambda}$, of all $\omega\in\Omega$ (the sample space) such that there exists some $z\in B_r$ for which
   \[
       \frac 1n \sum_{i=1}^n \ell(X_i(\omega),z)\leq \frac 1n \sum_{i=1}^n \ell(X_i(\omega),z_{\mu})+\varepsilon\text{ and }\|z-z_{\mu}\|^2>\max\{2C\varepsilon,\lambda(\sigma^*)^2\}.  
   \]
Recall from (A'4) that $\|z'-z_{\mu}\|^2\leq C\mathbb E_{\mu}[\ell(X,z')-\ell(X,z_{\mu})]$ for every $z'\in B_r$ hence, for any $z$ as in the definition of $\Omega_{\lambda}$, 
\[
\begin{aligned}
    &\left|\frac 1n \sum_{i=1}^n \ell(X_i(\omega),z)-\mathbb E_{\mu}[\ell(X,z)]-\frac 1n \sum_{i=1}^n \ell(X_i(\omega),z_{\mu})+\mathbb E_{\mu}[\ell(X,z_{\mu})]\right|
    \\
    &\geq \frac 1n \sum_{i=1}^n \ell(X_i(\omega),z_{\mu})-\frac 1n \sum_{i=1}^n \ell(X_i(\omega),z)+\mathbb E_{\mu}[\ell(X,z)] - \mathbb E_{\mu}[\ell(X,z_{\mu})]
    \\
    &\geq C^{-1}\|z-z_{\mu}\|^2-\varepsilon.
\end{aligned}
\] 
Now, suppose that the first term in the above display is bounded above by $\max\{\|z_{\mu}-z\|^2,\lambda(\sigma^*)^2\}/(2C)$ so that the chain $C^{-1}\|z-z_{\mu}\|^2\leq  \max\{\|z_{\mu}-z\|^2,\lambda(\sigma^*)^2\}/(2C)+\varepsilon$ holds. Suppose first that $\lambda(\sigma^*)^2\leq \|z_{\mu}-z\|^2$ in which case $\|z-z_{\mu}\|^2\leq \|z_{\mu}-z\|^2/2+C\varepsilon$, that is, $\|z_{\mu}-z\|^2\leq 2C\varepsilon$. In the other case $\lambda(\sigma^*)^2\geq \|z_{\mu}-z\|^2$. In summary, $\|z_{\mu}-z\|^2\leq \max\{2C\varepsilon,\lambda(\sigma^*)^2\}$, which is a contradiction of our choice $z \in \Omega_{\lambda}$. As a consequence, instead we have that 
\[
\begin{aligned}
&\left|\frac 1n \sum_{i=1}^n \ell(X_i(\omega),z)-\mathbb E_{\mu}[\ell(X,z)]-\frac 1n \sum_{i=1}^n \ell(X_i(\omega),z_{\mu})+\mathbb E_{\mu}[\ell(X,z_{\mu})]\right|
\\ &> \max\{\|z_{\mu}-z\|^2,\lambda(\sigma^*)^2\}/(2C).
\end{aligned}
\]
Now, suppose that $\lambda \geq 1$ so that there exists a nonnegative integer $L$ satisfying $2^{2(L+1)}\geq \lambda \geq 2^{2L}$. In this case, we may apply \cref{lem:probBound} under the additional proviso that $2^{2L}(\sigma^*)^2\leq \frac{BC\Gamma}{a}$ to obtain that 
\begin{equation}
    \mathbb P^*(\Omega_{\lambda})\leq 8.1\exp(-2^{2L+1}n(\sigma^*)^2/b)\leq 8.1\exp(-\lambda/2), \label{eqn:omega_lambda_prob_upperbound}
\end{equation}
where the final upper bound follows from the fact that $2^{2L+1}\geq \lambda/2$ and $n(\sigma^*)^2/b\geq 1$. Note that the above upper bound is vacuous if $\lambda\leq 2\ln(8.1)$, as it states that $\mathbb P^{\star}(\Omega_{\lambda})\leq 1$; this covers the case $\lambda<1$.

Now, note that $\|z-z_{\mu}\|\leq 2r$ for every $z\in B_r$. Hence, if $\lambda(\sigma^*)^2\geq 4r^2$, the set $\Omega_{\lambda}$ is necessarily empty since then $\|z-z_{\mu}\|^2\leq\max\{2C\varepsilon,\lambda(\sigma^*)^2\}$. Now, if $\lambda$ is such that $\lambda (\sigma^*)^2\geq 2^{2L}(\sigma^*)^2>\frac{BC\Gamma}{a}$, we may use  the fact that $C\geq \frac{Ba}{\Gamma}$ by definition to obtain that $\lambda(\sigma^*)^2\geq B^2$. Consequently, the upper bound in \cref{eqn:omega_lambda_prob_upperbound} remains true for every $\lambda>0$.

Now, if $z_{\hat \mu_n,\varepsilon}\in\mathbb R^d$ satisfies the property that $\frac 1n \sum_{i=1}^n\ell(X_i,z_{\hat \mu_n,\varepsilon})-\frac 1n \sum_{i=1}^n\ell(X_i,z_{\hat \mu_n})\leq \varepsilon$, it holds that
    \[
    \begin{aligned}
        \mathbb E\left[\|z_{\hat \mu_n,\varepsilon}-z_{\mu}\|\right] = \int_0^{\infty}\mathbb P\left(\|z_{\hat \mu_n,\varepsilon}-z_{\mu}\|>t\right)dt&= \int_0^{\infty}\mathbb P\left(\|z_{\hat \mu_n,\varepsilon}-z_{\mu}\|^2>t^2\right)dt   
        \\
        &\leq \int_{\sqrt{2C\varepsilon}}^{\infty}\mathbb P\left(\|z_{\hat \mu_n,\varepsilon}-z_{\mu}\|^2>t^2\right)dt +\sqrt{2C\varepsilon},
    \end{aligned} 
    \]
    where the final inequality follows from splitting the integration region and noting that the (outer) probability is bounded above by $1$.
    Taking the change of variables $t^2=\lambda(\sigma^*(n))^2$, we obtain that 
    \[
    \begin{aligned}
        \int_{\sqrt{2C\varepsilon}}^{\infty}\mathbb P\left(\|z_{\hat \mu_n,\varepsilon}-z_{\mu}\|^2>t^2\right)dt&=\frac{1}{2}\sigma^*(n) \int_{\frac{2C\varepsilon}{(\sigma^*(n))^2}}^{\infty}\lambda^{-1/2}\mathbb P\left(\|z_{\hat \mu_n,\varepsilon}-z_{\mu}\|^2>\lambda(\sigma^*(n))^2\right)d\lambda
        \\
        &\leq 4.05\sigma^*(n) \int_{\frac{2C\varepsilon}{(\sigma^*(n))^2}}^{\infty}\lambda^{-1/2}e^{-\lambda/2}d\lambda\\
        &\leq 4.05\sqrt{2\pi}\sigma^*(n),  
    \end{aligned} 
    \]
    where the first upper bound follows by applying the results of the first half of the proof and using monotonicity of the (outer) probability. The final upper bound follows by extending the integral over the complete interval $(0,\infty)$. The claimed formula then follows by inserting the expression for $\sigma^{*}(n)$ from \cref{prop:OptimalSolution}.
\end{proof}

\newpage
\bibliographystyle{abbrv}
\bibliography{bibliography}

\begin{thebibliography}{10}

\bibitem{amblard2004biomagnetic}
C.~Amblard, E.~Lapalme, and J.-M. Lina.
\newblock Biomagnetic source detection by maximum entropy and graphical models.
\newblock {\em IEEE Trans. Biomed. Eng.}, 51(3):427--442, 2004.

\bibitem{birge1993rates}
L.~Birg{\'e} and P.~Massart.
\newblock Rates of convergence for minimum contrast estimators.
\newblock {\em Probab. Theory Relat. Fields}, 97(1):113--150, 1993.

\bibitem{brown1986fundamentals}
L.~D. Brown.
\newblock {\em Fundamentals of Statistical Exponential Families: with
  Applications in Statistical Decision Theory}.
\newblock Institute of Mathematical Statistics, Hayward, California, 1986.

\bibitem{brunel2025asymptotics}
V.-E. Brunel.
\newblock Asymptotics of constrained ${M}$-estimation under convexity, 2025.

\bibitem{burke2021study}
J.~V. Burke, T.~Hoheisel, and Q.~V. Nguyen.
\newblock A study of convex convex-composite functions via infimal convolution
  with applications.
\newblock {\em Math. Oper. Res.}, 46(4):1324--1348, 2021.

\bibitem{cai2022diffuse}
Z.~Cai, A.~Machado, R.~A. Chowdhury, A.~Spilkin, T.~Vincent, {\"U}.~Aydin,
  G.~Pellegrino, J.-M. Lina, and C.~Grova.
\newblock Diffuse optical reconstructions of functional near infrared
  spectroscopy data using maximum entropy on the mean.
\newblock {\em Sci. Rep.}, 12(1), 2022.
\newblock 2316.

\bibitem{chowdhury2013meg}
R.~A. Chowdhury, J.~M. Lina, E.~Kobayashi, and C.~Grova.
\newblock {MEG} source localization of spatially extended generators of
  epileptic activity: comparing entropic and hierarchical bayesian approaches.
\newblock {\em PloS one}, 8(2), 2013.
\newblock E55969.

\bibitem{DBLP:journals/corr/CohenATS17}
G.~Cohen, S.~Afshar, J.~Tapson, and A.~van Schaik.
\newblock Emnist: Extending mnist to handwritten letters.
\newblock In {\em 2017 International Joint Conference on Neural Networks
  (IJCNN)}, pages 2921--2926, 2017.

\bibitem{dacunha1990maximum}
D.~Dacunha-Castelle and F.~Gamboa.
\newblock Maximum d'entropie et probl{\`e}me des moments.
\newblock {\em Ann. Inst. Henri Poincaré, Probab. Stat.}, 26(4):567--596,
  1990.

\bibitem{dudley2018real}
R.~M. Dudley.
\newblock {\em Real analysis and probability}.
\newblock CRC Press, 2018.

\bibitem{fermin2006bayesian}
A.~K. Fermin, J.-M. Loubes, and C.~Ludena.
\newblock Bayesian methods for a particular inverse problem seismic tomography.
\newblock {\em Int. J. Tomogr. Stat}, 4(W06):1--19, 2006.

\bibitem{gamboa1989methode}
F.~Gamboa.
\newblock {\em M{\'e}thode du Maximum d'Entropie sur la Moyenne et
  Applications}.
\newblock PhD thesis, Universit{\'e} Paris-Sud, 1989.

\bibitem{grova2006evaluation}
C.~Grova, J.~Daunizeau, J.-M. Lina, C.~G. B{\'e}nar, H.~Benali, and J.~Gotman.
\newblock Evaluation of {EEG} localization methods using realistic simulations
  of interictal spikes.
\newblock {\em Neuroimage}, 29(3):734--753, 2006.

\bibitem{heers2016localization}
M.~Heers, R.~A. Chowdhury, T.~Hedrich, F.~Dubeau, J.~A. Hall, J.-M. Lina,
  C.~Grova, and E.~Kobayashi.
\newblock Localization accuracy of distributed inverse solutions for electric
  and magnetic source imaging of interictal epileptic discharges in patients
  with focal epilepsy.
\newblock {\em Brain Topogr.}, 29(1):162--181, 2016.

\bibitem{jaynes1957information1}
E.~T. Jaynes.
\newblock Information theory and statistical mechanics.
\newblock {\em Phys. Rev.}, 106:620--630, 1957.

\bibitem{jaynes1957information2}
E.~T. Jaynes.
\newblock Information theory and statistical mechanics. {II}.
\newblock {\em Phys. Rev.}, 108:171--190, 1957.

\bibitem{mkr2024maximumentropy}
M.~King-Roskamp, R.~Choksi, and T.~Hoheisel.
\newblock Data-driven priors in the maximum entropy on the mean method for
  linear inverse problems.
\newblock {\em SIAM J. Math. Data Sci.}, 8(1):108--140, 2026.

\bibitem{kosorok2008introduction}
M.~R. Kosorok.
\newblock {\em Introduction to empirical processes and semiparametric
  inference}.
\newblock Springer, 2008.

\bibitem{kullback1951information}
S.~Kullback and R.~Leibler.
\newblock On information and sufficiency.
\newblock {\em Ann. Stat.}, 22(1):79--86, 1951.

\bibitem{le1999new}
G.~Le~Besnerais, J.-F. Bercher, and G.~Demoment.
\newblock A new look at entropy for solving linear inverse problems.
\newblock {\em IEEE Trans. Inf. Theory}, 45(5):1565--1578, 1999.

\bibitem{mai2021stability}
V.~Mai and M.~Johansson.
\newblock Stability and convergence of stochastic gradient clipping: Beyond
  {Lipschitz} continuity and smoothness.
\newblock In {\em International Conference on Machine Learning}, pages
  7325--7335. PMLR, 2021.

\bibitem{marechal1997unification}
P.~Mar{\'e}chal and A.~Lannes.
\newblock Unification of some deterministic and probabilistic methods for the
  solution of linear inverse problems via the principle of maximum entropy on
  the mean.
\newblock {\em Inverse Probl.}, 13(1), 1997.

\bibitem{massart2007concentration}
P.~Massart.
\newblock Gaussian model selection.
\newblock In {\em Concentration Inequalities and Model Selection: Ecole
  d'Et{\'e} de Probabilit{\'e}s de Saint-Flour XXXIII - 2003}, pages 83--146.
  Springer Berlin Heidelberg, Berlin, Heidelberg, 2007.

\bibitem{Navaza:a24265}
J.~Navaza.
\newblock On the maximum-entropy estimate of the electron density function.
\newblock {\em Acta Crystallogr. Sect. A}, 41(3):232--244, 1985.

\bibitem{navaza1986use}
J.~Navaza.
\newblock The use of non-local constraints in maximum-entropy electron density
  reconstruction.
\newblock {\em Acta Crystallogr. Sect. A}, 42(4):212--223, 1986.

\bibitem{netzer2011reading}
Y.~Netzer, T.~Wang, A.~Coates, A.~Bissacco, B.~Wu, A.~Y. Ng, et~al.
\newblock Reading digits in natural images with unsupervised feature learning.
\newblock In {\em NIPS workshop on deep learning and unsupervised feature
  learning}, volume~5, page~7, 2011.

\bibitem{rietsch1976maximum}
E.~Rietsch.
\newblock The maximum entropy approach to inverse problems-spectral analysis of
  short data records and density structure of the earth.
\newblock {\em J. Geophys.}, 42(1):489--506, 1976.

\bibitem{rioux2020maximum}
G.~Rioux, R.~Choksi, T.~Hoheisel, P.~Mar{\'e}chal, and C.~Scarvelis.
\newblock The maximum entropy on the mean method for image deblurring.
\newblock {\em Inverse Probl.}, 37(1):015011, 2020.

\bibitem{rioux2021blind}
G.~Rioux, C.~Scarvelis, R.~Choksi, T.~Hoheisel, and P.~Maréchal.
\newblock Blind deblurring of barcodes via {K}ullback-{L}eibler divergence.
\newblock {\em IEEE Trans. Pattern Anal. Mach. Intell.}, 43(1):77--88, 2021.

\bibitem{ROBBINS1971233}
H.~Robbins and D.~Siegmund.
\newblock A convergence theorem for non negative almost supermartingales and
  some applications.
\newblock In J.~S. Rustagi, editor, {\em Optimizing Methods in Statistics},
  pages 233--257. Academic Press, 1971.

\bibitem{rockafellar1970convex}
R.~T. Rockafellar.
\newblock {\em Convex Analysis}, volume~18.
\newblock Princeton University Press, 1970.

\bibitem{rockafellar1998variational}
R.~T. Rockafellar and R.~J.-B. Wets.
\newblock {\em Variational analysis}.
\newblock Springer, 1998.

\bibitem{urban1996retrieval}
B.~Urban.
\newblock Retrieval of atmospheric thermodynamical parameters using satellite
  measurements with a maximum entropy method.
\newblock {\em Inverse Probl.}, 12(5), 1996.
\newblock 779.

\bibitem{vaisbourd2022maximum}
Y.~Vaisbourd, R.~Choksi, A.~Goodwin, T.~Hoheisel, and C.-B. Sch{\"o}nlieb.
\newblock Maximum entropy on the mean and the {C}ram{\'e}r rate function in
  statistical estimation and inverse problems: properties, models, and
  algorithms.
\newblock {\em Math. Program.}, 214(1):441--490, 2025.

\bibitem{van2000asymptotic}
A.~W. van~der Vaart.
\newblock {\em Asymptotic statistics}, volume~3.
\newblock Cambridge University Press, 2000.

\bibitem{vaart1996empirical}
A.~W. van~der Vaart and J.~A. Wellner.
\newblock {\em Weak Convergence and Empirical Processes: With Applications to
  Statistics}.
\newblock Springer, 1996.

\end{thebibliography}

\end{document}